\documentclass[english]{article}
\usepackage[utf8]{inputenc}
\usepackage[T1]{fontenc}
\usepackage{babel}
\usepackage{graphicx}
\usepackage{fancyhdr}
\usepackage{amsmath,amsthm,amsfonts,graphicx}
\usepackage{amssymb}
\usepackage{mathrsfs}
\usepackage{xcolor}
\usepackage{indentfirst}
\usepackage{extarrows}
\usepackage{float}
\usepackage{cite}
\usepackage{comment}
\usepackage{mathtools}
\usepackage{appendix}
\usepackage{epstopdf}
\usepackage{hyperref}
\usepackage{ulem}

\newtheorem{theorem}{Theorem}[section]
\newtheorem{proposition}[theorem]{Proposition}
\newtheorem{lemma}[theorem]{Lemma}

\newtheorem{remark}[theorem]{Remark}
\newtheorem{corollary}[theorem]{Corollary}
\newtheorem{definition}{Definition}

\newtheorem*{theorem*}{Theorem}

\numberwithin{equation}{section}
\title{\bf Asymptotic spreading of interacting species with multiple fronts II: Exponentially decaying initial data}

\author{Qian Liu$^{1,2}$, Shuang Liu$^{1,2}$ and King-Yeung Lam$^{2}$}
\date{\small
    $^1$Institute for Mathematical Sciences, Renmin University of China, Beijing, China\\%
    $^2$Department of Mathematics, The Ohio State University, Columbus, OH 43210, USA\\[2ex]%
    \today
}
\begin{document}

\maketitle

\begin{abstract}
This is part two of our study on the spreading properties of the Lotka-Volterra competition-diffusion systems with a stable coexistence state. We focus on the case when the initial data are exponential decaying. By establishing a comparison principle for Hamilton-Jacobi equations, we are able to apply the Hamilton-Jacobi approach for Fisher-KPP equation due to Freidlin, Evans and Souganidis. As a result, the exact formulas of spreading speeds and their dependence on initial data are derived. Our results indicate that sometimes the spreading speed of the slower species is nonlocally determined. Connections of our results with the traveling profile due to Tang and Fife, as well as the more recent spreading result of Girardin and Lam, will be discussed.
\end{abstract}

{\small
{\makeatletter
\def\@@underline#1{#1}
\def\@@overline#1{#1}
\tableofcontents
\makeatother}
}
\section{Introduction}

For monotone dynamical systems, the pioneering work of Weinberger et al. \cite{Weinberger_1982,Weinberger_2002b} (see also \cite{Lui_1989}) relates the spreading speed of the population to the minimal speed of (monostable) traveling wave solutions. Their result can be applied to the diffusive Lotka-Volterra competition system. Suitably non-dimensionalized, the system is given by
\begin{equation}\label{eq:1-1}
\left\{
\begin{array}{ll}
\partial_t u-\partial_{xx}u=u(1-u-av),& \text{ in }(0,\infty)\times \mathbb{R},\\
\partial _t v-d\partial_{xx}v=r v(1-bu-v),& \text{ in }(0,\infty)\times \mathbb{R},\\
u(0,x)=u_0(x), & \text { on }  \mathbb{R},\\
v(0,x)=v_0(x), & \text { on } \mathbb{R},
\end{array}
\right .
\end{equation}
with $a,b\in(0,1)$.  
 It is clear that \eqref{eq:1-1} admits  a trivial equilibrium $(0, 0)$, two semi-trivial equilibria $(1,0)$ and $(0,1)$, 
and further a linearly stable equilibrium $$(k_1,k_2)=\left(\frac{1-a}{1-ab},\frac{1-b}{1-ab}\right).$$ 
\begin{theorem}[Lewis et al.{\cite{Lewis_2002}}]\label{thm:LLW}
Let $(u,v)$
be the solution of \eqref{eq:1-1} with initial data
$$u(0,x)=\rho_1(x), \,\,\,\, v(0,x)=1-\rho_2(x), 
$$
where $0\leq \rho_i<1$ $(i=1,2)$ are compactly supported functions in $\mathbb{R}$. Then there exists some $c_{\rm LLW} \in [2\sqrt{1-a}, 2]$ such that
\begin{align*}
 \left\{
\begin{array}{ll}
\smallskip
\lim\limits_{t\rightarrow \infty}  \sup\limits_{|x|<c t} (|u(t,x)-k_1|+|v(t,x)-k_2|)=0 &\text{ for each } c<c_{\rm LLW},\\
\lim\limits_{t\rightarrow \infty}  \sup\limits_{|x|>c t} (|u(t,x)|+|v(t,x)-1|)=0 &\text{ for each }~ c>c_{\rm LLW}.
\end{array}
\right.
 \end{align*}
In this case, we say that $u$ spreads at speed $c_{\rm LLW}$.
\end{theorem}
\begin{remark}\label{rmk:LLW}
If the initial data $(u,v)(0,x)$ is a compact perturbation of $(1,0)$, then there exists $\tilde{c}_{\rm LLW} \in [2\sqrt{dr(1-b)}, 2\sqrt{dr}]$ such that the species $v$ spreads at speed $\tilde{c}_{\rm LLW}$.
\end{remark}
It is shown in \cite{Li_2005, Liang_2007} that the spreading speed $c_{\rm LLW}$ (resp. $\tilde{c}_{\rm LLW}$) is identical to the minimum wave speed of traveling wave solution connecting the pair of equilibria $(k_1,k_2)$ and $(0,1)$ (resp. $(1,0)$). It is crucial for the theory that the pair of equilibria forms an ordered pair of equilibria (regarding the comparability of steady states in the theory of monotone semi-flows, see \cite{Smith_1995}).

For the weak competitive diffusive system \eqref{eq:1-1}, Tang and Fife \cite{Tang_1980} proved an additional class of traveling wave solutions connecting the positive equilibrium $(k_1,k_2)$ with the trivial equilibrium $(0,0)$. In this case, the equilibria $(0,0)$ and $(k_1,k_2)$ are un-ordered, and hence the existence of traveling wave, due to Tang and Fife \cite{Tang_1980}, does not directly follow from the monotone dynamical systems framework due to Weinberger et al. \cite{Weinberger_2002,Weinberger_2002b} (see also \cite{Fang_2014,Liang_2007}).

A natural question is whether the speed traveling wave solutions due to Tang and Fife, which connect $(k_1,k_2)$ to $(0,0)$, determine the spreading speed of the populations in the Cauchy problem  \eqref{eq:1-1}, provided the initial data $(u_0,v_0)$ has the same asymptotics at $x=\pm \infty$ as the traveling wave solution? What happens for more general exponentially decaying initial data? Does the two species spread with different speeds?

In this paper, we continue our investigation in \cite{LLL2019} on the spreading properties of solutions of the Cauchy problem \eqref{eq:1-1}.
We are interested in determining the spreading speeds of each of the populations $u$ and $v$,  for the class of initial data $(u_0,v_0)$ satisfying $(u_0,v_0)(-\infty) = (1,0)$, $(u_0,v_0)(\infty) = (0,0)$ and such that $u_0 \to 0$ exponentially at $\infty$ with rate $\lambda_u>0$; $v_0 \to 0$ decays exponentially at $ \infty$ (resp. $-\infty$) with rate $\lambda^{+}_v>0$ (resp. $\lambda^-_v>0$).

We introduce the Hamilton-Jacobi approach to study the spreading of two-interacting species into an open habitat, and resolve a conjecture by Shigesada \cite[Ch. 7]{Shigesada_1997}. Inspired by the pioneering work of Freidlin \cite{Freidlin_1985} and of Evans and Souganidis \cite{Evans_1989} on the Fisher-KPP equation, we shall derive, via the thin-front limit, a couple of Hamilton-Jacobi equations for which solutions have to be understood in the viscosity sense. In our previous work \cite{LLL2019}, we considered the Cauchy problem \eqref{eq:1-1} endowed with compactly supported initial data, and used the dynamics programming approach to show the uniqueness of the limiting Hamilton-Jacobi equations, and to evaluate the solution by determining the path that minimizes certain action functional.  In contrast to our previous paper, we will tackle the Cauchy problem with exponentially decaying initial data using entirely PDE arguments. For this purpose, we establish a general comparison principle for discontinuous viscosity solutions associated with piecewise Lipschitz Hamiltonians, the latter arising naturally in the spreading of multiple species. The proof of the comparison result is based on combining the ideas due to Ishii \cite{Ishii_1997} and Tourin \cite{Tourin_1992}. With this comparison principle at our disposal, we are able to obtain large-deviation type estimates of the solutions $(u,v)$ to the Cauchy problem  \eqref{eq:1-1} by
 explicit construction of simple piecewise linear super- and sub-solutions.

\subsection{Known results of a single population}
We first recall some classical asymptotic spreading results concerning the single Fisher-KPP equation:
%
\begin{equation}\label{eq:single}
\left\{
\begin{array}{ll}
\medskip
\partial _t  \phi-\tilde{d}\partial_{xx}\phi=\tilde{r}\phi(1 - \phi),&  \text{ in }(0,\infty)\times \mathbb{R},\\
\phi(0,x)=\phi_0(x), & \text{ on } \mathbb{R},\\
\end{array}
\right .
\end{equation}
where $\tilde{d},\tilde{r}$ are positive constants. If the initial data is a Heaviside function,  supported on $(-\infty,0]$, 
   it is shown \cite{Aronson_1975,Fisher_1937,Kolmogorov_1937} that the population, whose density is given by $\phi(t,x)$ has the spreading speed $c^* = 2\sqrt{\tilde{d} \tilde{r}}$, i.e., 
\begin{equation*}
\left\{
\begin{array}{ll}
\medskip
\lim\limits_{t\rightarrow \infty}  \sup\limits_{x<c t} |\phi(t,x)-1|=0 &\text{ for all } c<c^*,\\
\lim\limits_{t\rightarrow \infty}  \sup\limits_{x>c t} |\phi(t,x)|=0 &\text{ for all } c>c^*.\\
\end{array}
\right .
\end{equation*}
In addition, the spreading speed $c^*$ coincides with the minimal speed of the traveling wave solutions to \eqref{eq:single} in this case. If we broaden the scope of initial data $\phi_0$ to include the class of exponentially decaying data, then
the asymptotic behavior of the solution to \eqref{eq:single} is sensitive to the rate of decay of $\phi_0$ at $x=\pm\infty$ (see e.g. \cite[pp.42]{Frank_1975}), which is the {leading edge} of the front. This is related to the fact that
$0$ is a saddle for \eqref{eq:single}, see \cite{Kametaka_1976,Mckean_1975,Booty_1993,Ebert_2000,Saarloos_2003}.


Precisely, denoting $\lambda^*=\sqrt{{\tilde{r}}/{\tilde{d}}}$. It is proved \cite{Kametaka_1976, Mckean_1975} that:
\begin{itemize}
             \item [\rm{(i)}] When the initial data $\phi_0(x)$ decays faster than $\exp\{-\lambda^*x\}$  at $x=\infty$, then the spreading speed $c^*=2\sqrt{\tilde{d}\tilde{r}}$;  
             \item [\rm{(ii)}] When the initial data $\phi_0(x)$ is the form of $\exp\{-(\lambda+o(1)) x\}$  at $x=\infty$ with $\lambda<\lambda^*$, then the population has the spreading speed $c(\lambda) = \tilde{d}\lambda + \frac{\tilde{r}}{\lambda}$ which is strictly greater than $2\sqrt{\tilde{d}\tilde{r}}$. 
 \end{itemize}

For recent developments in asymptotic spreading of a single population in heterogeneous environments, we refer to \cite{Berestycki_2012,Berestycki_2018, Fang_2016} for the  one-dimensional case, and to \cite{Berestycki_2008,BNpreprint,Nolen_2008,Weinberger_2002} for higher-dimensional case.

\subsection{Known results of multiple populations}

For close to three decades, researchers have been trying to extend these results to reaction-diffusion systems describing two or more interacting populations.
%

Motivated by the northward spreading of several tree species into the newly de-glaciated North American continent at the end of the last ice age, Shigesada et al. \cite[Ch. 7]{Shigesada_1997} formulated the question of spreading of two or more competing species into an open habitat, i.e., one that is unoccupied by either species. In case of two competing species, it is conjectured that for large time, the solution behaves like stacked traveling fronts, i.e., it exhibits two transition layers moving at two different speeds $c_1 > c_2$, connecting three homogeneous equilibrium states $(0,0)$, $E_1$ and $E_2$. Here $E_1$ is the semi-trivial equilibrium where the faster species is present, and $E_2$ is either the other semi-trivial equilibrium or the coexistence equilibrium (if the latter exists). While it is not difficult to see that the spreading speed $c_1$ of the faster species can be predicted by the underlying single equation (since the slower species is essentially absent at the leading edge of the front), the determination of the second speed remained open over a decade.
Lin and Li \cite{Lin_2012} first worked on the spreading properties of \eqref{eq:1-1} in the weak competition case $0<a,b<1$ with compactly supported initial condition $(u_0,v_0)$ and obtained estimates for the spreading speed $c_2$ of the slower species. For the strong competition case $a,b>1$, Carr\`{e}re  \cite{Carrere_2018} determined both of the spreading speeds, where $c_2$ is determined by the unique speed of traveling wave solutions connecting the semi-trivial steady state $(1,0)$ and $(0,1)$. The predator-prey system was considered by Ducrot et al. \cite{Ducrot_preprint}. For cooperative systems with equal diffusion coefficients, the existence of stacked fronts for cooperative systems was also studied by \cite{Iida_2011}. In these cases, the spreading speeds of each individual species can be determined locally and is not influenced by the presence of other invasion fronts.

However, the second speed $c_2$ can in general be influenced by the first front with speed $c_1$, as demonstrated by the work of Holzer and Scheel \cite{Holzer_2014} which applies in particular to \eqref{eq:1-1} for the case $a=0$ and $b>0$. They showed that the second speed $c_2$ can be determined by the linear instability of the zero solution of a single equation with space-time inhomogeneous coefficient. For coupled systems, the case $0 < a < 1 < b$ was treated in a recently appeared paper of Girardin and the third author \cite{Girardin_2018}. By deriving an explicit formula for $c_2$, it is observed that $c_2$ can sometimes be strictly greater than the minimal speed of traveling wave connecting $E_1$ and $E_2$, and that it depends on the first speed $c_1$ in a non-increasing manner. The proof in \cite{Girardin_2018} is based on a delicate construction of (piecewise smooth) super- and sub-solutions for the parabolic system. In our previous paper \cite{LLL2019}, we showed that in the weak competition case $0< a,b<1$ the formula for $c_2$ is exactly the same as the one in \cite{Girardin_2018} but with a novel strategy of proof based on obtaining large deviation estimates via analyzing the Hamilton-Jacobi equations obtained in the thin-front limit. We also mention that coupled parabolic systems were also treated in \cite{Evans_1989b,Freidlin_1991} based on the large deviations approach, but in these papers all components spread with a single spreading speed.

\subsection{Main results}
In this paper, we study the spreading of two competing species into an open habitat with exponentially decaying (in space) initial data, with attention to how the spreading speeds are influenced by the exponential rates of decay at infinity.

For a function $g: \mathbb{R} \to \mathbb{R}$ and $\lambda \in \mathbb{R}$, we say that $g(x) \sim e^{-\lambda x}$ at $\infty$ if
$$
0 < \liminf_{x \to \infty} e^{\lambda x}g(x) \leq \limsup_{x \to \infty} e^{\lambda x}g(x) <\infty.
$$
Definition for $g(x) \sim e^{\lambda x}$ at $-\infty$ is similar. We now state our hypothesis for the initial data $(u_0,v_0)$.

$$
\rm{(H_\lambda)}\begin{cases} \text{The initial value } (u_0, v_0)\in C(\mathbb{R};[0,1])^2 \text{ is strictly positive on }\mathbb{R},\\
\text{ and there exist positive constants } \theta_0,\lambda_u, \lambda^+_v, \lambda^-_v \text{ such that } \\
u_0(x) \geq \theta_0\quad \text{ in }(-\infty,0], \quad u_0(x) \sim e^{-\lambda_u x} \,\, \text{ at }\,\,\infty,\\
v_0(x) \sim e^{\lambda_v^- x} \,\,\text{ at }\,\, -\infty,\quad \text{ and }\quad v_0(x) \sim e^{-\lambda_v^+ x} \,\, \text{ at }\,\, \infty.
\end{cases}.
$$

We denote
 \begin{align}\label{eq:sigma}
 \begin{cases}\displaystyle
 \medskip
\sigma_1=d(\lambda_v^+\wedge \sqrt{\frac{r}{d}})+\frac{r}{\lambda_v^+\wedge\sqrt{\frac{r}{d}}},\quad \sigma_2=(\lambda_u\wedge 1)+\frac{1}{\lambda_u\wedge 1},\\
\displaystyle \sigma_3=d(\lambda_v^-\wedge \sqrt{\frac{r(1-b)}{d}})+\frac{r(1-b)}{\lambda_v^-\wedge\sqrt{\frac{r(1-b)}{d}}},
\end{cases}
\end{align}
where $a\wedge b = \min\{a,b\}$ for $a,b \in \mathbb{R}$. Here the quantity $\sigma_1$ (resp. $\sigma_2$) denotes the spreading speed of $v$ (resp. $u$) in the absence of the competitor
\cite{Kametaka_1976,Mckean_1975}. Without loss of generality, we assume  $\sigma_1\geq \sigma_2 $ throughout this paper. This amounts to fixing the choice of $v$ to be the faster spreading species.

Our main result is stated as follows.

\begin{theorem}\label{thm:1-2}
Assume $\sigma_1 > \sigma_2$. Let $(u,v)$ be the solution of \eqref{eq:1-1} such that the initial data satisfies $\mathrm{(H_\lambda)}$. Then
 there exist $c_1, c_2, c_3\in \mathbb{R}$ such that $c_3 < 0 < c_2 < c_1$, and
%
for each small $\eta>0$, the following spreading results hold:
\begin{equation}
\begin{cases}
\lim\limits_{t\rightarrow \infty} \sup\limits_{ x>(c_{1}+\eta) t} (|u(t,x)|+|v(t,x)|)=0, \\
\lim\limits_{t\rightarrow \infty} \sup\limits_{(c_2+\eta) t< x<(c_{1}-\eta) t} (|u(t,x)|+|v(t,x)-1|)=0, \\
\lim\limits_{t\rightarrow \infty} \sup\limits_{(c_{3}+\eta)t< x<(c_2-\eta) t} (|u(t,x)-k_1|+|v(t,x)-k_2|)=0 , \\
 \lim\limits_{t\rightarrow \infty} \sup\limits_{x<(c_{3}-\eta)t} (|u(t,x)-1|+|v(t,x)|)=0. 
 \end{cases} \label{eq:spreadingly}
\end{equation}
Precisely, the spreading speeds $c_3 < 0 < c_2 < c_1$ can be determined as follows:
\begin{equation}\label{eq:c123}
c_1 = \sigma_1, \quad c_2 = \max\{c_{\rm LLW}, \hat{c}_{\rm nlp}\},\quad c_3 = - \max\{ \tilde{c}_{\rm LLW}, \sigma_3\},
\end{equation}
where $c_{\rm LLW}$ $($resp. $\tilde{c}_{\rm LLW})$ is 
given in Theorem \ref{thm:LLW} $($resp. Remark \ref{rmk:LLW}$)$,
and
\begin{equation}\label{eq:hcacc}
\hat{c}_{\rm nlp} = \begin{cases}
\frac{\sigma_1}{2} - \sqrt{a} + \frac{1-a}{\frac{\sigma_1}{2} - \sqrt{a}}, & \text{ if }\sigma_1 < 2\lambda_u\,\, \text{ and }\,\, \sigma_1 \leq 2 (\sqrt a + \sqrt{1-a}),\\
\tilde\lambda_{\rm nlp} + \frac{1-a}{\tilde\lambda_{\rm nlp}}, & \text{ if }\sigma_1  \geq 2\lambda_u \,\,\text{ and }\,\, \tilde\lambda_{\rm nlp} \leq \sqrt{1-a},\\
2\sqrt{1-a}, &\text{ otherwise,}
\end{cases}
\end{equation}
with the quantity $\tilde\lambda_{\rm nlp}$ being given by
\begin{equation}\label{eq:lambdaacc}
\tilde\lambda_{\rm nlp} = \frac{1}{2}\left[\sigma_1 - \sqrt{(\sigma_1-2\lambda_u)^2 + 4a}\right].
\end{equation}
\end{theorem}
To visualize the spreading result \eqref{eq:spreadingly} visually, we consider the scaling
$$(\hat u,\hat v)(t,x)=\lim_{\epsilon\to0}(u,v)\left(\frac{t}{\epsilon},\frac{x}{\epsilon}\right) \quad \text{ for } (t,x)\in (0,\infty)\times \mathbb{R},
$$
whose asymptotic behaviors can be given in Figure \ref{figure5}.
\begin{figure}[http!!]
  \centering
\includegraphics[height=1.69in]{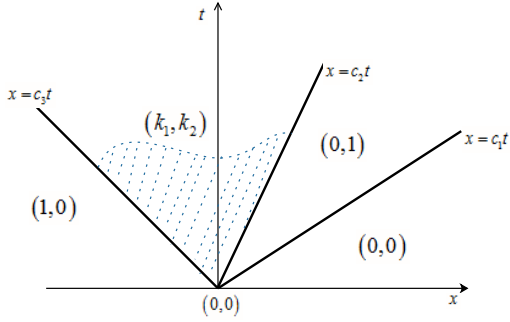}
  \caption{The asymptotic behaviors of  $(\hat u,\hat v)$.}\label{figure5}
  \end{figure}

Note that while the spreading speed $c_1$ of the faster species $v$ is entirely determined by $\lambda_v^+$ (the exponential  decay  of $v_0$ at $x \approx \infty$), and is unaffected by the slower species $u$, the corresponding speed $c_2$ of species $u$ depends upon  $\sigma_1$ and $\lambda_u$ (the exponential  decay  of $u_0$ at $x \approx \infty$). In particular, when $\lambda_v^+\geq\sqrt{\frac{r}{d}}$ and $\lambda_u> \frac{\sigma_1}{2}$, i.e., $v_0(x)$ and $u_0(x)$ decay fast enough, the speeds $c_1$ and $c_2$ are the same as that of the case of compactly supported initial data (see \cite[Theorem 1.2]{LLL2019}).

\begin{remark}
We point out that the speed $c_2$ in Theorem \ref{thm:1-2} is non-increasing in both $\sigma_1$ and $\lambda_u$, which follows from the following observations: {\rm(i)} $\tilde\lambda_{\rm nlp}$ given by \eqref{eq:lambdaacc} is non-decreasing in both $\sigma_1$ and $\lambda_u$; {\rm(ii)}  $s + \frac{1-a}{s}$ is non-increasing in $(0,\sqrt{1-a}]$.
This fact makes intuitive sense: {\rm{(i)}} a higher $\sigma_1$ means the region dominated by species $v$, which is roughly $\{(t,x): c_2t < x < \sigma_1 t\}$, is larger and thus rendering it more difficult for species $u$ to invade;  {\rm{(ii)}} a higher $\lambda_u$ means there are less population at the front to pull the invasion wave, which also makes it difficult for species $u$ to invade.
\end{remark}

Fix $\sigma_1,\, \lambda_u >0$ and $0 <a <1$, such that $\sigma_1 >\sigma_2$ holds. We shall see that the quantity $\hat{c}_{\rm nlp}$ in \eqref{eq:hcacc} can be equivalently defined by
$$
\{(t,x): \overline w_2 (t,x) =0\} = \{ (t,x): t>0 \, \text{ and }\, x \leq  \hat{c}_{\rm nlp} t\},
$$
where $\overline w_2(t,x)$ is the unique viscosity solution of the Hamilton-Jacobi equation
\begin{equation}\label{eq:hj_j11}
\left\{\begin{array}{ll}
\medskip
\min\{\partial_t w + |\partial_x w|^2 + 1 - a \chi_{\{x < \sigma_1 t\}},w\} = 0, & \text{ in } (0,\infty)\times \mathbb{R},\\
w(0,x) = \lambda_u \max\{x, 0\}, &\text{ on }\mathbb{R}.
\end{array}
\right.
\end{equation}
Here $\chi_{S}$ is the indicator function of the set $S\in (0,\infty)\times \mathbb{R}$.

A further point of interest is the involvement of $(0,0)$ and $(k_1,k_2)$ in co-invasion process of \eqref{eq:1-1}, which happens only in the weak competition case $0<a,b<1$. In this case, the equilibrium states $(0,0)$ and $(k_1,k_2)$ are un-ordered, and hence the existence of traveling wave, due to Tang and Fife \cite{Tang_1980}, cannot be established by monotone dynamical systems framework due to Weinberger et al. \cite{Weinberger_2002b} (see also \cite{Fang_2014,Liang_2007}). We will see that the invasion front $(k_1,k_2)$ into $(0,0)$ is indeed realized in \eqref{eq:1-1} for initial data with certain values of exponential decay rates $\lambda_u, \lambda^{+}_v$ at infinity,
 namely, when $\sigma_1=\sigma_2$.


\begin{theorem}\label{thm:1-2b} Assume $\sigma_1 = \sigma_2$.
Let $(u,v)$ be the solution of \eqref{eq:1-1} such that the initial data satisfies
 $\mathrm{(H_\lambda)}$.  Then
%
for each small $\eta>0$, it holds that
\begin{equation}
\begin{cases}
\lim\limits_{t\rightarrow \infty} \sup\limits_{ x>(\sigma_{1}+\eta) t} (|u(t,x)|+|v(t,x)|)=0, \\
\lim\limits_{t\rightarrow \infty} \sup\limits_{(c_{3}+\eta)t< x<(\sigma_1-\eta) t} (|u(t,x)-k_1|+|v(t,x)-k_2|)=0 , \\
 \lim\limits_{t\rightarrow \infty} \sup\limits_{x<(c_{3}-\eta)t} (|u(t,x)-1|+|v(t,x)|)=0, 
 \end{cases} \label{eq:spreadingly'}
\end{equation}
where $c_3 = - \max\{ \tilde{c}_{\rm LLW}, \sigma_3\}$ and that
$\tilde{c}_{\rm LLW}$ is 
given in Remark \ref{rmk:LLW}.
\end{theorem}
 For initial data with general exponential decay rates, Theorem \ref{thm:1-2} demonstrates that there are two separate monostable fronts where each of the two species invades with distinct speeds. Moreover, if the parameters of \eqref{eq:1-1} changes in such a way that $|\sigma_1-\sigma_2| \to 0$, the distance of the two fronts tends to zero. Therefore, the invasion front of $(k_1,k_2)$ transitioning directly into $(0,0)$, due to Tang and Fife, is in fact the special case when these two monostable fronts coincide (Theorem \ref{thm:1-2b}).
%
%
%
%
\begin{remark}
As in \cite{Evans_1989,LLL2019}, our approach can be applied to the spreading problem of competing species in higher dimensions under minor modifications. However, we choose to focus here on the one-dimensional case to keep our exposition simple, and close to the original formulation of the conjecture in \cite[Chapter 7]{Shigesada_1997}.
\end{remark}

\subsection{Outline of main ideas}
 To determine $c_1,\,c_2,\,c_3$, we introduce large deviation approach and construct appropriate viscosity super- and sub-solutions for certain Hamilton-Jacobi equations, and then apply the comparison principle (Theorem \ref{eq:D1}) to obtain the desired estimations.
We outline the main steps leading to the determination of the nonlocally pulled spreading speed $c_2$, as stated Theorem \ref{thm:1-2}, and remark that $c_1,c_3$ can be obtained by a similar even simpler argument as $c_2$.
\begin{enumerate}
\item  To estimate $c_2$ from below, we consider the transformation $w_2^\epsilon(t,x) = -\epsilon \log u\left( \frac{t}{\epsilon}, \frac{x}{\epsilon}\right)$ and show that the half-relaxed limits
$$
w_{2,*}(t,x) =  \hspace{-0.3cm} \liminf_{\scriptsize \begin{array}{c}\epsilon \to 0\\ (t',x') \to (t,x)\end{array}} \hspace{-.3cm}w_2^\epsilon(t',x') \quad \text{ and } \quad w_2^*(t,x) =  \hspace{-0.3cm}  \limsup_{\scriptsize\begin{array}{c} \epsilon \to 0 \\ (t',x') \to (t,x) \end{array}}  \hspace{-0.3cm}  w_2^\epsilon(t',x')
$$
exist, upon establishing uniform bounds in $C_{\mathrm{loc}}$ (see Lemma \ref{lem:3-1'}). By constructing viscosity super-solution $\overline{w}_2$, which satisfies
$$
\{(t,x):\overline w_2(t,x) =0\} = \{ (t,x): t>0 \, \text{ and }\, x \leq  \hat{c}_{\rm nlp} t\},
$$
 and using the comparison principle (Theorem \ref{thm:D}), we can show that $w_2^*\leq \overline{w}_2$, and thus $w_2^\epsilon \to 0$ locally uniformly in $\{(t,x): \,  x <  \hat{c}_{\rm nlp} t\}.$ One can then apply the arguments in \cite[Section 4]{Evans_1989} to show that
$$
\liminf_{\epsilon \to 0} u\left( \frac{t}{\epsilon}, \frac{x}{\epsilon}\right) >0 \quad \text{ in } \left\{ (t,x): t>0 \, \text{ and }\, x <  \hat{c}_{\rm nlp} t\right\}.
$$
This implies that ${c}_2 \geq \hat{c}_{\rm nlp}$ (see Lemma \ref{lem:underlinec2}).

\item To estimate $c_{2}$ from above, we construct viscosity sub-solution $\underline{w}_2$ and apply Theorem \ref{thm:D} to estimate $w_{2}$ from below, see Proposition \ref{prop:overlinec2}. This enables us to obtain a large deviation estimate of $u$. Namely, for each small $\delta>0$, let $\hat{c}_\delta  = \sigma_1-\delta$, we have
$$
u(t, \hat{c}_\delta t) \leq \exp \left( -[ \hat{\mu}_\delta + o(1)]t \right) \quad \text{ for } t \gg 1,
$$
where $\hat{\mu}_\delta  =  \overline w_2(1, \hat{c}_\delta)= \overline w_2(1, \sigma_1 - \delta)$. Now, recalling that $(u,v)$ is a solution to \eqref{eq:1-1} restricted to the domain $\{(t,x): 0\leq x \leq \hat{c}_\delta t\}$, with boundary condition satisfying
$$
\lim_{t\to\infty} (u,v)(t,0) = (k_1,k_2) \quad \text{ and }\quad \lim_{t\to\infty} (u,v)(t, \hat{c}_\delta t) = (0,1),
$$
we may apply Lemma \ref{lem:appen1} in Appendix to show that $\hat{c}_\delta$ and $\hat{\mu}_\delta$ completely controls the spreading speed $c_2$ of $u$ from above.\end{enumerate}


The rest of the paper is organized as follows: In Section \ref{S2}, we give upper estimates $c_i$ for $i=1,\,2,\,3$ and $c_2\geq c_{\rm LLW}$. In Section \ref{S3}, we give lower estimates of $c_1,\,c_2$. The approximate asymptotic expressions of $u$ and $v$ are established in Section \ref{S4}, where we also determine $c_2,\, c_3$. In Section \ref{S4b},  we discuss the relation of our results with the invasion mode due to Tang and Fife \cite{Tang_1980}. In Section \ref{S5}, we discuss the relation of our result with that of \cite{Girardin_2018} due to Girardin and the last author. In Section \ref{S6}, we prove an extension which is associated to the spreading speeds of the three-species competition systems. We conclude the article with the Appendix. Therein we give the comparison principle of Hamilton-Jacobi equation with piecewise Lipschitz continuous Hamiltonian and two other useful lemmas.

This paper concerns the Cauchy problem of a system of reaction-diffusion equations modeling two competing species. For the spreading of two species into an open habitat, we refer to
\cite{Li_2018} for an integro-difference competition model, and to\cite{Du_2018} for a competition model with free-boundaries. See also \cite{Guo_2015,Wu_2015,Wang_2017,Wang_2018,Liu_2019} for other related results in free-boundary problems. We also note that in those works the spreading speeds are always locally determined and thus do not interact.

\section{Estimating the maximal and minimal speeds}\label{S2}
The concepts of maximal and minimal spreading speeds are introduced in \cite[Definition 1.2]{Hamel2012} for a single species;  see also \cite{Garnier_2012,LLL2019}. In our setting, we define
\begin{equation}\label{eq:speeds}
\begin{cases}
\smallskip
\overline{c}_1=\inf{\{c>0~|~\limsup \limits_{t\rightarrow \infty}\sup\limits_{x>ct} v(t,x)=0\}},\\
\smallskip
\underline{c}_1=\sup{\{c>0~~|\liminf \limits_{t\rightarrow \infty}\inf\limits_{ct-1<x<ct} v(t,x)>0\}}, \\
\smallskip
\overline{c}_2=\inf\{c>0~|~\limsup\limits_{t\rightarrow \infty}\sup\limits_{ x> ct}u(t,x)=0\}, \\
\smallskip
\underline{c}_2=\sup{\{c>0~~|\liminf \limits_{t\rightarrow \infty}\inf\limits_{ct-1<x<ct} u(t,x)>0\}},\\
\smallskip
\overline{c}_3=\inf\{c<0~|~\liminf\limits_{t\rightarrow \infty}\inf\limits_{ ct<x< ct+1}v(t,x) >0\},\\
\underline{c}_3=\sup\{c<0~|~\limsup\limits_{t\rightarrow \infty}\sup\limits_{ x< ct}v(t,x)=0\},
 \end{cases}
\end{equation}
where $\overline{c}_1$ and $\underline{c}_1$ (resp.  $\overline{c}_2$ and $\underline{c}_2$) are the maximal and minimal rightward spreading speeds of species $v$ (resp. species $u$), whereas  $-\underline{c}_3$ and $-\overline{c}_3$ are the maximal and minimal leftward spreading speeds of $v$, respectively.

In this section, for initial data satisfying ($\mathrm{H_\lambda})$, we will give some estimates of the maximal and minimal spreading speeds. The main result of this section can be precisely stated as follows.


\begin{proposition}\label{prop:1}
Let $(u,v)$ be a solution of \eqref{eq:1-1} with initial data  satisfying $\rm{(H_\lambda)}$. Then the spreading speeds defined in \eqref{eq:speeds} satisfy
\begin{itemize}
\item[\rm{(i)}] $\overline{c}_i\leq \sigma_i$ for $i=1,\,2$ and $\overline{c}_3\leq -\sigma_3$;
\item[\rm{(ii)}] $\underline{c}_2\geq c_{\rm LLW}$, and $\overline{c}_3\leq -\tilde{c}_{\rm LLW}$,
\end{itemize}
where $\sigma_1, \sigma_2, \sigma_3$ are defined in \eqref{eq:sigma} and $c_{\rm LLW},\, \tilde{c}_{\rm LLW}$ are given respectively in Theorem \ref{thm:LLW} and Remark \ref{rmk:LLW}. Furthermore, we have
\begin{equation}\label{eq:spreadingly2.1}
\lim\limits_{t\to \infty} (|u(t,0)-k_1| + |v(t,0)-k_2|) =0.
\end{equation}
\end{proposition}
\begin{proof}
  We will complete the proof in the following order: \rm{(1)} $\overline{c}_2 \leq \sigma_2$,  \rm{(2)} $\overline{c}_1\leq \sigma_1$, \rm{(3)} $\overline c_3\leq -\sigma_3$,  \rm{(4)} $\overline c_3\leq -\tilde c_{\rm LLW}$, \rm{(5)} $\underline c_2 \geq c_{\rm LLW}$, \rm{(6)} \eqref{eq:spreadingly2.1} holds. \\
  \noindent{\bf Step 1.}  We show assertions \rm{(1)}, \rm{(2)} and \rm{(3)}.

Observe that for some $M>0$ the function
$$
\overline{u}(t,x):= \min\{1, M \exp(-\min\{\lambda_u,1\}(x - \sigma_2 t))\}
$$
is a weak super-solution to
 the single KPP-type equation
\begin{equation*}
\partial_t  \overline{u} - \partial_{xx} \overline{u} = \overline{u} (1-\overline{u} ) \,\quad \text{in }(0,\infty) \times \mathbb{R},
\end{equation*}
of which $u(t,x)$ is clearly a sub-solution.
By choosing the constant $M>0$ so large that $u_0(x)\leq \overline{u}(0,x)$ in $\mathbb{R}$, it follows by comparison that
\begin{equation}\label{eq:ukpp00}
u(t,x) \leq \overline{u}(t,x)=  \min\left\{1, M \exp(-\min\{\lambda_u,1\}\left(x -\sigma_2 t\right) \right\} 
\end{equation}
for $(t,x) \in [0,\infty) \times \mathbb{R}$.  In particular,
\begin{equation}\label{eq:ukpp}
\lim_{t \to \infty} \sup_{x > (\sigma_2+\eta)t} |u(t,x)| =0\,\quad \text{ for each }\eta >0.
\end{equation}
This proves $\overline{c}_2 \leq \sigma_2$, i.e., assertion \rm{(1)} holds.

Similarly, we deduce assertion \rm{(2)}
by comparison with
$$
\bar{v}(t,x):= \min\{1, M \exp(-\min\{\lambda^+_v, \sqrt{r/d}\}(x-\sigma_1t))\}
$$
which is
 the solution of
 \begin{equation*}
\left\{
\begin{array}{ll}
\partial_t \overline{v} - d\partial_{xx} \overline{v} = r\overline{v}(1-\overline{v}), &\text{in }(0,\infty) \times \mathbb{R},\\
\overline{v}(0,x)=\min(1,M e^{-\min\{\lambda_v^+,\sqrt{\frac{r}{d}}\}x}), & x\in\mathbb{R}.
\end{array}
\right.
\end{equation*}

To prove assertion \rm{(3)}, let $\tilde v(t,x)=v(t,-x)$, we turn to consider another single KPP-type equation
 \begin{equation*}
\left\{
\begin{array}{ll}
\partial_t \underline{v} - d\partial_{xx} \underline{v} = r\underline{v}(1-b-\underline{v}), &\text{in }(0,\infty) \times \mathbb{R},\\
\underline{v}(0,x)=v_0(-x), & x\in\mathbb{R}.
\end{array}
\right.
\end{equation*}
Again the scalar comparison principle implies $v(t,-x)=\tilde v(t,x)\geq \underline{v}$. By the results in \cite{Kametaka_1976} or \cite{Mckean_1975}, we have
\begin{equation}\label{estim_3}
 \liminf\limits_{t\to\infty} \inf\limits_{ (-\sigma_3+\eta)t<x\leq 0} v\geq \liminf\limits_{t\to\infty} \inf\limits_{ |x|<(\sigma_3-\eta)t} \tilde v\geq \frac{1-b}{2},
\end{equation}
which means $\overline{c}_3\leq -\sigma_3$.

\noindent{\bf Step 2.} We show assertions \rm{(4)} and \rm{(5)}.

Given any non-trivial, compactly supported function $\tilde {v}_0$  such that $0 \leq \tilde{v}_0 \leq v_0$. Then
$$
(u_0(x),v_0(x)) \preceq (1, \tilde{v}_0(x)) \,\quad \text{ in }\mathbb{R}.
$$
Let $(\tilde u_{\rm LLW}, \tilde v_{\rm LLW})$ be the solution to \eqref{eq:1-1} with initial value $(1, \tilde{v}_{0}(x))$. Then
Theorem \ref{thm:LLW} and Remark \ref{rmk:LLW} guarantee the existence of
$\tilde{c}_{\rm LLW} \ \geq 2\sqrt{dr(1-b)}$, such that 
$$
\liminf_{t \to \infty} \inf_{|x| < |c|t } \tilde v_{\rm LLW}(t,x)  >0  \quad \text{ for each } c \in (-\tilde{c}_{\rm LLW}, 0).
$$
By the comparison principle for \eqref{eq:1-1}, we have
$(u,v) \preceq (\tilde u_{\rm LLW},\tilde v_{\rm LLW})$ for all $(t,x) \in (0,\infty) \times \mathbb{R}$, which yields, for each $c \in (- \tilde{c}_{\rm LLW}, 0)$,
$$
\liminf_{t \to \infty} \inf_{ct < x < ct + 1} v(t,x) \geq  \liminf_{t \to \infty} \inf_{ct < x < ct + 1} \tilde v_{\rm LLW}(t,x)  >0.
$$
This proves $\overline{c}_3 \leq - \tilde{c}_{\rm LLW}$ and thus assertion \rm{(4)} holds.

Similarly, we can get show assertion (5), i.e.,  $\underline{c}_2 \geq c_{\rm LLW}$. By comparing $(u,v)$ with the solution $(u_{\rm LLW}, v_{\rm LLW})$ of \eqref{eq:1-1} with initial condition $(\tilde{u}_0, 1)$, for some compactly supported $\tilde{u}_0$ satisfying $0 \leq \tilde{u}_0 \leq u_0$, and then using  Theorem \ref{thm:LLW}. In this way, we get
\begin{align}\label{eq:positiveu'}
\liminf_{t \to \infty} \inf_{|x| < c t } u \geq \liminf_{t \to \infty} \inf_{|x| < c t } u_{\rm LLW} >0 \quad \text{ for each } c \in (0, c_{\rm LLW}).
\end{align}

\noindent{\bf Step 3.} We show assertion \rm{(6)}.
 In view of \eqref{estim_3} and \eqref{eq:positiveu'}, one can deduce \eqref{eq:spreadingly2.1} from items \rm{(a)} and \rm{(c)} of Lemma \ref{lem:entire1}.
\end{proof}

\section{Estimating $\displaystyle \overline c_1$ and $\displaystyle \overline c_2$ from below}\label{S3} 
We assume $\sigma_1>\sigma_2$ throughout this section. In this section, we estimate $\underline c_1$ and $\underline c_2$ from below via the large deviation approach and applying Theorem \ref{thm:D}.
To this end, we introduce a small parameter $\epsilon$ via the following scaling
\begin{equation}\label{scaling}
u^\epsilon(t,x)=u\left(\frac{t}{\epsilon},\frac{x}{\epsilon}\right)\quad \mathrm{and}\quad v^\epsilon(t,x)=v\left(\frac{t}{\epsilon},\frac{x}{\epsilon}\right).
\end{equation}
Under the new scaling, we rewrite  the equation of $u^\epsilon$ and $v^\epsilon$ in \eqref{eq:1-1} as
\begin{align}\label{eq:1-1'}
\left \{
\begin{array}{ll}
\partial_tu^\epsilon=\epsilon\partial_{xx} u^\epsilon+\frac{u^{\epsilon}}{\epsilon}(1-u^\epsilon-av^\epsilon), & \text{ in } (0,\infty)\times\mathbb{R},\\
\partial_tv^\epsilon=\epsilon d\partial_{xx} v^\epsilon+r\frac{v^{\epsilon}}{\epsilon}(1-bu^\epsilon-v^\epsilon), & \text{ in } (0,\infty)\times\mathbb{R},\\
u^\epsilon(0,x)=u_0(\frac{x}{\epsilon}),& \text{ on } \mathbb{R},\\
v^\epsilon(0,x)=v_0(\frac{x}{\epsilon}), & \text{ on } \mathbb{R}.
\end{array}
\right.
\end{align}

To obtain the asymptotic behaviors of  $v^\epsilon$ and $u^\epsilon$ as $\epsilon\rightarrow 0$, the idea is to consider the WKB ansatz $w_1^\epsilon$ and $w_2^\epsilon$, which are given respectively  by
\begin{equation}\label{eq:w}
w_1^\epsilon(t,x)=-\epsilon\log{v^\epsilon(t,x)}, \quad w_2^\epsilon(t,x)=-\epsilon\log{u^\epsilon(t,x)},
\end{equation}
 and satisfy, respectively, the equations
\begin{align}\label{eq:epsilonw1}
\left \{
\begin{array}{ll}
\partial_t w^\epsilon-\epsilon d\partial_{xx} w^\epsilon+d| \partial_xw^\epsilon|^2+r(1-bu^\epsilon-v^\epsilon)=0, & \text{ in } (0,\infty)\times\mathbb{R},\\
w^\epsilon(0,x)=-\epsilon\log{v^\epsilon(0,x)},  & \text{ on } ~\mathbb{R},\\
\end{array}
\right.
\end{align}
and
\begin{align}\label{eq:epsilonw2}
\left \{
\begin{array}{ll}
\partial_tw^\epsilon-\epsilon\partial_{xx} w^\epsilon+| \partial_xw^\epsilon|^2+1-u^\epsilon-av^\epsilon=0, & \text{ in } (0,\infty)\times\mathbb{R},\\
w^\epsilon(0,x)=-\epsilon\log{u^\epsilon(0,x)},  & \text{ on } \mathbb{R}.\\
\end{array}
\right.
\end{align}

\begin{lemma}\label{lem:underlinec} Let $G$ be an open set in $(0,\infty)\times \mathbb{R}$ and $K, K'$ be compact sets such that $K \subset {\rm Int}\,K'  \subset K'  \subset G.$
\begin{itemize}
\item[\rm{(a)}] If
$w_2^\epsilon \to 0$ uniformly in $K'$ as $\epsilon \to 0$, then
 \begin{equation}\label{eq:underlineu}
\liminf_{\epsilon\to 0} \inf_{K}
 u^\epsilon \geq 1-a\limsup_{\epsilon\to 0} \sup_{K'}
v^\epsilon;
\end{equation}
\item[\rm{(b)}] If
$w_1^\epsilon \to 0$ uniformly in $K' $ as $\epsilon \to 0$, then
 \begin{equation}\label{eq:underlinev}
\liminf_{\epsilon\to 0} \inf_{K}
 v^\epsilon \geq 1-b\limsup_{\epsilon\to 0} \sup_{K'}
 u^\epsilon.
\end{equation}
\end{itemize}
\end{lemma}
\begin{proof}
We first prove (a) by adapting the arguments from \cite[Section 4]{Evans_1989}.
Let $K, K'$ and $G$ be given as above.


Fix an arbitrary $(t_0,x_0) \in K$ and define the test function
$$
\rho(t,x) = |x-x_0|^2 + (t-t_0)^2.
$$
Since (i)  $(t_0,x_0) \in K \subset {\rm Int}\, K'$ and (ii) $w^\epsilon_2 \to 0$ uniformly in $K'$, the function $w^\epsilon_2 - \rho$ attains global maximum over $K'$ at $(t_\epsilon,x_\epsilon)\in {\rm Int}\,K'$ such that $(t_\epsilon,x_\epsilon) \to (t_0,x_0)$ as $\epsilon \to 0$.
Furthermore, $\partial_t \rho(t_\epsilon,x_\epsilon), \partial_x \rho(t_\epsilon,x_\epsilon) \to 0$, so that at the point $(t_\epsilon,x_\epsilon)$,
$$
o(1) = \partial_t \rho - \epsilon \partial_{xx} \rho + |\partial_x\rho|^2 \leq \partial_t w_2^\epsilon - \epsilon \partial_{xx} w_2^\epsilon + |\partial_x w_2^\epsilon|^2 \leq u^\epsilon - 1 + a\limsup_{\epsilon\to 0} \sup_{K'}
 v^\epsilon.
$$
 This yields
$$
u^\epsilon(t_\epsilon, x_\epsilon) \geq 1-a\limsup_{\epsilon\to 0} \sup_{K'}
 v^\epsilon + o(1).
$$
Since $w^\epsilon_2 - \rho$ attains maximum over $K'$ at $(t_\epsilon,x_\epsilon)$,  we have in particular
$$
w_2^\epsilon(t_\epsilon,x_\epsilon)  \geq (w_2^\epsilon - \rho)(t_\epsilon,x_\epsilon) \geq(w_2^\epsilon - \rho)(t_0,x_0) = w_2^\epsilon(t_0,x_0).
$$
%
Recalling $u^\epsilon(t_0,x_0) = e^{-\epsilon w^\epsilon_2(t_0,x_0)}$ and $u^\epsilon(t_\epsilon,x_\epsilon) = e^{-\epsilon w^\epsilon_2(t_\epsilon,x_\epsilon)}$, we therefore have
$$
u^\epsilon(t_0,x_0) \geq u^\epsilon(t_\epsilon,x_\epsilon) \geq 1-a\limsup_{\epsilon\to 0} \sup_{K'}
 v^\epsilon + o(1).
$$ Since this argument is uniform for $(t_0,x_0) \in K$ (depends only on $K, K'$ and $G$), we deduce assertion \rm{(a)}.  The proof for {\rm(b)} is analogous.
\end{proof}

Next, we will pass to the (upper and lower) limits using the half-relaxed limit method, 
which is due to  Barles and Perthame \cite{BP1987}. Define
$$w_1^*(t,x)=\limsup\limits_{\scriptsize \begin{array}{c}\epsilon \to 0\\ (t',x') \to (t,x)\end{array}} \hspace{-.3cm} w_1^\epsilon (t',x'),$$
$$w_2^*(t,x)=\limsup\limits_{\scriptsize \begin{array}{c}\epsilon \to 0\\ (t',x') \to (t,x)\end{array}} \hspace{-.3cm} w_2^\epsilon (t',x')\quad \mathrm{and} \quad w_{2,*}(t,x)=\liminf\limits_{\scriptsize \begin{array}{c}\epsilon \to 0\\ (t',x') \to (t,x)\end{array}} \hspace{-.3cm} w_2^\epsilon (t',x').$$

That the above are well defined is due to
 the following lemma:
\begin{lemma}\label{lem:3-1'}
Let $w_1^\epsilon$ and $w_2^\epsilon$ be the solutions to \eqref{eq:epsilonw1} and \eqref{eq:epsilonw2}, respectively. Then there exits some $Q>0$, independent of $\epsilon$ small, such that
\begin{subequations}\label{eq:spreadinglyprop}
\begin{equation}\label{estamition1}
  \max\{\lambda_v^+x_++\lambda_v^-x_-- Q(t+\epsilon),0\} \leq {w}_1^{\epsilon}(t,x)
\leq\lambda_v^+x_++\lambda_v^-x_-+Q(t+\epsilon),
\end{equation}
\begin{equation}\label{estamition2}
\max\{\lambda_u x_+- Q(t+\epsilon),0\} \leq {w}_2^{\epsilon}(t,x)
\leq \lambda_u x_++  Q(t+\epsilon),
\end{equation}
\begin{equation}\label{estamition3}
0 \leq {w}_1^{\epsilon}(t,x)
\leq   Q(\lambda_v^+x_++\lambda_v^-x_-+\epsilon),
\end{equation}
\begin{equation}\label{estamition4}
0 \leq {w}_2^{\epsilon}(t,x) \leq  Q(\lambda_ux_++\epsilon),
\end{equation}
\end{subequations}
 for $ (t,x) \in [0,\infty)\times \mathbb{R}$, where $x_+=\max\{x,0\}$ and $x_-=\max\{-x,0\}$.
\end{lemma}
 \begin{proof}
We only prove \eqref{estamition1} and the estimations \eqref{estamition2}-\eqref{estamition4}  follow from a quite similar argument. Since $v^\epsilon\leq 1$, we have $w_1^\epsilon\geq 0$ by definition.  By ${(\rm{H_\lambda})}$,  
 there exist positive constants $C_1$ and $C_2$ such that
$$
 C_2e^{-(\lambda_v^+x_++\lambda_v^- x_-)}\leq v(0,x)\leq C_1e^{-(\lambda_v^+x_++\lambda_v^- x_-)} \quad \text{ for }  x \in \mathbb{R}.
$$
By definition \eqref{eq:w}, we have
\begin{equation}\label{eq:w10}
\lambda_v^+x_++\lambda_v^- x_--\epsilon\log C_1\leq w_1^\epsilon(0,x)\leq\lambda_v^+x_++\lambda_v^- x_--\epsilon\log C_2 .
\end{equation}

Define
$$\overline{z}^\epsilon_1=\lambda_v^+ x+Q(t+\epsilon).$$
We shall choose large $Q$ independent of $\epsilon$ such that
\begin{equation}\label{estimate}
  w_1^\epsilon(t,x)\leq \overline{z}^\epsilon_1 \,\quad\text{ in } [0,\infty)\times [0,\infty).
\end{equation}

 To this end,  observe that $\overline{z}^\epsilon_1$ is a (classical) super-solution of \eqref{eq:epsilonw1} in $(0,\infty)\times (0,\infty)$ provided $Q\geq r$. By \eqref{eq:spreadingly2.1} in Proposition \ref{prop:1}, we find $-\log v(t,0)$ is uniformly bounded in $[0,\infty)$ (since $v(0,x)>0$ in $\mathbb{R}$), so that we may choose
 \begin{equation}\label{Q}
   Q=\max\left\{\sup_{t\in[0,\infty)}[-\log v(t,0)], \,|\log C_2|,\, r\right\},
 \end{equation}
such that
$$w^\epsilon_1(t,0)\leq \overline{z}^\epsilon_1(t,0)\,\quad \text{ for all }t\geq 0, \quad  \,w^\epsilon_1(0,x)\leq \overline{z}^\epsilon_1(0,x)\quad \text{ for all }x\geq 0,$$
where the last inequality is due to \eqref{eq:w10}. By comparison,  \eqref{estimate} thus holds.

By a similar argument, we can verify
$$\overline{z}^\epsilon_2=-\lambda_v^- x+Q(t+\epsilon)$$
is a super-solution of \eqref{eq:epsilonw1} in $(0,\infty)\times (-\infty,0)$, so that
\begin{equation}\label{estimate1}
  w_1^\epsilon(t,x)\leq \overline{z}^\epsilon_2(t,x) \,\quad\text{ in } [0,\infty)\times (-\infty, 0],
\end{equation}
where $Q$ is defined by \eqref{Q}. Combining with \eqref{estimate} and \eqref{estimate1} gives the desired upper bound of $ w_1^\epsilon$.

To obtain the lower bound of $ w_1^\epsilon$, we may define functions
$$\underline{z}^\epsilon_1=\lambda_v^+ x-Q(t+\epsilon)\,\quad \text{ and }\,\quad\underline{z}^\epsilon_2=-\lambda_v^+ x-Q(t+\epsilon).$$
By the same arguments as before, we can check
\begin{equation*}
 w_1^\epsilon(t,x)\geq \underline{z}^\epsilon_1 \,\, \text{ in } [0,\infty)\times \mathbb{R}\,\,\text{ and } \,w_1^\epsilon(t,x)\geq \underline{z}^\epsilon_2 \,\, \text{ in } [0,\infty)\times \mathbb{R},
\end{equation*}
by choosing $Q=\max\left\{ \,|\log C_1|,\, d(\lambda_v^+)^2+d(\lambda_v^-)^2 +r \right\}$.
This completes the proof of \eqref{estamition1}.
\end{proof}

\begin{remark}\label{rmk:w1w20}
According to Lemma \ref{lem:3-1'}, by letting $t=0$ and then $\epsilon \to 0$ in \eqref{estamition1} and \eqref{estamition2}, we deduce that
\begin{equation*}
  w^*_1(0,x)=
\begin{cases}
\lambda^+_v x, &  \text{for }x\in [0,\infty),\\
\lambda^-_v x, &\text{for }x\in (-\infty,0],
\end{cases}
\end{equation*}
and
\begin{equation*}
w^*_2(0,x)=w_{2,*}(0,x)=
\begin{cases}
\lambda_u x, &  \text{for }x\in [0,\infty),\\
0, &\text{for }x\in (-\infty,0].
\end{cases}
\end{equation*}
Similarly, by setting $x=0$ and then $\epsilon \to 0$ in \eqref{estamition3} and \eqref{estamition4}, we have
$$w_1^*(t,0)=w_2^*(t,0)=w_{2,*}(t,0)=0\,\,\text{ for }\,t\geq 0.$$
\end{remark}

\subsection{Estimating $\underline{c}_1$ from below}
By Proposition \ref{prop:1}, $\overline c_2 \leq \sigma_2$, so we deduce
\begin{equation}\label{eq:u_ep}
0\leq \hspace{-.3cm}\limsup_{\scriptsize \begin{array}{c}\epsilon \to 0\\ (t',x') \to (t,x)\end{array}} \hspace{-.5cm}  u^\epsilon(t',x')\leq \chi_{\{x \leq \sigma_2 t\}}.
\end{equation}
\begin{lemma}\label{lem:sub-solutionw1}
Let $(u,v)$ be a solution of \eqref{eq:1-1} with initial data satisfying $\rm{(H_\lambda)}$. Then
\begin{itemize}
\item[{\rm (a)}] $w_1^*$ is a viscosity sub-solution of
\begin{align}\label{eq:sub-solutionw1}
\begin{cases}
\min\{\partial_t w+d|\partial_xw|^2+r(1-b\chi_{\{x\leq \sigma_2 t\}}) ,w\}=0,&\text{ in } (0,\infty)\times(0,\infty),\\
w(0,x)=\lambda_v^+x, & \text{ on }[0,\infty),\\
w(t,0)= 0, &\text{ on }[0,\infty);
\end{cases}
\end{align}
\item[{\rm (b)}] $w_1^*$ is a viscosity sub-solution of
\begin{align}\label{eq:sub-solutionw1'}
\begin{cases}
\min\{\partial_t w+d|\partial_xw|^2+r(1-b\chi_{\{x\leq \sigma_2 t\}}) ,w\}=0,&\text{ in } (0,\infty)\times(-\infty,0),\\
w(0,x)=-\lambda_v^-x, & \text{ on }(-\infty,0],\\
w(t,0)= 0, &\text{ on }[0,\infty),
\end{cases}
\end{align}
\end{itemize}
where $\sigma_2$ is defined by \eqref{eq:sigma} and $\lambda_v^-,\,\lambda_v^+ \in (0,\infty)$ are given in $\mathrm{(H_\lambda)}$.
\end{lemma}
\begin{proof}
First, observe that  $w^*_1$ is upper semicontinuous (usc) by construction. By Remark \ref{rmk:w1w20},
the initial and boundary conditions of \eqref{eq:sub-solutionw1} and \eqref{eq:sub-solutionw1'} are satisfied.


It remains to show that $w^*_1$ is a viscosity sub-solution of $\min\{\partial_t w+d|\partial_x w|^2+r(1-b\chi_{\{x\leq \sigma_2 t\}}) ,w\}=0$ in the domain $(0,\infty)\times \mathbb{R}$.
According to definition of viscosity sub-solution of  Hamilton-Jacobi equation, (see Appendix \ref{SD}), 
let $ \varphi \in C^\infty((0,\infty)\times\mathbb{R})$ and let $(t_0,x_0)$ be a strict local maximum point of $w_1^*-\varphi$ such that $w_1^*(t_0,x_0)>0$. By passing to a sequence $\epsilon = \epsilon_k$ if necessary, $ w_1^\epsilon-\varphi$ has a local maximum point at $(t_\epsilon,x_\epsilon)$ such that $w_1^\epsilon(t_\epsilon,x_\epsilon) \to w_1^*(t_0,x_0)$ and  $(t_\epsilon,x_\epsilon)\to (t_0,x_0)$ uniformly as $\epsilon\to 0$.
 At the point $(t_\epsilon, x_\epsilon)$, we have
\begin{align*}
\epsilon d\partial_{xx}\varphi \geq \epsilon d\partial_{xx} w_1^\epsilon &= \partial_t w_1^\epsilon+d|\partial_x w_1^\epsilon|^2+r(1-bu^\epsilon -e^{-\frac{w_1^\epsilon}{\epsilon}})\\
&= \partial_t \varphi+d|\partial_x \varphi|^2+r(1-bu^\epsilon -e^{-\frac{w_1^\epsilon}{\epsilon}}).
\end{align*}
By the fact that $e^{-w_1^\epsilon(t_\epsilon,x_\epsilon)/\epsilon}\to 0$ (as $w_1^\epsilon(t_\epsilon,x_\epsilon) \to w_1^*(t_0,x_0)>0$), we may pass to the limit $\epsilon = \epsilon_k \to 0$ so that
$$
0 \geq \partial_t \varphi(t_0,x_0)+d|\partial_x \varphi(t_0,x_0)|^2+r(1-b\chi_{\{(t,x):\,x \leq \sigma_2 t\} (t_0,x_0)}-0).
$$
Hence $w_1^*$ is a viscosity sub-solution of \eqref{eq:sub-solutionw1} and \eqref{eq:sub-solutionw1'}.
\end{proof}

\begin{lemma}\label{lem:underlinec1}
Let $(u,v)$ be a solution of \eqref{eq:1-1} with initial data satisfying $\rm{(H_\lambda)}$.
Then
$$\underline{c}_1\geq \sigma_1,$$
where $\sigma_1$ is defined by \eqref{eq:sigma}.
\end{lemma}
\begin{proof}
Define the function $\overline{w}_1: [0,\infty) \times [0,\infty) \to [0,\infty)$ by
\begin{equation*}
\begin{split}
\overline w_1(t,x) =
\left\{
\begin{array}{ll}
\medskip
\lambda_v^+(x-(d\lambda_v^++\frac{r}{\lambda_v^+})t), & \text{for}~ \frac{x}{t}>d\lambda_v^++\frac{r}{\lambda_v^+},\\
0, &  \text{for}~0\leq \frac{x}{t}\leq d\lambda_v^++\frac{r}{\lambda_v^+},\\
\end{array}
\right.
\end{split}
\end{equation*}
when $ \lambda_v^+ \leq \sqrt{\frac{r}{d}}$, and by
\begin{equation*}
\begin{split}
\overline w_1(t,x) =
\left\{
\begin{array}{ll}
\medskip
\lambda_v^+(x-(d\lambda_v^++\frac{r}{\lambda_v^+})t), & \text{for}~ \frac{x}{t}> 2d\lambda_v^+,\\
\medskip
\frac{t}{4d}(\frac{x^2}{t^2}-4dr), &\text{for}~ 2\sqrt{dr}< \frac{x}{t}\leq 2d\lambda_v^+,\\
0, & \text{for}~0\leq \frac{x}{t}\leq 2\sqrt{dr},\\
\end{array}
\right.
\end{split}
\end{equation*}
when $ \lambda_v^+ > \sqrt{\frac{r}{d}}$.

By construction, $\overline{w}_1$ is continuous in $[0,\infty)\times [0,\infty)$. Next, we claim that the continuous $\overline{w}_1$ is a viscosity super-solution of \eqref{eq:sub-solutionw1}. We will check the latter case of $ \lambda_v^+ > \sqrt{\frac{r}{d}}$ as the former case can be verified analogously.
%
%
%
Under the condition $ \lambda_v^+ > \sqrt{\frac{r}{d}}$, we have $\sigma_1=2\sqrt{dr}$.
According to definition of viscosity super-solution of Hamilton-Jacobi equation (see Appendix \ref{SD}), let $\varphi \in C^\infty((0,\infty)\times\mathbb{R})$ and let $(t_0,x_0)$ be a strict local minimum point of $\overline w_1-\varphi$. 

If $x_0/t_0 \neq 2\sqrt{dr}$, then $\overline{w}_1$ is a classical solution of \eqref{eq:sub-solutionw1}.

If $x_0/t_0 = 2\sqrt{dr}$, then $\overline{w}_1(t_0,x_0) = 0$ by definition. Moreover,
$$
 -\varphi(t, 2\sqrt{dr} t) = (\overline{w}_1 -\varphi)(t, 2\sqrt{dr} t)  \geq (\overline{w}_1 -\varphi)(t_0,x_0) = - \varphi(t_0,x_0) \quad \text{ for }t \approx t_0,$$ and we must have $\partial_t \varphi(t_0,x_0)+2\sqrt{dr} \partial_x \varphi(t_0,x_0)) = 0$, and hence
\begin{align*}
&\quad \partial_t \varphi(t_0,x_0)+ d|\partial_x \varphi(t_0,x_0)|^2 + r(1-b \chi_{\{(t,x):\,x \leq \sigma_2 t\}} (t_0,x_0))\\
&= -2\sqrt{dr} \partial_x \varphi(t_0,x_0)+ d|\partial_x \varphi(t_0,x_0)|^2  + r\\
&= \left(\sqrt{d}\partial_x \varphi(t_0,x_0) - \sqrt{r}\right)^2 \geq 0,
\end{align*}
where the first equality follows from the fact that $x_0/t_0 = 2\sqrt{dr}=\sigma_1 > \sigma_2$.

By Remark \ref{rmk:w1w20} and the expression of $\overline{w}_1$, we have
$$
\overline{w}_1(t,x) = \lambda^+_v x=w_1^*(t,x) \,\,\text{ on } \partial [(0,\infty) \times (0,\infty)].
$$
And recalling Lemma \ref{lem:sub-solutionw1}(a), $\overline{w}_1$ and $w_1^*$ is a pair of viscosity super and sub-solutions of \eqref{eq:sub-solutionw1}. Then, we may apply Theorem \ref{thm:D} to get 
$$0 \leq w^*_1 \leq \overline{w}_1 \quad \quad \text{ in } [0,\infty)\times[0,\infty),$$
which implies that
$$\{(t,x):w_1^*(t,x)=0\}\supset\{(t,x):\overline w_1(t,x)=0\}= \{(t,x):0\leq x\leq  \sigma_1t\}.$$
Letting $\epsilon \to 0$, we arrive at
$$w_1^\epsilon(t,x) = -\epsilon \log{v^\epsilon(t,x)} \to 0 \text{ locally uniformly on } \{(t,x):0\leq x < \sigma_1t\}.$$
Hence for each small $\eta>0$, by choosing the compact sets $K=\{(1,x): \eta \leq x \leq \sigma_1 - \eta \}$ and $K'=\{(1,x): \frac{\eta}{2} \leq x \leq \sigma_1 - \frac{\eta}{2} \}$, we may apply Lemma \ref{lem:underlinec}\rm{(b)} to deduce that
\begin{equation}\label{eq:asymptoticvc1'}
\liminf_{t\to \infty}\inf_{\eta  t< x < (\sigma_1 - \eta) t} v(t,x)=\liminf_{\epsilon\to 0}\inf_{K} v^\epsilon (t,x)\geq \frac{1-b}{2}>0. 
\end{equation}
This implies $\underline{c}_1\geq \sigma_1$.
\end{proof}

\begin{corollary}\label{cor:underlinec1}
 Let $\sigma_1 > \sigma_2$ and let $(u,v)$ be a solution of \eqref{eq:1-1} with initial data satisfying $\rm{(H_\lambda)}$.  Then for each small $\eta>0$,
\begin{subequations}
\begin{align}
&\lim\limits_{t\to\infty}\sup\limits_{x>(\sigma_1+\eta)t} (|u|+|v|)=0,\label{eq:spreadc1_a} \\
&\lim\limits_{t\to\infty}\sup\limits_{(\overline{c}_2+\eta)t<x<(\sigma_1-\eta)t} (|u|+|v-1|)=0, \label{eq:spreadc1_b}
\end{align}
\end{subequations}
where $\sigma_1$ is defined by \eqref{eq:sigma}.
\end{corollary}
\begin{proof}
 By definition, $\underline{c}_1 \leq \overline{c}_1$. It follows from Proposition \ref{prop:1} and Lemma \ref{lem:underlinec1} that  $\sigma_1 \leq \underline{c}_1\leq \overline{c}_1 \leq \sigma_1$. Hence, $\underline{c}_1=\overline{c}_1=\sigma_1$. By Proposition \ref{prop:1}(i), $\overline{c}_2 \leq \sigma_2 < \sigma_1$, so that \eqref{eq:spreadc1_a} holds.
In view of \eqref{eq:asymptoticvc1'} and  definition of $\overline{c}_2$, we have, for each small $\eta>0$,
$$
\liminf_{t\to \infty}\inf_{\eta  t< x < (\sigma_1 - \eta) t} v(t,x)>0,\quad \text{ and }\quad
\lim_{t\to\infty} \sup_{x > (\overline{c}_2 + \eta)t} u = 0.
$$
We may then apply Lemma \ref{lem:entire1}{\rm{(d)}} to deduce \eqref{eq:spreadc1_b}.
\end{proof}

\subsection{Estimating $\underline c_2$ from below}
By Corollary \ref{cor:underlinec1}, we have
\begin{equation}\label{eq:v_ep}
\chi_{\{\sigma_2 t<x<\sigma_1t\}}\leq\hspace{-.5cm}\liminf_{\scriptsize \begin{array}{c}\epsilon \to 0\\ (t',x') \to (t,x)\end{array}} \hspace{-.5cm}  v^\epsilon(t',x')\leq\hspace{-.5cm}\limsup_{\scriptsize \begin{array}{c}\epsilon \to 0\\ (t',x') \to (t,x)\end{array}} \hspace{-.5cm}  v^\epsilon(t',x')\leq \chi_{\{x \leq \sigma_1 t\}}.
\end{equation}
\begin{lemma}\label{lem:sub-solutionw2}
Let $( u, v)$ be a solution of \eqref{eq:1-1} with initial data satisfying $\rm{(H_\lambda)}$. Then, $w_2^*$ 
  is a viscosity  sub-solution of
\begin{align}\label{eq:sub-solutionw2}
\left \{
\begin{array}{ll}
\min\{\partial_t w+|\partial_xw|^2+1-a\chi_{\{x\leq \sigma_1 t\}} ,w\}=0,&\text{ in } (0,\infty)\times\mathbb{R},\\
w(0,x)=\lambda_u\max\{x, 0\}, & \text{ on }\mathbb{R},
\end{array}
\right.
\end{align}
where $\sigma_1$ is defined by \eqref{eq:sigma} and $\lambda_u >0$ is given in $\mathrm{(H_\lambda)}$.
\end{lemma}
\begin{proof}
The proof is analogous to Lemma \ref{lem:sub-solutionw1} and we omit the details.
\end{proof}

\begin{lemma}\label{lem:underlinec2}
Let $(u,v)$ be a solution of \eqref{eq:1-1} with initial data satisfying $\rm{(H_\lambda)}$. Then
$$\underline c_2\geq \hat{c}_{\rm nlp},$$
where $\hat{c}_{\rm nlp}$ is defined in Theorem \ref{thm:1-2}. 
\end{lemma}
\begin{proof}
According to definition of  $\hat{c}_{\rm nlp}$ in Theorem \ref{thm:1-2}, we consider the three cases separately: (a) $\sigma_1 < 2\lambda_u$ and $\sigma_1 < 2(\sqrt a + \sqrt{1-a})$; (b) $\sigma_1 \geq 2\lambda_u$  and $\tilde\lambda_{\rm nlp}\leq \sqrt{1-a}$; (c) otherwise.

First, we claim that we are done in case (c). Since in that case $\hat{c}_{\rm nlp} = 2\sqrt{1-a}$, and, according to Proposition \ref{prop:1}(ii), $\underline{c}_2 \geq c_{\rm LLW}$ where $c_{\rm LLW} \geq 2\sqrt{1-a}$ by Theorem \ref{thm:LLW}. Thus
$$
\underline{c}_2 \geq c_{\rm LLW} \geq 2\sqrt{1-a} = \hat{c}_{\rm nlp}.
$$
It remains to consider cases (a) and (b). We start by defining
$$\bar{c}_{\rm nlp} = \frac{\sigma_1}{2} - \sqrt{a} + \frac{1-a}{\frac{\sigma_1}{2} - \sqrt{a}}$$
 and
\begin{equation}\label{eq:tlambdanlp}
\tilde{c}_{\rm nlp} = \tilde\lambda_{\rm nlp} + \frac{1-a}{\tilde\lambda_{\rm nlp}} \quad \text{ where } \quad \tilde\lambda_{\rm nlp}= \frac{1}{2}\left(\sigma_1 -\sqrt{(\sigma_1-2\lambda_u)^2 + 4a}\right).
\end{equation}

Suppose case (a) holds, then $\hat{c}_{\rm nlp}=\bar{c}_{\rm nlp}$. Define $\overline{w}_2$ by
\begin{equation*}
\begin{split}
\overline w_2(t,x) =
\left\{
\begin{array}{ll}
\medskip
\lambda_u(x-(\lambda_u+\frac{1}{\lambda_u})t), & \text{for}~\frac{x}{t}\geq 2\lambda_u,\\
\medskip
\frac{t}{4}(\frac{x^2}{t^2}-4), &\text{for}~ \sigma_1\leq \frac{x}{t}<2\lambda_u,\\
\medskip
(\frac{\sigma_1}{2}-\sqrt{a})(x-\bar{c}_{\rm nlp}t), & \text{for}~\bar{c}_{\rm nlp}<\frac{x}{t}<\sigma_1,\\
0, &  \text{for}~\frac{x}{t}\leq \bar{c}_{\rm nlp}.
\end{array}
\right.
\end{split}
\end{equation*}
By construction, $\overline{w}_2$ is continuous in $[0,\infty)\times \mathbb{R}$. We claim that continuous $\overline{w}_2$ is a viscosity super-solution of \eqref{eq:sub-solutionw2}. (Actually, it is the unique viscosity solution of \eqref{eq:sub-solutionw2}, but we do not need this fact.) Indeed, $\overline{w}_2$ is a classical solution for \eqref{eq:sub-solutionw2} whenever $\frac{x}{t} \not \in \{\sigma_1, \bar{c}_{\rm nlp}\}$.  Now, it remains to consider the case when $\overline{w}_2 - \varphi$ attains a strict local minimum at $(t_0,x_0)$  for $\forall \varphi\in C^\infty(0,\infty)\times\mathbb{R})$, when $\frac{x_0}{t_0} = \sigma_1$ or $\bar{c}_{\rm nlp}$. In case $\frac{x_0}{t_0} = \sigma_1$, $(\overline{w}_2-\varphi)(t,\sigma_1 t) \geq (\overline{w}_2-\varphi)(t_0, x_0)$ for all $t\approx t_0$, so that $\partial_t \varphi(t_0,x_0) + \sigma_1\partial_x\varphi(t_0,x_0) = \frac{\sigma^2_1}{4} -1$. Hence, at $(t_0,x_0)$, (note that $(-a\chi_{\{x \leq  \sigma_1 t\}})^* =  - a\chi_{\{x < \sigma_1 t\}}$)
\begin{align*}
\partial_t\varphi +|\partial_x \varphi|^2 + 1 - a\chi_{\{x < \sigma_1 t\}}
&= \frac{\sigma_1^2}{4} - 1 - \sigma_1 \partial_x \varphi + |\partial_x \varphi|^2 + 1\\
&= \left( \partial_x \varphi - \frac{\sigma_1}{2}\right)^2 \geq 0.
\end{align*}
On the other hand, if $\frac{x_0}{t_0} = \bar{c}_{\rm nlp}$, then $\nabla \varphi(t_0,x_0) \cdot (1,\bar{c}_{\rm nlp}) = 0$, and
$$
 0 \leq \nabla \varphi(t_0,x_0) \cdot (-\bar{c}_{\rm nlp}t,1) \leq \nabla[({\frac{\sigma_1}{2}} - \sqrt{a})(x - \bar{c}_{\rm nlp}t)] \cdot(-\bar{c}_{\rm nlp},1),
$$
which means $\partial_t \varphi(t_0,x_0) = -\bar{c}_{\rm nlp}\partial_x \varphi$ and $0 \leq \partial_x \varphi(t_0,x_0) \leq  \frac{\sigma_1}{2} - \sqrt{a}$, whence
\begin{align*}
\partial_t\varphi +|\partial_x \varphi|^2 + 1 - a\chi_{\{x < \sigma_1 t\}}
&=  -\bar{c}_{\rm nlp} \partial_x \varphi + |\partial_x \varphi|^2 + 1-a\\
&= \left( \partial_x \varphi - \frac{1-a}{\frac{\sigma_1}{2}-\sqrt{a}}\right)\left( \partial_x \varphi - {\frac{\sigma_1}{2}+\sqrt{a}} \right) \geq 0
\end{align*}
at $(t_0,x_0)$. The last inequality holds because $\partial_x \varphi \leq \frac{\sigma_1}{2} - \sqrt{a} \leq \sqrt{1-a} \leq \frac{1-a}{\frac{\sigma_1}{2}-\sqrt{a}}$. Hence, $\overline{w}_2$ is a viscosity super-solution of \eqref{eq:sub-solutionw2}.

By Remark \ref{rmk:w1w20} and the express of $\overline{w}_2$, we have
$$
\overline{w}_2(0,x) = \lambda_u \max\{x,0\}=w_2^*(0,x) \,\quad \text{ for } x\in  \mathbb{R}.
$$
And recalling that $w_2^*$ is a viscosity sub-solution of \eqref{eq:sub-solutionw2}, we may deduce by Theorem \ref{thm:D} that
\begin{equation}\label{eq:ineqw2u*}
0 \leq w_2^*\leq \overline w_2 \quad\quad \text{ in } [0,\infty)\times\mathbb{R}.
\end{equation}
Now,
$$\{(t,x):w_2^*(t,x)=0\}\supset\{(t,x):\overline w_2(t,x)=0\}= \{(t,x): x\leq\hat c_{\rm nlp}t\}.
$$  Hence,
$$ w_2^\epsilon(t,x) = -\epsilon \log{u^\epsilon(t,x)} \to 0 \quad\text{ locally uniformly on } \{(t,x):~x< \hat c_{\rm nlp}t\}.$$
Hence for each small $\eta>0$, by choosing the compact sets $K=\{(1,x): \eta\leq  x \leq \hat c_{\rm nlp} - \eta \}$ and $K'=\{(1,x): \frac{\eta}{2} \leq x \leq \hat c_{\rm nlp} - \frac{\eta}{2} \}$, we may apply Lemma \ref{lem:underlinec}\rm{(a)} to get
$$\liminf_{t\to \infty}\inf_{\eta t\leq x\leq ( \hat c_{\rm nlp} - \eta)t} u(t,x)=\liminf_{\epsilon\to 0}\inf_{K} u^\epsilon (t,x)\geq \frac{1-a}{2}>0,$$
which implies $\underline c_2\geq \hat{c}_{\rm nlp}$.

Finally, for case (b), then we have $\hat{c}_{\rm nlp}=\tilde{c}_{\rm nlp}$. We define
\begin{equation*}
\begin{split}
\overline w_2 (t,x) =
\left\{
\begin{array}{ll}
\medskip
\lambda_u(x-(\lambda_u+\frac{1}{\lambda_u})t), & \text{for}~\frac{x}{t}\geq \sigma_1,\\
\medskip
\tilde \lambda_{\rm nlp} (x-\tilde {c}_{\rm nlp}t), &  \text{for}~\tilde{c}_{\rm nlp}< \frac{x}{t}<\sigma_1,\\
0, &   \text{for}~\frac{x}{t}\leq \tilde{c}_{\rm nlp}.
\end{array}
\right.
\end{split}
\end{equation*}
Then one can verify that $\overline{w}_2$ is likewise a viscosity super-solution of \eqref{eq:sub-solutionw2}, so that one can repeat the arguments for case (a) to show, again, that $\underline c_2\geq \hat{c}_{\rm nlp}$.
\end{proof}

\section{Estimating $\displaystyle \overline c_2$ from above and $\displaystyle \underline c_3$ from below}\label{S4}
We assume $\sigma_1>\sigma_2$ throughout this section. It remains to show $$\overline c_2\leq \max\{c_{\rm LLW}, \hat{c}_{\rm nlp}\} \quad \text{  and } \quad \underline c_3\geq -\max\{\tilde c_{\rm LLW}, \sigma_3\}.$$

\subsection{Estimating $\overline c_2$ from above}
For $\delta \geq 0$, we will construct an exponent $\hat\mu_\delta$ depending continuously on $\delta$  such that
$$
u(t, (\sigma_1-\delta) t)\leq \exp\left(  - (\hat \mu_\delta+ o(1))t\right) \quad \text{ for }t \gg 1,
$$
so that we may apply Lemma \ref{lem:appen1}\rm{(a)} to estimate $\overline{c}_2$ from above.


\begin{lemma}\label{lem:supsolutionw2}
 Let $( u, v)$ be a solution of \eqref{eq:1-1} with initial data satisfying $\rm{(H_\lambda)}$. Then
 $w_{2,*}$ is a viscosity  super-solution of
\begin{align}\label{eq:supsolutionw2}
\left\{
\begin{array}{ll}
\min\{\partial_tw+|\partial_xw|^2+1-a\chi_{\{\sigma_2 t<x<\sigma_1t\}},w\}=0,&\text{ in } (0,\infty)\times\mathbb{R},\\
w(0,x)=\lambda_u\max\{x, 0\}, & \text{ on } \mathbb{R},
\end{array}
\right.
\end{align}
where $\sigma_1$  and $\sigma_2$ are defined in \eqref{eq:sigma}.
\end{lemma}
\begin{proof}
It follows from standard arguments as in Lemma \ref{lem:sub-solutionw1}.
\end{proof}


\begin{proposition}\label{prop:overlinec2}
 Let $(u,v)$ be a solution of \eqref{eq:1-1} with initial data satisfying $\rm{(H_\lambda)}$. Then
$$\overline c_2\leq \max\{c_{\rm LLW}, \hat{c}_{\rm nlp}\},$$
\end{proposition}
where $c_{\rm LLW}$ and $\hat{c}_{\rm nlp}$ are defined respectively in Theorem \ref{thm:LLW} and \ref{thm:1-2}.
\begin{proof}
\noindent{\bf Step 1.}  Define $\underline{w}_2: [0,\infty) \times \mathbb{R}$ by
\begin{equation}\label{eq:underw2}
\begin{split}
\underline{w}_2(t,x) =
\left\{
\begin{array}{ll}
\medskip
\lambda_u(x-(\lambda_u+\frac{1}{\lambda_u})t), & \text{for}~\frac{x}{t}\geq 2\lambda_u,\\
\medskip
\frac{t}{4}(\frac{x^2}{t^2}-4), &\text{for}~ 2\leq \frac{x}{t}<2\lambda_u,\\
\medskip
0, & \text{for}~{\frac{x}{t}<2},
\end{array}
\right.
\end{split}
\end{equation}
in case $\lambda_u >1$, and by
\begin{equation*}
\underline{w}_2(t,x) =
\lambda_u \max\left\{x-(\lambda_u+\frac{1}{\lambda_u})t, 0\right\},
\end{equation*}
in case $\lambda _u\leq 1$.
Then it is straightforward to verify that $\underline{w}_2$ is a viscosity sub-solution of \eqref{eq:supsolutionw2}. Since, $w_{2,*}(0,x) = \lambda_u \max\{x,0\} = \underline{w}_2(0,x)$ in $\mathbb{R}$ (by Remark \ref{rmk:w1w20}), we may apply Theorem \ref{thm:D} to deduce
 \begin{equation}\label{eq:equivaw2}
 w_{2,*}(t, x)\geq \underline w_2(t,x) \quad \text{ for } [0,\infty)\times\mathbb{R}.
 \end{equation}

 \noindent{\bf Step 2.} To show that, for each $\hat{c}\geq 0$,
 \begin{equation}\label{eq:inequ}
 u(t,\hat{c}t)\leq \exp\{-(\underline w_2(1,\hat{c})+o(1))t\} \quad \text{ for } t\gg1.
 \end{equation}
 And that $\underline w_2(1,\sigma_1)$ is given by
\begin{equation}\label{eq:muc1}
\underline w_2(1,\sigma_1)=\left\{
\begin{array}{ll}
\medskip
(\frac{\sigma_1}{2}-\sqrt{a})(\sigma_1-\bar c_{\rm nlp}), &\text{ for } \sigma_1< 2\lambda_u,\\
\tilde\lambda_{\rm nlp}(\sigma_1-\tilde c_{\rm nlp}), &\text{ for } \sigma_1\geq 2\lambda_u,\\
\end{array}
\right.
\end{equation}
and $\bar{c}_{\rm nlp}, \,\tilde{c}_{\rm nlp},\,\tilde\lambda_{\rm nlp}$ are all given in Lemma \ref{lem:underlinec2}.

By definition of $w_{2,*}$ and $w_2^\epsilon(t,x)=-\epsilon \log{u^\epsilon(t,x)}$, for each small $\epsilon>0$, by applying Step 1, we have
\begin{equation*}
-\epsilon\log{u\left(\frac{1}{\epsilon},\frac{\hat c}{\epsilon}\right)}\geq w_{2,*}(1,\hat c)+o(1) \geq \underline w_2(1,\hat c)+o(1)
\end{equation*}
$$\Longleftrightarrow \,\, u\left(\frac{1}{\epsilon},\frac{\hat c}{\epsilon}\right) \leq \exp\left(-\frac{\underline w_2(1,\hat c) + o(1)}{\epsilon}\right),  
$$
which implies \eqref{eq:inequ}.
By the formula of $\underline w_2$, we can show
\begin{itemize}
\item[\rm{(i)}] For $\sigma_1<2\lambda_u$, we substitute $(t,x) = (1,\sigma_1)$ in \eqref{eq:underw2} to obtain
\begin{equation}\label{eq:4.3b}\underline w_2(1,\sigma_1)=\frac{1}{4}(\sigma_1^2-4)=(\frac{\sigma_1}{2}-\sqrt{a})(\sigma_1-\bar c_{\rm nlp}), 
\end{equation}
where $\bar c_{\rm nlp}=\frac{\sigma_1}{2}-\sqrt{a}+\frac{1-a}{\frac{\sigma_1}{2}-\sqrt{a}}$;
\item[\rm{(ii)}] For $\sigma_1\geq 2\lambda_u$,
we substitute $(t,x) = (1,\sigma_1)$ in \eqref{eq:underw2} to obtain
\begin{equation}\label{eq:4.3ccc}
\underline w_2(1,\sigma_1)=\lambda_u\left(\sigma_1-(\lambda_u+\frac{1}{\lambda_u})\right).
\end{equation}
Recalling the definition of $\tilde{\lambda}_{\rm nlp}$ in \eqref{eq:tlambdanlp}, we have
$$
\tilde{\lambda}_{\rm nlp} - \lambda_u = \frac{1}{2} \left[(\sigma_1 - 2\lambda_u) - \sqrt{\sigma_1 - 2\lambda_u) ^2 + 4a }\right],
$$ so that
\begin{equation}\label{eq:4.3cc}
(\tilde{\lambda}_{\rm nlp} - \lambda_u)^2 - (\sigma_1 - 2\lambda_u)(\tilde{\lambda}_{\rm nlp} - \lambda_u) - a=0.
\end{equation}

Hence, \eqref{eq:4.3ccc} becomes
\begin{equation}\label{eq:4.3c}
\underline w_2(1,\sigma_1)=\lambda_u\left(\sigma_1-(\lambda_u+\frac{1}{\lambda_u})\right)=\tilde\lambda_{\rm nlp}(\sigma_1-\tilde c_{\rm nlp}),
\end{equation}
where $\tilde c_{\rm nlp}, \tilde\lambda_{\rm nlp}$ are as in \eqref{eq:tlambdanlp}.
\end{itemize}
This implies \eqref{eq:muc1} holds, which completes Step 2.


\noindent{\bf Step 3.} To show $\overline c_2\leq \max\{c_{\rm LLW}, \hat{c}_{\rm nlp}\}$.

It follows from Proposition \ref{prop:1} and Corollary \ref{cor:underlinec1} that for $\hat c\in (\sigma_2,\sigma_1)$,
$$\lim\limits_{t\to\infty}(u,v)(t,0)=(k_1,k_2)\,\,\text{ and }\,\,\lim\limits_{t\to\infty}(u,v)(t, \hat c t)=(0,1).$$

By Step 2 and observation  $ \lambda_{\rm LLW}{c}_{\rm LLW}= \lambda_{\rm LLW}^2+1-a$, then we apply Lemma \ref{lem:appen1}\rm{(a)} in Appendix to conclude that for $\hat c\in (\sigma_2,\sigma_1)$,
\begin{equation}\label{estim}
 \overline c_2\leq c_{\hat{c},\underline w_2(1,\hat{c})} =\left\{
\begin{array}{ll}
\medskip
c_{\rm LLW},& \text{ if } \underline w_2(1,\hat{c})\geq -\lambda_{\rm LLW}^2+\lambda_{\rm LLW} \hat{c}-(1-a),\\
 \hat{c}-\frac{2{\underline w_2(1,\hat{c})}}{\hat{c}-\sqrt{\hat{c}^2-4({\underline w_2(1,\hat{c})}+1-a)}},&  \text{ if } \underline w_2(1,\hat{c})< -\lambda_{\rm LLW}^2+\lambda_{\rm LLW} \hat{c}-(1-a).\\
\end{array}
\right .
\end{equation}
Letting $\hat{c}\nearrow \sigma_1$, \eqref{estim} can be expressed as (denote $\hat\mu=\underline w_2(1,\sigma_1)$)
\begin{equation}\label{estimc2}
 \overline c_2\leq c_{\sigma_1,{\hat\mu}} =\left\{
\begin{array}{ll}
\medskip
c_{\rm LLW},& \text{ if } {\hat\mu}\geq -\lambda_{\rm LLW}^2+\lambda_{\rm LLW}\sigma_1-(1-a),\\
\sigma_1-\frac{2{\hat\mu}}{\sigma_1-\sqrt{\sigma_1^2-4({\hat\mu}+1-a)}},&  \text{ if } {\hat\mu}< -\lambda_{\rm LLW}^2+\lambda_{\rm LLW}\sigma_1-(1-a).\\
\end{array}
\right .
\end{equation}
It remains to verify $c_{\sigma_1,{\hat\mu}}=\max\{c_{\rm LLW},\hat c_{\rm nlp}\}$, where $\hat c_{\rm nlp}=\hat\lambda_{\rm nlp}+\frac{1-a}{\hat\lambda_{\rm nlp}}$ and
\begin{equation}\label{eq:lcacc}
\hat{\lambda}_{\rm nlp} = \begin{cases}
\frac{\sigma_1}{2}-\sqrt{a},& \text{ if }\sigma_1 < 2\lambda_u\,\, \text{ and }\,\, \sigma_1 \leq 2 (\sqrt a + \sqrt{1-a}),\\
\tilde\lambda_{\rm nlp},  & \text{ if }\sigma_1  \geq 2\lambda_u \,\,\text{ and }\,\, \tilde\lambda_{\rm nlp} \leq \sqrt{1-a},\\
\sqrt{1-a}, &\text{ otherwise,}
\end{cases}
\end{equation}
and $\tilde\lambda_{\rm nlp}$ is given in Lemma \ref{lem:underlinec2}. Note that
\begin{equation}\label{eq:lambdamuhat}
\hat{\lambda}_{\rm nlp} = \min\{\lambda_{\hat\mu}, \sqrt{1-a}\}, \quad \text{ where }\quad \lambda_{\hat\mu}:= \left\{
\begin{array}{ll}
\medskip
\frac{\sigma_1}{2}-\sqrt{a}, &\text{ for } \sigma_1< 2\lambda_u,\\
\tilde\lambda_{\rm nlp}, &\text{ for } \sigma_1\geq 2\lambda_u.\\
\end{array}
\right.
\end{equation}
By \eqref{eq:4.3b} and \eqref{eq:4.3c}, $\hat\mu = \underline{w}_2(1,\sigma_1)$ can be expressed as

\begin{equation}\label{eq:muc1'}
{\hat\mu}=G(\lambda_{\hat\mu}), \quad \text {where }\quad  G(\lambda): = -\lambda^2+\sigma_1\lambda-(1-a)
\end{equation}
and
$\lambda_{\hat\mu}$ is as defined in \eqref{eq:lambdamuhat}.
Note that $G(\lambda)$ is strictly increasing on $[0,\frac{\sigma_1}{2}]$.

We note for later purposes that \eqref{eq:muc1'} is a quadratic equation in $\lambda_{\hat\mu}$, so that
\begin{equation}\label{eq:lambdac1}
\lambda_{\hat\mu} = \frac{\sigma_1 - \sqrt{\sigma_1^2 - 4(\hat\mu + 1 - a)}}{2}.
\end{equation}
 Since $\lambda_{\rm LLW} \in (0, \sqrt{1-a}]$, we divide our discussion into two cases: (i) $\lambda_{\hat\mu} < \lambda_{\rm LLW}$; (ii) $\lambda_{\rm LLW} \leq \lambda_{\hat\mu}$. 
\begin{itemize}

\item[(i)] Case $\lambda_{\hat\mu} < \lambda_{\rm LLW}$.  (Recall that $\lambda_{\rm LLW} \leq \sqrt{1-a}$.)

By \eqref{eq:lambdamuhat}, $\hat\lambda_{\rm nlp} = \lambda_{\hat\mu}< \lambda_{\rm LLW}$, whence it follows from the observation
\begin{equation}\label{eq:oob}
\hat{c}_{\rm nlp} = \hat\lambda_{\rm nlp} + \frac{1-a}{\hat\lambda_{\rm nlp}}\,\, \text{ and }\,\, c_{\rm LLW}  = \lambda_{\rm LLW}  + \frac{1-a}{\lambda_{\rm LLW}},
\end{equation} and the monotonicity of $s + \frac{1-a}{s}$  in $(0,\sqrt{1-a}]$
that $\hat{c}_{\rm nlp} \geq c_{\rm LLW}$. It remains to show that $c_{\sigma_1,\hat\mu} = \hat{c}_{\rm nlp}$.

Now, by monotonicity of $G$, we have
$${\hat\mu}=G(\lambda_{\hat\mu})  <G(\lambda_{\rm{LLW}})=-\lambda_{ \rm LLW}^2+\lambda_{\rm LLW}\sigma_1-(1-a).$$
By \eqref{estimc2}, we have $c_{\sigma_1,{\hat\mu}}= \sigma_1-\frac{2{\hat\mu}}{\sigma_1-\sqrt{ \sigma_1^2-4({\hat\mu}+1-a)}}$. Hence,
$$c_{\sigma_1,{\hat\mu}}
=\sigma_1- \frac{{\hat\mu}}{\lambda_{\hat\mu}}=\lambda_{\hat\mu}+\frac{1-a}{\lambda_{\hat\mu}}=\hat{c}_{\rm nlp},$$
where the  first and second equalities follow from \eqref{eq:lambdac1} and  \eqref{eq:muc1'}, respectively.

\item[(ii)] Case $\lambda_{\rm LLW} \leq \lambda_{\hat\mu}$.

By \eqref{eq:lambdamuhat},
$$
\hat\lambda_{\rm nlp} = \min\{\lambda_{\hat\mu}, \sqrt{1-a}\} \geq \min\{\lambda_{\rm LLW}, \sqrt{1-a}\} =\lambda_{\rm LLW}.
$$ It follows from \eqref{eq:oob} that $\hat{c}_{\rm nlp} \leq c_{\rm LLW}$. It remains to show that $\hat\mu \geq G(\lambda_{\rm LLW})$, so that $c_{\sigma_1,\hat\mu} = c_{\rm LLW} = \max\{c_{\rm LLW}, \hat{c}_{\rm nlp}\}$. Indeed, one can check that $\lambda_{\rm LLW} \leq \lambda_{\hat\mu} \leq \sigma_1/2$, and we deduce
$$
\hat\mu = G(\lambda_{\hat\mu}) \geq G(\lambda_{\rm LLW}),
$$
by the monotonicity of $G$ in $[0,\sigma_1/2]$.

\end{itemize}

The proof of Proposition \ref{prop:overlinec2} is now complete.
\end{proof}

\subsection{Estimating $\underline{c}_3$ from below}
For convenience, let $\tilde u(t,x)=u(t,-x),\,\tilde v(t,x)=v(t,-x)$, and define
$$  \tilde u^\epsilon (t,x)=\tilde u\left(\frac{t}{\epsilon},\frac{x}{\epsilon}\right),\,\, \tilde v^\epsilon (t,x)=\tilde v\left(\frac{t}{\epsilon},\frac{x}{\epsilon}\right),\,\, w_3^\epsilon= -\epsilon \log{\tilde v^\epsilon (t,x)}\,\, \text{ in } [0,\infty)\times\mathbb{R}.$$
Again we pass to the half-relaxed limit:
$$w_{3,*}(t,x)=\liminf_{\scriptsize \begin{array}{c}\epsilon \to 0\\ (t',x') \to (t,x)\end{array}} \hspace{-.5cm} w_3^\epsilon (t',x').$$

\begin{lemma}\label{lem:tildeuv}
Let $(\tilde u,\tilde v)$ be a solution of \eqref{eq:1-1} such that  $x\rightarrow\left(\tilde u(0,-x),\tilde v(0,-x)\right)$ satisfies $\rm{(H_\lambda)}$. Then, for each small $\eta>0$,
\begin{equation}\label{eq:tildeuv}
\lim\limits_{t\to\infty} \sup_{x>(d\lambda^-_v+\frac{r}{\lambda^-_v}+\eta)t}(|\tilde u(t,x)-1|+|\tilde v(t,x)|)=0.
\end{equation}
\end{lemma}
\begin{proof}
Let $v_{\rm KPP}$ be the solution of
\begin{equation*}
\left\{
\begin{array}{ll}
\partial_t v_ {\rm KPP}- d\partial_{xx} v_{\rm KPP} =r v_{\rm KPP}(1-v_{\rm KPP}), &\text{ in }(0,\infty) \times \mathbb{R},\\
v_{\rm KPP} = \min\{1, Ce^{-\lambda_v^-x}\}, &\text{ on } x \in \mathbb{R}.
\end{array}
\right.
\end{equation*}
By choosing $C$ to be sufficiently large, we may apply comparison principle to get
 $0\leq \tilde v \leq v_{\rm KPP}$. Therefore, for each $\eta>0$, 
\begin{equation}\label{eq:tildevkpp}
\lim_{t \to \infty} \sup_{x > (d\lambda^-_v+\frac{r}{\lambda^-_v}+\eta)t} |\tilde v(t,x)| = 0 \, \text{ for each }\eta >0.
\end{equation}

Let $u_{\rm KPP}$ be the solution of
\begin{equation*}
\left\{
\begin{array}{ll}
\partial_t u_{\rm KPP}- \partial_{xx} u_{\rm KPP} = u_{\rm KPP}(1-a-u_{\rm KPP}), &\text{ in }(0,\infty) \times \mathbb{R},\\
u_{\rm KPP}(0,x) = u_0(x), &\text{ on } x \in \mathbb{R}.
\end{array}
\right.
\end{equation*}
Again the scalar comparison principle implies $u \geq u_{\rm KPP}$.
By the results in \cite{Kametaka_1976} or \cite{Mckean_1975}, we have, for each small $\eta>0$,
\begin{equation}\label{eq:tildeukpp}
\lim\limits_{t \to \infty} \inf\limits_{x >-(2\sqrt{1-a}+\eta)t} \tilde u(t,x)=\lim_{t \to \infty} \inf_{x <(2\sqrt{1-a}-\eta)t} u(t,x)\geq \frac{1-a}{2}.
\end{equation}

By small $\eta$>0, we have \eqref{eq:tildevkpp} and \eqref{eq:tildeukpp} hold, thus we may apply Lemma \ref{lem:entire1}\rm{(b)} to deduce \eqref{eq:tildeuv}.
\end{proof}

 In view of Lemma \ref{lem:tildeuv}, we obtain
\begin{equation}\label{eq:tildeu}
\chi_{\{x>(d\lambda^-_v+\frac{r}{\lambda^-_v})t\}}\leq \hspace{-.5cm}\liminf_{\scriptsize \begin{array}{c}\epsilon \to 0\\ (t',x') \to (t,x)\end{array}} \hspace{-.5cm} \tilde u^\epsilon(t',x')\leq \hspace{-.5cm}\limsup_{\scriptsize \begin{array}{c}\epsilon \to 0\\ (t',x') \to (t,x)\end{array}} \hspace{-.5cm} \tilde u^\epsilon(t',x')\leq 1 .
\end{equation}
\begin{lemma}\label{lem:sub-solutionw3}
Let $(\tilde u,\tilde v)$ be a solution of \eqref{eq:1-1} such that  $x\rightarrow\left(\tilde u(0,-x),\tilde v(0,-x)\right)$ satisfies $\rm{(H_\lambda)}$. Then, $w_{3,*}$ 
 is a viscosity  super-solution of
\begin{equation}\label{eq:supsolutionw3}
\left\{
\begin{array}{ll}
\min\{\partial_tw+d|\partial_xw|^2+r(1-b\chi_{\{x>(d\lambda^-_v+\frac{r}{\lambda^-_v})t\}}),w\}=0&\text{ in } (0,\infty)\times (0,\infty),\\
\medskip
w(0,x)=\lambda_v^-x, & \text{ on }[0,\infty),\\
w(t,0)=0, & \text{ for }~t>0.
\end{array}
\right.
\end{equation}
\end{lemma}
\begin{proof} The proof is similar to proof of Lemma \ref{lem:sub-solutionw1}(b) and is omitted.
\end{proof}

\begin{proposition}\label{prop:underlinec3}
Let $( u, v)$ be a solution of \eqref{eq:1-1} with initial data satisfying $\rm{(H_\lambda)}$. Then
$$\underline c_3\geq -\max\{\tilde c_{\rm LLW},\sigma_3\}.$$
where $\tilde c_{\rm LLW}$ and $\sigma_3$ are defined in Remark \ref{rmk:LLW} and \eqref{eq:sigma}, respectively.
\end{proposition}
\begin{proof}
\noindent{\bf Step 1.} To show
 \begin{equation}\label{eq:equivaw3}
 w_{3,*}(t, x)\geq \underline w_3(t,x) \quad \text{ for } [0,\infty)\times[0,\infty),
 \end{equation}
where $\underline w_3: [0,\infty)\times[0, \infty)$ is defined by
$$\underline w_3(t,x)=\lambda_v^-  \max\left\{x-(d\lambda_v^-+\frac{r}{\lambda_v^-})t, 0\right\}.$$
As in Step 1 of Proposition \ref{prop:overlinec2}, one can verify that $\underline{w}_3$ is a viscosity sub-solution of \eqref{eq:supsolutionw3}. By the expression of $\underline{w}_3$, Remark \ref{rmk:w1w20} and $w_{1,*}(t,-x)=w_{3,*}(t,x)$, we have
 $
\underline{w}_3(t,x) = \lambda_v^- \max\{x,0\}=w_{3,*}(t,x) \,\,\text{ on } \partial[ (0,\infty) \times (0,\infty)] $. Hence we apply Theorem \ref{thm:D} to obtain \eqref{eq:equivaw3}.

 \noindent{\bf Step 2.} To show for each $ \hat c \geq 0$, we have
 \begin{equation}\label{eq:ineqv}
 \tilde v(t,\hat c t)\leq \exp\{(\underline w_3(1,\hat c )+o(1))t\} \quad \text{ for } t\gg 1.
 \end{equation}
 This can be done as in Step 2 of Proposition \ref{prop:overlinec2}.

%

\noindent{\bf Step 3.} To show $\underline c_3\geq -\max\{\tilde c_{\rm LLW},\sigma_3\}$.

Fix $\hat c>(d\lambda^-_v+\frac{r}{\lambda^-_v})$.
By Proposition \ref{prop:1} and Lemma \ref{lem:tildeuv}, we arrive at
\begin{equation}\label{eq:asymuvc3}
\lim\limits_{t\to \infty}(\tilde{u},\tilde{v})(t,0)=\lim\limits_{t\to \infty}(u,v)(t,0)=(k_1,k_2) \text{ and }\lim\limits_{t\to \infty}(\tilde{u},\tilde{v})(t,\hat{c}t)=(1,0).
\end{equation}
This verifies condition {\rm{(i)}} of Lemma \ref{lem:appen1}\rm{(b)}. Next, by Step 2,
we have
$$\tilde{v}(t,\hat c t)\leq  \exp\{-(\hat\mu_2+ o(1)) t\} \quad \text{ for } t\gg 1,$$
where
\begin{equation}\label{eq:muc2}
{\hat\mu_2}=\underline w_3(1,\hat c)=\lambda_v^-(\hat c-(d\lambda_v^-+\frac{r}{\lambda_v^-})).
\end{equation}

We note for later purposes that  $\hat\mu_2$ is a quadratic expression in $\lambda_v^-$, so that
\begin{equation}\label{eq:lambdav-}
\hat\mu_2= \lambda_v^-\hat c-d(\lambda_v^-)^2-r, \quad \text{ and }\quad \lambda_v^- = \frac{\hat c- \sqrt{\hat c^2 - 4d(\hat\mu_2 + r)}}{2d}.
\end{equation}

We may then apply Lemma \ref{lem:appen1}\rm{(b)}  to conclude
\begin{equation}\label{eq:estimc3}
-\underline{c}_3 \leq 
\tilde{c}_{\hat{c},\hat\mu_2}= \left\{
\begin{array}{ll}
\medskip
\tilde c_{\rm LLW},& \text{ if } \hat\mu_2\geq \tilde{\lambda}_{\rm LLW} (\hat{c} - \tilde{c}_{\rm LLW}),\\
\hat c-\frac{2d\hat\mu_2}{\hat c-\sqrt{\hat c^2-4d[\hat\mu_2+r(1-b)]}},&  \text{ if } 0<\hat\mu_2<\tilde{\lambda}_{\rm LLW} (\hat{c} - \tilde{c}_{\rm LLW}).
\end{array}
\right.
\end{equation}
To complete the proof, we need to verify
$$\limsup_{\hat c\to\infty}\tilde c_{\hat c,{\hat\mu_2}}\leq\max\left\{\tilde{c}_{\rm LLW},\sigma_3\right\}.$$ 
Since $0  =-d\tilde{\lambda}^2_{\rm LLW}+\tilde{\lambda}_{\rm LLW}\tilde c_{\rm LLW}-r(1-b)$,  then
\begin{equation}\label{eq:mu2comp}
    \begin{split}
   \hat\mu_2-\tilde{\lambda}_{\rm LLW} (\hat c- \tilde{c}_{\rm LLW})
   =&\hat\mu_2-(-d\tilde{\lambda}^2_{\rm LLW}+\tilde{\lambda}_{\rm LLW}\hat c -r(1-b))\\
   =&\left (\lambda_v^--\tilde{\lambda}_{\rm LLW}\right)\left[\hat c- d(\lambda_v^-+\tilde{\lambda}_{\rm LLW})\right]-rb,
\end{split}
\end{equation}
where \eqref{eq:lambdav-} is used for the last inequality.

\begin{itemize}
\item [\rm{(i)}] For the case $\lambda_v^-> \tilde{\lambda}_{\rm LLW}$, we take $\hat c\to\infty$ in \eqref{eq:mu2comp} to get
  $$\hat\mu_2\geq \tilde{\lambda}_{\rm LLW} \left(\hat{c} - \tilde{c}_{\rm LLW}\right),$$
   so that by \eqref{eq:estimc3}, $-\underline{c}_3 \leq \tilde{c}_{\rm LLW}\leq\max\left\{\tilde{c}_{\rm LLW},\sigma_3\right\}$;

  \item [\rm{(ii)}] For the case $\lambda_v^-\leq\tilde{\lambda}_{\rm LLW}$, we have $\lambda_v^-\leq\tilde{\lambda}_{\rm LLW}\leq \sqrt{\frac{r(1-b)}{d}}$ and
      $$\sigma_3=d\lambda_{v}^-+\frac{r(1-b)}{\lambda_{v}^-}\geq d\tilde{\lambda}_{\rm LLW}+\frac{r(1-b)}{\tilde{\lambda}_{\rm LLW}}=\tilde{c}_{\rm LLW}.$$ 
      we have
 $$0<\hat\mu_2<\tilde{\lambda}_{\rm LLW} \left(\hat{c} - \tilde{c}_{\rm LLW}\right).$$
  Denote
$\lambda_{\hat{c},\hat\mu_2}=\frac{\hat c-\sqrt{\hat c^2-4d[\hat\mu_2+r(1-b)]}}{2d}$. Then
\begin{equation}\label{eq:ppp}
d(\lambda_{\hat{c},\hat\mu_2})^2-\hat c \lambda_{\hat{c},\hat\mu_2}+\hat\mu_2+r(1-b)=0,
\end{equation}
and $\lambda_{\hat{c},\hat\mu_2}\leq\lambda^-_v$ (by comparing with the second part of \eqref{eq:lambdav-}).
Hence, we arrive at
\begin{equation}\label{eq:crewrite}
-\underline{c}_3\leq c_{\hat{c},\hat\mu_2}= \hat{c} - \frac{\hat\mu_{2}}{\lambda_{\hat{c}, \hat{\mu}_2}} = d\lambda_{\hat{c},\hat\mu_2}+\frac{r(1-b)}{\lambda_{\hat{c},\hat\mu_2}}.
\end{equation}

Next, we claim that
\begin{equation}\label{eq:pppp}
\lim_{\hat c\to\infty}\lambda_{\hat{c},\hat\mu_2}=\lambda^-_v.
\end{equation}
 To this end,
subtract the first part of \eqref{eq:lambdav-} from \eqref{eq:ppp} to get
$$d(\lambda_{\hat{c},\hat\mu_2})^2-\hat c (\lambda_{\hat{c},\hat\mu_2}-\lambda^-_v)-d(\lambda^-_v)^2-rb=0.$$
Dividing the above by $\hat{c}$ and letting $\hat{c} \to \infty$, we obtain \eqref{eq:pppp}.

By \eqref{eq:pppp}, we can take $\hat{c} \to \infty$ in \eqref{eq:crewrite} to get $-\underline{c}_3\leq \sigma_3\leq\max\left\{\tilde{c}_{\rm LLW},\sigma_3\right\}$.

\end{itemize}
This completes the proof of Proposition \ref{prop:underlinec3}.
\end{proof}

\subsection{Proof of  Theorem \ref{thm:1-2}}
\begin{proof}[Proof of Theorem \ref{thm:1-2}]
For $i=1,2,3$, let  $\overline{c}_i, \underline{c}_i$ be the maximal and minimal spreading speeds defined in \eqref{eq:speeds}.
It follows from definition directly that $\overline{c}_i \geq \underline{c}_i$.
By Corollary \ref{cor:underlinec1}, we have $\overline{c}_1 = \underline{c}_1= \sigma_1$. By Proposition  \ref{prop:1}(ii) and Lemma \ref{lem:underlinec2}, we arrive at  $\underline{c}_2 \geq \max\{c_{\rm LLW}, \hat{c}_{\rm nlp}\}$, which, together with $\overline{c}_2 \leq \max\{c_{\rm LLW}, \hat{c}_{\rm nlp}\}$  in Proposition \ref{prop:overlinec2}, we have $\overline{c}_2 = \underline{c}_2 = \max\{c_{\rm LLW}, \hat{c}_{\rm nlp}\}$. Moreover, combining with Propositions \ref{prop:1} and \ref{prop:underlinec3} gives $\overline{c}_3  = \underline{c}_3 =- \max\{\sigma_3, \tilde{c}_{\rm LLW}\}$.
Recalling the $c_i$ as defined in \eqref{eq:c123}, we have $\overline{c}_i = \underline{c}_i = c_i$ for all $i=1,2,3$. To complete the proof of Theorem \ref{thm:1-2}, it remains to show \eqref{eq:spreadingly}.

Observe that the first two items of \eqref{eq:spreadingly} is a direct consequence of  Corollary \ref{cor:underlinec1}. Next, we shall show that
\begin{equation}\label{positivev}
    \displaystyle \liminf_{t \to \infty} \inf_{(c_3 + \eta)t < x < (\sigma_1 - \eta)t} v(t,x) >0 \quad \text{ for  small } \eta>0.
\end{equation}

Given some small $\eta>0$, definitions of $\overline{c}_3$ and $\underline{c}_1$ imply the existence of $c_3' \in (c_3, c_3 +\eta)$, $\sigma_1' \in (\sigma_1 - \eta, \sigma_1)$ and $T>0$ such that
$$
\inf_{t \geq T}\min\{v(t, {c}_3' t), v(t, \sigma'_1 t)\} >0.
$$
Now, define
\begin{equation*}
    \begin{array}{l}
     \delta:= \min\left\{\frac{1-b}{2},\, \inf\limits_{c_3' T < x < \sigma'_1 T}v(T,x),\,\inf\limits_{t \geq T}\min\{v(t, {c}_3' t), v(t, \sigma'_1 t)\} \right\} >0.
     \end{array}
\end{equation*}
Observe that $v(t,x)$ and $\delta$ form a pair of super- and sub-solutions to the KPP-type equation $\partial_t v = d\partial_{xx}v + rv(1-b - v)$ such that $v(t,x) \geq \delta$ on the parabolic boundary of the domain $\Omega:=\{(t,x): t \geq T,\,c_3' t < x < \sigma'_1 t \}$. It follows from the maximum principle that $v \geq \delta$ in $\Omega$. In particular, \eqref{positivev} holds.

Similarly, we can show that
\begin{equation}\label{positiveu}
    \displaystyle \liminf_{t \to \infty} \inf_{x <(c_2-\eta)t} u(t,x) >0\quad \text{ for small }\eta>0,
\end{equation}
by definition of $\underline c_2$ and by \eqref{eq:tildeukpp} in Lemma \ref{lem:tildeuv}.
Fix small $\eta>0$.
In view of \eqref{positivev} and \eqref{positiveu}, the third item of \eqref{eq:spreadingly} holds by applying {\rm{(a)}} and {\rm{(c)}} in Lemma \ref{lem:entire1}.  Finally, since  \eqref{positiveu} and $\lim\limits_{t\to \infty } \sup\limits_{x<(c_3-\eta)t} v=0$ (by definition of $\underline c_3$), by applying Lemma \ref{lem:entire1}{\rm{(b)}}, the fourth
item of \eqref{eq:spreadingly} holds true. The proof of Theorem \ref{thm:1-2} is now complete.
\end{proof}

%
\section{The invasion mode due to Tang and Fife}\label{S4b}
In this section, we assume $\sigma_1=\sigma_2$ and prove Theorem \ref{thm:1-2b}.  

\begin{proof}[Proof of Theorem \ref{thm:1-2b}] 
For any small $\delta\in(0,1)$, let $(\underline{u}^\delta,\overline{v}^\delta)$ and  $(\overline{u}^\delta,\underline{v}^\delta)$ be respectively any  solution of
\begin{equation}\label{eq:1-1'}
\left\{
\begin{array}{ll}
\medskip
\partial_t u-\partial_{xx}u=u(1-u-av),& \text{ in }(0,\infty)\times \mathbb{R},\\
\partial _t v-d\partial_{xx}v=rv(1+\delta-bu-v),& \text{ in }(0,\infty)\times \mathbb{R},\\
\end{array}
\right .
\end{equation}
and
\begin{equation}\label{eq:1-1''}
\left\{
\begin{array}{ll}
\medskip
\partial_t u-\partial_{xx}u=u(1-u-av),& \text{ in }(0,\infty)\times \mathbb{R},\\
\partial _t v-d\partial_{xx}v=r v(1-\delta-bu-v),& \text{ in }(0,\infty)\times \mathbb{R},\\
\end{array}
\right .
\end{equation}
with initial data satisfying $\rm(H_\lambda)$. By comparison, we deduce that
\begin{equation}\label{eq:uvsupersub}
(\underline{u}^\delta,\overline{v}^\delta)\preceq (u,v)\preceq (\overline{u}^\delta,\underline{v}^\delta) \,\, \text{ in } [0,\infty)\times\mathbb{R}.
\end{equation}

Notice that $(\underline{u}^\delta,\overline{v}^\delta)$ is a solution of \eqref{eq:1-1'} if and only if
\begin{equation}\label{scaling1}
  (\underline{U}^\delta,\overline V^\delta)= \left(\underline{u},\frac{\overline{v}^\delta}{1+\delta}\right)
\end{equation}
 is a solution of
 \begin{equation}\label{eq:1-1'trs}
\left\{
\begin{array}{ll}
\medskip
\partial_t U-\partial_{xx}U=U(1-U-\overline{a}^\delta V),& \text{ in }(0,\infty)\times \mathbb{R},\\
\partial _t V-d\partial_{xx}V=\overline{r}^\delta V(1-\underline{b}^\delta U-V),& \text{ in }(0,\infty)\times \mathbb{R},\\
\end{array}
\right .
\end{equation}
where $\overline{a}^\delta=(1+\delta)a,\, \overline{r}^\delta=(1+\delta)r$ and $ \underline{b}^\delta=\frac{b}{1+\delta}$.
Observe that $\overline\sigma_1^\delta=d(\lambda_v^+\wedge \sqrt{\frac{ \overline{r}^\delta}{d}})+\frac{ \overline{r}^\delta}{\lambda_v^+\wedge\sqrt{\frac{ \overline{r}^\delta}{d}}}>\sigma_1=\sigma_2$ and  $0<\overline{a}^\delta,\underline{b}^\delta<1$ by choosing $\delta$ small enough. By applying Theorem \ref{thm:1-2} to \eqref{eq:1-1'trs}, we deduce that the rightward and leftward spreading speeds $\overline{c}^\delta_1$  and $\underline{c}^\delta_3$ of $\overline{V}^\delta$ (which is the same as $\overline{v}^\delta$), and the rightward spreading speed $\underline{c}^\delta_2$ of $\underline{U}^\delta$ (same as $\underline{u}^\delta$) exist. Furthermore, they can be characterized by
$$\overline{c}_1^\delta = \overline\sigma_1^\delta, \quad \underline{c}_2^\delta = \max\{\underline{c}_{\textup{LLW}}^\delta, \underline{\hat{c}}_{\textup{nlp}}^\delta\},\quad \underline{c}_3^\delta = -\max\{\overline{\tilde{c}}_{\textup{LLW}}^\delta,\overline{\sigma}_3^\delta\}.
$$
Precisely,
$\underline{c}_{\textup{LLW}}^\delta$ 
$(\text{resp.}~\overline{\tilde{c}}_{\textup{LLW}}^\delta)$ 
is the spreading speed for \eqref{eq:1-1'trs}  
as given in Theorem \ref{thm:LLW} $($resp. Remark \ref{rmk:LLW}$)$,  $\overline\sigma_3^\delta=d(\lambda_v^-\wedge \sqrt{\frac{ \overline{r}^\delta}{d}})+\frac{ \overline{r}^\delta(1-\underline{b}^\delta)}{\lambda_v^+\wedge\sqrt{\frac{ \overline{r}^\delta}{d}}}$ and moreover
\begin{equation}\label{eq:hcaccdelta}
\underline{\hat{c}}_{{\rm{nlp}}} ^\delta= \begin{cases}
\frac{\overline\sigma_1^\delta}{2} - \sqrt{\overline{a}^\delta} + \frac{1-\overline{a}^\delta}{\frac{\overline\sigma_1^\delta}{2} - \sqrt{\overline{a}^\delta}}, & \text{ if }\overline\sigma_1^\delta < 2\lambda_u\,\, \text{ and }\,\, \overline\sigma_1^\delta \leq 2 (\sqrt{\overline{a}^\delta} + \sqrt{1-\overline{a}^\delta}),\\
\overline{\tilde{\lambda}}_{{\rm{nlp}}}^\delta + \frac{1-\overline{a}^\delta}{\overline{\tilde{\lambda}}_{{\rm{nlp}}}^\delta}, & \text{ if }\overline\sigma_1^\delta  \geq 2\lambda_u \,\,\text{ and }\,\, \overline{\tilde{\lambda}}_{{\rm{nlp}}}^\delta \leq \sqrt{1-\overline{a}^\delta},\\
2\sqrt{1-\overline{a}^\delta}, &\text{ otherwise,}
\end{cases}
\end{equation}
where
$
\overline{\tilde{\lambda}}_{{\rm{nlp}}}^\delta= \frac{1}{2}\left[\overline\sigma_1^\delta- \sqrt{(\overline\sigma_1^\delta-2\lambda_u)^2 + 4\overline{a}^\delta}\right].
$
Now, by the relation \eqref{eq:uvsupersub}, we can compare with the spreading speeds $\overline{c}_1$, $\underline{c}_2$ and $\underline{c}_3$ of $(u,v)$:
\begin{equation}\label{inequality1}
  \overline{c}_1\leq \overline{c}_{1}^\delta,\quad \underline{c}_2\geq \underline{c}_2^\delta\quad \text{ and } \quad \underline{c}_3\geq \underline{c}_{3}^\delta.
\end{equation}
It remains to show that, assuming $\sigma_1 = \sigma_2$, we have $\underline{\hat{c}}_{{\rm{nlp}}}^\delta\to \sigma_2$ as $\delta\to 0$. Divide into the two cases:
\begin{itemize}
\item[\rm{(i)}] If $\lambda_u>1$, then $\sigma_1=\sigma_2=2<2\lambda_u$. Since $1<\sqrt{a}+\sqrt{1-a}$, by choosing  $\delta>0$ small enough, we get  $\overline\sigma_1^\delta<2\lambda_u$ and $\overline\sigma_1^\delta \leq 2 (\sqrt{\overline{a}^\delta} + \sqrt{1-\overline{a}^\delta})$, which implies $\underline{\hat{c}}_{{\rm{nlp}}}^\delta=\frac{\overline\sigma_1^\delta}{2} - \sqrt{\overline{a}^\delta} + \frac{1-\overline{a}^\delta}{\frac{\overline\sigma_1^\delta}{2} - \sqrt{\overline{a}^\delta}} \to 1-\sqrt{a} + \frac{1-a}{1-\sqrt{a}} = 2 = \sigma_2$ as $\delta \to 0$.

\item[\rm{(ii)}] If $\lambda_u\leq 1$, then first  claim that
\begin{equation}\label{eq:firstclaimm}
\overline\sigma_1^\delta \geq \sigma_1 \geq 2\lambda_u,
\end{equation}
which is due to $\overline\sigma_1^\delta\geq\sigma_1=\sigma_2=\lambda_u+\frac{1}{\lambda_u}\geq 2\geq 2\lambda_u.$

Next, we claim that
\begin{equation}\label{eq:firstclaim}
\tilde\lambda_{\rm nlp} < \sqrt{1-a}, 
\end{equation}
where $\tilde\lambda_{\rm nlp}$ is given in \eqref{eq:lambdaacc}. To this end, observe that
\begin{equation}\label{eq:sss3}
\sigma_1-2\sqrt{1-a}<\sqrt{(\sigma_1-2\lambda_u)^2 + 4{a}}
\end{equation}
which is a consequence of
\begin{align*}
\begin{split}
(\sigma_1-2\sqrt{1-a})^2-\left[(\sigma_1-2\lambda_u)^2 + 4{a}\right]
&=4(2-2a-\sigma_1\sqrt{1-a})\\
&\leq4(2-2a-2\sqrt{1-a})< 0.
\end{split}
\end{align*}
From definition of $\tilde\lambda_{\rm nlp} = \frac{1}{2}[\sigma_1 - \sqrt{(\sigma_1 - 2\lambda_u)^2 + 4a}]$, we deduce \eqref{eq:firstclaim}.

By \eqref{eq:firstclaimm} and \eqref{eq:firstclaim}, we have $\bar\sigma_1^\delta \geq 2\lambda_u$ and $\overline{\tilde\lambda}^\delta_{\rm nlp} < \sqrt{1-\overline{a}^\delta}$ for $\delta$ small,  so
$$
\underline{\hat{c}}^\delta_{\rm nlp} = \overline{\tilde{\lambda}}_{{\rm{nlp}}}^\delta + \frac{1-\overline{a}^\delta}{\overline{\tilde{\lambda}}_{{\rm{nlp}}}^\delta}  \to  \tilde\lambda_{\rm nlp} + \frac{1-a}{\tilde\lambda_{\rm nlp}} \quad \text{ as } \delta \to 0.
$$
Since we want $\underline{\hat{c}}^\delta_{\rm nlp} \to \sigma_2$, it remains to show that $\sigma_2 =  \tilde\lambda_{\rm nlp} + \frac{1-a}{\tilde\lambda_{\rm nlp}}$. To this end, recall, from the definition of $\tilde\lambda_{\rm nlp}$ \eqref{eq:lambdaacc}, that 
$$
\tilde\lambda_{\rm nlp} = \frac{\sigma_1 - \sqrt{(\sigma_1 - 2\lambda_u)^2 + 4a}}{2} = \frac{2(\sigma_1 \lambda_u - \lambda_u^2 - a)}{\sigma_1 + \sqrt{(\sigma_1 - 2\lambda_u)^2 + 4a}}.
$$
Using $\sigma_1= \sigma_2 = \lambda_u + \frac{1}{\lambda_u}$, we deduce
\begin{equation}\label{eq:sss2}
\tilde\lambda_{\rm nlp} =\frac{\sigma_2 - \sqrt{(\sigma_2 - 2\lambda_u)^2 + 4a}}{2} = \frac{2(1-a)}{\sigma_2+  \sqrt{(\sigma_2 - 2\lambda_u)^2 + 4a}}.
\end{equation}
This implies $\sigma_2 =  \tilde\lambda_{\rm nlp} + \frac{1-a}{\tilde\lambda_{\rm nlp}}$. The proof is now complete.
\end{itemize}

Hence, by the continuity of  $\underline{c}_{\textup{LLW}}^\delta$ and $\overline{\tilde{c}}_{\textup{LLW}}^\delta$ in $\delta$ (see, e.g. \cite[Theorem 4.2 of Ch. 3]{Volpert_1994}), letting $\delta\to 0$ in \eqref{inequality1} yields
\begin{equation}\label{eq:upperci}
\overline{c}_1\leq \sigma_1,\quad \underline{c}_2\geq \sigma_2 \quad \text{ and }\quad \underline{c}_3\geq - \max\{ \tilde{c}_{\rm LLW}, \sigma_3\}.
\end{equation}

By a quite similar process, we can obtain
$(\overline{u}^\delta,\underline{v}^\delta)$ is a solution of \eqref{eq:1-1''} if and only if
 $$(\overline{U}^\delta,\underline V^\delta)=\left(\overline{u}^\delta,\frac{\underline{v}^\delta}{1-\delta}\right)$$
 is a solution of
 \begin{equation}\label{eq:1-1''trs}
\left\{
\begin{array}{ll}
\medskip
\partial_t U-\partial_{xx}U=U(1-U-\underline{a}^\delta V),& \text{ in }(0,\infty)\times \mathbb{R},\\
\partial _t V-d\partial_{xx}V=\underline{r}^\delta V(1-\overline{b}^\delta U-V),& \text{ in }(0,\infty)\times \mathbb{R},
\end{array}
\right .
\end{equation}
where $\underline{a}^\delta=(1-\delta)a$, $ \underline{r}^\delta=(1-\delta)r$ and $\overline{b}^\delta=\frac{b}{1-\delta}$. Observe that $\underline\sigma_1^\delta=d(\lambda_v^+\wedge \sqrt{\frac{ \underline{r}^\delta}{d}})+\frac{ \underline{r}^\delta}{\lambda_v^+\wedge\sqrt{\frac{ \underline{r}^\delta}{d}}}<\sigma_1=\sigma_2$ and  $0<\underline{a}^\delta,\overline{b}^\delta<1$ for small $\delta$.  By exchanging the roles of $u$ and $v$ in \eqref{eq:1-1}, we may follow the arguments above, and
apply Theorem \ref{thm:1-2} once again to deduce that
\begin{equation}\label{eq:lowerci}
\underline{c}_1\geq \sigma_1, \quad \overline{c}_2\leq \sigma_2 \quad\text{ and}  \quad \overline{c}_3\leq - \max\{ \tilde{c}_{\rm LLW}, \sigma_3\}.
\end{equation}
Theorem \ref{thm:1-2b} follows from combining $\underline{c}_i\leq \overline{c}_i$, \eqref{eq:upperci}, \eqref{eq:lowerci} and $\sigma_1=\sigma_2$.
\end{proof}

\section{The case $0<a<1<b$ due to Girardin and Lam}\label{S5}

The Hamilton-Jacobi approach, which we have so far applied to study the weak competition case ($0<a,b<1$), can also be applied to tackle the case ($0<a<1<b$), which was previously studied by Girardin and the third author \cite{Girardin_2018}. This provides an alternative approach which is more transparent than the involved construction of global super- and sub-solutions for the Cauchy problem, as was done in \cite{Girardin_2018}.
%
By arguing similarly as in Theorem \ref{thm:1-2}, one can prove the following result.
\begin{theorem}\label{thm:1-3'}
Assume $0<a < 1 < b$ and $\sigma_1>\sigma_2$. Let $(u,v)$ be the solution of \eqref{eq:1-1} such that the initial data satisfies
$\mathrm{(H_\lambda)}$. 
Then there exist $c_1, c_2\in (0,\infty)$ such that $c_1 > c_2$ and,
for each small $\eta>0$, the following spreading results hold:
\begin{equation}\label{spreadingpro}
\begin{cases}
\lim\limits_{t\rightarrow \infty} \sup\limits_{ x>(c_{1}+\eta) t} (|u(t,x)|+|v(t,x)|)=0, \\
\lim\limits_{t\rightarrow \infty} \sup\limits_{(c_2+\eta) t< x<(c_{1}-\eta) t} (|u(t,x)|+|v(t,x)-1|)=0, \\
\lim\limits_{t\rightarrow \infty} \sup\limits_{ x<(c_2-\eta) t} (|u(t,x)-1|+|v(t,x)|)=0. \\
\end{cases}
\end{equation}
Precisely, the spreading speeds $c_1$ and $c_2$ can be determined as follows:
$$
c_1 = \sigma_1 , \quad c_2 = \max\{\hat{c}_{\rm LLW}, \hat{c}_{\rm nlp}\},$$
where $\sigma_1$ is defined in \eqref{eq:sigma}, $\hat{c}_{\rm LLW}$ denotes the minimal wave speed of \eqref{eq:1-1} connecting $(1,0)$ with $(0,1)$ and $\hat{c}_{\rm nlp}$ is given by
\begin{equation}\label{eq:hcacc'}
\hat{c}_{{\rm{nlp}}} = \begin{cases}
\frac{\sigma_1}{2} - \sqrt{a} + \frac{1-a}{\frac{\sigma_1}{2} - \sqrt{a}}, & \text{ if }\sigma_1 < 2\lambda_u\,\, \text{ and }\,\, \sigma_1 \leq 2 (\sqrt a + \sqrt{1-a}),\\
\tilde\lambda_{{\rm{nlp}}} + \frac{1-a}{\tilde\lambda_{{\rm{nlp}}}}, & \text{ if }\sigma_1  \geq 2\lambda_u \,\,\text{ and }\,\, \tilde\lambda_{{\rm{nlp}}} \leq \sqrt{1-a},\\
2\sqrt{1-a}, &\text{ otherwise,}
\end{cases}
\end{equation}
with
\begin{equation}\label{eq:lambdaacc'}
\tilde\lambda_{{\rm{nlp}}} = \frac{1}{2}\left[\sigma_1 - \sqrt{(\sigma_1-2\lambda_u)^2 + 4a}\right].
\end{equation}
\end{theorem}

By Theorem \ref{thm:1-3'}, the spreading speed $c_2$ is determined by  $\sigma_1$ (i.e., $c_1$) and $\lambda_u$.  In what follows, we explore the relation of $c_2$ and $\sigma_1$ for fixed $\lambda_u$.  Define the following auxiliary functions:
$$f(\sigma_1)=\frac{\sigma_1}{2} - \sqrt{a} + \frac{1-a}{\frac{\sigma_1}{2} - \sqrt{a}}, \quad g(\sigma_1)=\tilde\lambda_{{\rm{nlp}}} + \frac{1-a}{\tilde\lambda_{{\rm{nlp}}}},$$
where $\tilde\lambda_{{\rm{nlp}}}$ is given by \eqref{eq:lambdaacc'}. It is easily seen that
 $f$ is decreasing and bijective in $[2\sqrt{1-a}, 2(\sqrt{1-a}+\sqrt{a})]$, while $g$ is decreasing and bijective in
 \begin{equation*}
    \begin{cases}
      \medskip
      \left[2\sqrt{1-a},\lambda_u+\sqrt{1-a}+\frac{1-a}{\lambda_u-\sqrt{1-a}}\right] &\text{ as } \lambda_u\geq\sqrt{1-a}, \\
     \left(\lambda_u+\sqrt{1-a}+\frac{1-a}{\lambda_u-\sqrt{1-a}},\infty\right) & \text{ as } \lambda_u< \sqrt{1-a}.
    \end{cases}
 \end{equation*}
More precisely, it follows that
 \begin{equation*}
  \begin{cases}
  \medskip
f^{-1}(c_2)=c_2-\sqrt{c_2^2-4(1-a)}+2\sqrt{a},\\
g^{-1}(c_2)=\lambda_u+\frac{c_2-\sqrt{c_2^2-4(1-a)}}{2}+\frac{a}{\lambda_u-\frac{c_2-\sqrt{c_2^2-4(1-a)}}{2}}.
\end{cases}
\end{equation*}

 In view of $\tilde\lambda_{{\rm{nlp}}}\to\lambda_u$ as $\sigma_1\to\infty$, 
  $g_{\infty}:=g(\infty)=\lambda_u+\frac{1-a}{\lambda_u}$. For fixed $\lambda_u$ and varied $\lambda_v^+$ (or $\sigma_1$), by Theorem \ref{thm:1-3'} we can rewrite the spreading speed $c_2$  as follows.
\begin{itemize}
 \item[{\rm (a)}] For $g_{\infty}\leq\hat c_{\rm LLW}$, we have the followings:
 \begin{itemize}
  \item[{\rm (a1)}] If $\lambda_u\geq (\sqrt{a}+\sqrt{1-a})$, then
  \begin{equation*}
c_2 = \begin{cases}
f(\sigma_1), & \text{ for }\max\{2\sqrt{dr},\sigma_2\}\leq \sigma_1\leq f^{-1}(\hat c_{\rm LLW}),\\
\hat c_{\rm LLW},& \text{ for }\sigma_1> f^{-1}(\hat c_{\rm LLW}),
\end{cases}
\end{equation*}
  independent $\lambda_u$;
  \item[{\rm (a2)}] If $\sqrt{dr}\leq \lambda_u< \sqrt{a}+\sqrt{1-a}$ and $g^{-1}(\hat c_{\rm LLW} )>2\lambda_u$, then
    \begin{equation*}
c_2 = \begin{cases}
f(\sigma_1), & \text{ for }\max\{2\sqrt{dr},\sigma_2\}\leq\sigma_1< 2\lambda_u,\\
g(\sigma_1), & \text{ for }2\lambda_u\leq \sigma_1<g^{-1}(\hat c_{\rm LLW}),\\
\hat c_{\rm LLW},& \text{ for }\sigma_1\geq g^{-1}(\hat c_{\rm LLW});
\end{cases}
\end{equation*}
 \item[{\rm (a3)}]If 
 $\lambda_u<\sqrt{dr}$, then
    \begin{equation*}
c_2 = \begin{cases}
g(\sigma_1), & \text{ for }\max\{2\sqrt{dr},\sigma_2\}\leq \sigma_1<g^{-1}(\hat c_{\rm LLW}),\\
\hat c_{\rm LLW},& \text{ for }\sigma_1\geq g^{-1}(\hat c_{\rm LLW});
\end{cases}
\end{equation*}
\end{itemize}
\item[{\rm (b)}] For $g_{\infty}>\hat c_{\rm LLW}$, we have the followings:
 \begin{itemize}
  \item[{\rm (b1)}] If $\lambda_u\geq (\sqrt{a}+\sqrt{1-a})$, then
  \begin{equation*}
c_2 = \begin{cases}
f(\sigma_1), & \text{ for }\max\{2\sqrt{dr},\sigma_2\}\leq \sigma_1\leq f^{-1}(\hat c_{\rm LLW}),\\
\hat c_{\rm LLW},& \text{ for }\sigma_1> f^{-1}(\hat c_{\rm LLW}),
\end{cases}
\end{equation*}
  independent $\lambda_u$;
  \item[{\rm (b2)}] If $\sqrt{dr}\leq \lambda_u< \sqrt{a}+\sqrt{1-a}$, then
    \begin{equation*}
c_2 = \begin{cases}
f(\sigma_1), & \text{ for }\max\{2\sqrt{dr},\sigma_2\}\leq\sigma_1< 2\lambda_u,\\
g(\sigma_1), & \text{ for }\sigma_1\geq2\lambda_u;
\end{cases}
\end{equation*}
 \item[{\rm (b3)}]If  $\lambda_u<\sqrt{dr}$, then
    \begin{equation*}
c_2 = g(\sigma_1) \quad \text{ for }\sigma_1\geq\max\{2\sqrt{dr},\sigma_2\}.
\end{equation*}
\end{itemize}
\end{itemize}

For the case {\rm (a)} $g_{\infty}\leq\hat c_{\rm LLW}$, the relationship between the spreading speeds $\sigma_1$  and $c_2$ given by {\rm (a1)}-{\rm (a3)} is illustrated in Figure \ref{figure2}. 
 Therein we may obtain the exact spreading speeds of \eqref{eq:1-1}, which are determined entirely by $\lambda_u,\,\lambda_v^+\in(0,\infty)$. Traversing all of $\lambda_u$,  the set of admissible speeds $\sigma_1$ and $c_2$ agrees with  \cite[Figure 1.1]{Girardin_2018}. Particularly, a direct consequence of Theorem \ref{thm:1-3'} is the following proposition, which improves upon \cite[Theorem 1.3]{Girardin_2018} by clarifying the role of exponential decay $(\lambda_u,\,\lambda_v^+)$ of the initial data.

\begin{proposition}\label{overlinec}
Let $(\overline{c}, \underline{c}) \in(2\sqrt{dr},\infty) \times (\hat{c}_{\rm LLW}, \infty)$ such that $\overline{c} > \underline{c}$.
\begin{itemize}
\item[\rm (a)] If $\underline{c}<f(\overline{c})$, then the pair of spreading speeds $(\overline{c}, \underline{c})$ is not realized by solutions of \eqref{eq:1-1} with initial data satisfying $\mathrm{(H_\lambda)}$.
\item[\rm (b)] If $\underline{c} = f(\overline{c})$, then there exists a unique $\lambda_v^+ = \frac{1}{2d}(\overline{c} - \sqrt{\overline{c}^2 - 4dr})$ such that for $\lambda_u \in [ \overline{c}/2, \infty)$, the pair of spreading speeds $(\overline{c}, \underline{c})$ can be realized by solutions of \eqref{eq:1-1} with initial data satisfying $\mathrm{(H_\lambda)}$.

\item[\rm (c)] If $\underline{c} >f(\overline{c})$, then there exists a unique pair $(\lambda^+_v, \lambda_u)$ such that the pair of spreading speeds $(\overline{c}, \underline{c})$ can be realized by solutions of \eqref{eq:1-1} with initial data satisfying $\mathrm{(H_\lambda)}$.
\end{itemize}
%
\end{proposition}
\begin{proof}
Assertion (a) follows directly from \cite[Theorem 1.2]{Girardin_2018}. For assertion (b),  $\underline{c} > \hat{c}_{\rm LLW} \geq 2\sqrt{1-a}$, so that we have $\overline{c} \leq 2(\sqrt{a} + \sqrt{1-a})$. Hence it
 follows directly from \eqref{eq:hcacc'}. It remains to show (c).

First, we define $\lambda_v^+=\frac{\overline{c}-\sqrt{\overline{c}^2-4dr}}{2d}\in(0,\sqrt{\frac{r}{d}})$ such that $\overline{c}=\sigma_1=d\lambda_v^++\frac{r}{\lambda_v^+}$. Since $\sigma_1$ is strictly monotone in $(0,\sqrt{\frac{r}{d}})$, the choice of such $\lambda_v^+$ is unique.
Then we shall determine $\lambda_u$ such that $c_2=\underline{c}=\tilde{\lambda}_{\rm{nlp}}+\frac{1-a}{\tilde{\lambda}_{\rm{nlp}}}.$

  Since $\underline{c}>f(\overline{c})\geq 2\sqrt{1-a}$ and $\underline{c}>\hat{c}_{\rm{LLW}}$,  to satisfy $c_2=\underline{c}$, by \eqref{eq:hcacc'} we must have $\lambda_u\in(\frac{\overline{c}-\sqrt{\overline{c}^2-4a}}{2}, \frac{\overline{c}}{2}) $ and
\begin{equation}\label{condition}
\underline c=g(\overline{c}) \,\,\text{ and }\,\, \tilde\lambda_{\rm nlp} = \frac{1}{2}\left[\overline{c} - \sqrt{(\overline{c}-2\lambda_u)^2 + 4a}\right]<\sqrt{1-a}.
\end{equation}
Hence, it suffices to choose the unique $\lambda_u\in(\frac{\overline{c}-\sqrt{\overline{c}^2-4a}}{2}, \frac{\overline{c}}{2}) $ such that \eqref{condition} holds.
\begin{itemize}
  \item [{\rm(i)}] If $\overline{c}\leq2(\sqrt{1-a}+\sqrt{a})$, then observe that when $\lambda_u\in(\frac{\overline{c}-\sqrt{\overline{c}^2-4a}}{2}, \frac{\overline{c}}{2}) $,
$
\tilde\lambda_{\rm nlp} \in \left(0,\overline{c}/{2}-\sqrt{a}\right)
$
is increasing in $\lambda_u$,  so that
\begin{equation*}
    g(\overline{c})=\tilde\lambda_{\rm nlp}+\frac{1-a}{\tilde\lambda_{\rm nlp}}\in \left(f(\overline{c}),\overline c \right),
\end{equation*}
is decreasing in $\lambda_u$. Noting that $\underline{c}\in\left(f(\overline{c}),\overline c\right)$, we may select the unique $\lambda_u\in(\frac{\overline{c}-\sqrt{\overline{c}^2-4a}}{2}, \frac{\overline{c}}{2})$ such that \eqref{condition} holds;
  \item [{\rm(ii)}] If $\overline{c}>2(\sqrt{1-a}+\sqrt{a})$, then to satisfy $\tilde\lambda_{\rm nlp} <\sqrt{1-a}$ in \eqref{condition}, it is necessary that
  $\lambda_u\in(\frac{\overline{c}-\sqrt{\overline{c}^2-4a}}{2}, \frac{\overline{c}-\sqrt{(\overline{c}-2\sqrt{1-a})^2-4a}}{2})$.
  In this case,
  \begin{equation*}
   \tilde\lambda_{\rm nlp} \in \left(0,\sqrt{1-a}\right) \,\text{ and thus }\, g(\overline{c})=\tilde\lambda_{\rm nlp}+\frac{1-a}{\tilde\lambda_{\rm nlp}}\in \left(2\sqrt{1-a},\overline c\right),
\end{equation*}
are also strictly monotone in $\lambda_u$, so that there is the  unique $\lambda_u$ such that \eqref{condition} holds.
\end{itemize}
The proof is now complete.
\end{proof}

\begin{figure}[http!!]
  \centering
\includegraphics[height=5.2in]{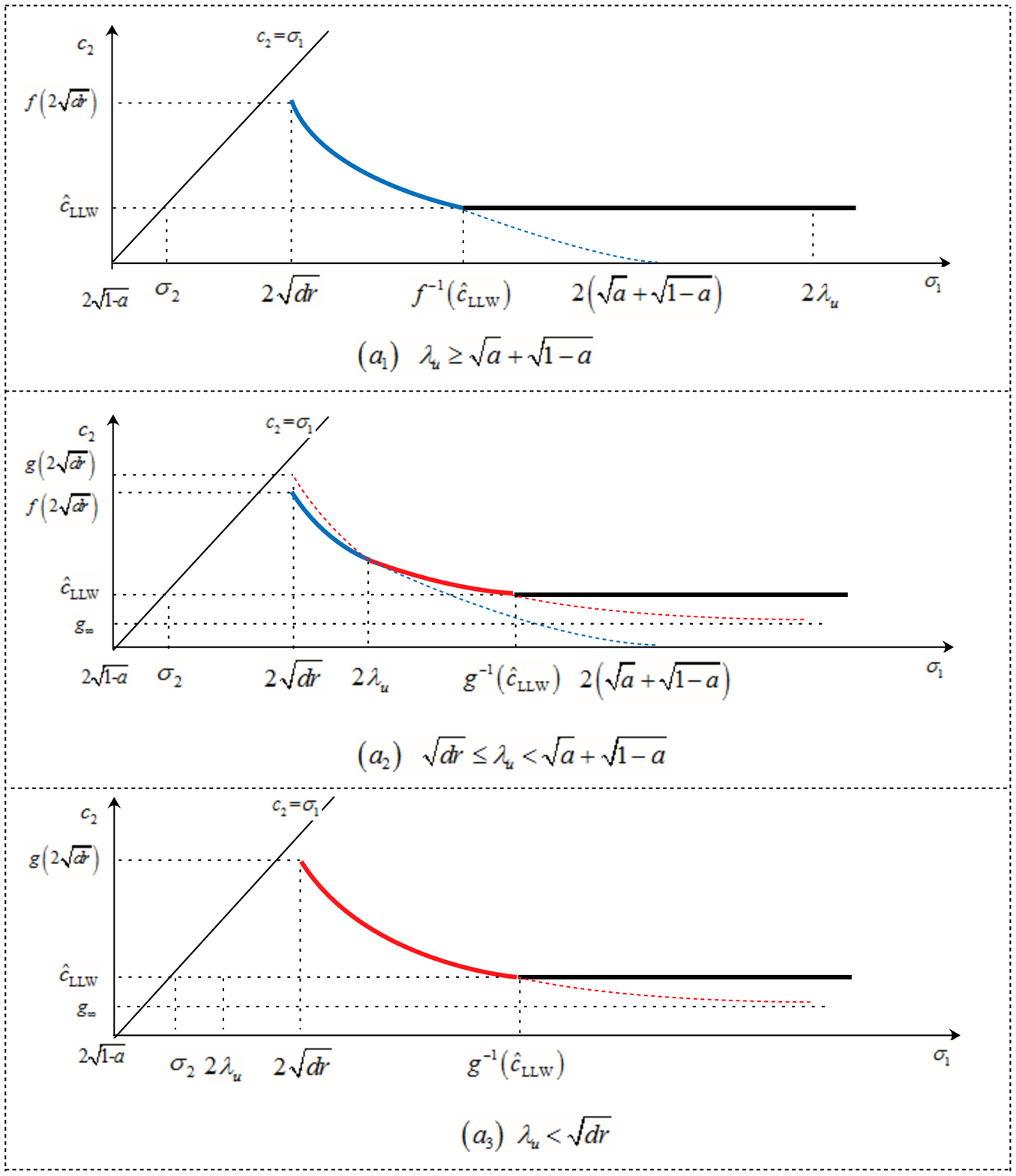}
  \caption{ The profile of $c_2(\sigma_1)$ for case {\rm (a)} $g_{\infty}\leq\hat c_{\rm LLW}$, which is expressed by the solid line with the blue one representing $f$ and the red one representing $g$.
                    }\label{figure2}
  \end{figure}

\section{An extension}\label{S6}
In this section, we consider the following competition system with forcing:
\begin{equation}\label{eq:extend}
\left\{
\begin{array}{ll}
\partial_t u-\partial_{xx}u=u(1-u-av-h(t,x)),& \text{ in }(0,\infty)\times \mathbb{R},\\
\partial _t v-d\partial_{xx}v=r v(1-bu-v-k(t,x)),& \text{ in }(0,\infty)\times \mathbb{R},\\
u(0,x)=u_0(x), & \text { on }  \mathbb{R},\\
v(0,x)=v_0(x), & \text { on }  \mathbb{R},
\end{array}
\right .
\end{equation}
where
 \begin{equation}\label{eq:hk}
\lim_{t\to\infty}\sup_{x\geq c_0 t}(|h(t,x)|+|k(t,x)|)=0 \quad \text{ for some }c_0 \in \mathbb{R}.
\end{equation}
We will make an observation in preparation for our forthcoming work on three-species competition systems.  Recall the definitions of $\sigma_i$ ($i=1,2,3$) from \eqref{eq:sigma}.

\begin{theorem}
Let $d,r,b>0$, $0<a<1$ and $\sigma_1>\sigma_2$.
Suppose that $h(x,t),k(x,t)$ are non-negative and satisfy \eqref{eq:hk}. Let $(u,v)$ be the solution of \eqref{eq:extend} with the initial data satisfying
 $\mathrm{(H_\lambda)}$. Assume
 $$c_0<\sigma_2'
\,\quad \text{ where } \sigma_2'=(\lambda_u\wedge \sqrt{1-a})+\frac{1-a}{\lambda_u\wedge \sqrt{1-a}}.$$
 Then,
 $$\underline{c}_1=\overline{c}_1=\sigma_1,~~\overline{c}_2\leq\max\{c_{\rm LLW}, \hat{c}_{{\rm{nlp}}}\},~~ \underline{c}_2\geq \hat{c}_{{\rm{nlp}}}, $$
 where $\underline{c}_i$, $\overline{c}_i$ ($i=1,2$) are defined in \eqref{eq:speeds}.
 Furthermore, for each small $\eta>0$,
\begin{equation}\label{eq:spreadingly'}
\begin{cases}
\lim\limits_{t\rightarrow \infty} \sup\limits_{ x>(\sigma_{1}+\eta) t} (|u(t,x)|+|v(t,x)|)=0, \\
\lim\limits_{t\rightarrow \infty} \sup\limits_{(\overline{c}_2+\eta) t< x<(\sigma_{1}-\eta) t} (|u(t,x)|+|v(t,x)-1|)=0, \\
\lim\limits_{t\rightarrow \infty} \inf\limits_{(c_0+\eta) t< x<(\underline{c}_2-\eta) t} u(t,x) >0,
 \end{cases}
\end{equation}
where $\sigma_1$, $\sigma_2$ are defined in \eqref{eq:sigma} and $c_{\rm LLW}$, $\hat{c}_{{\rm{nlp}}}$ are respectively 
given in Theorem \ref{thm:LLW} and \ref{thm:1-2}.
\end{theorem}
\begin{proof}
The proof can be mimicked after that of Theorem \ref{thm:1-2}. 

\noindent{\bf Step 1.} The estimates $\overline{c}_1\leq\sigma_1$ and $\overline c_2\leq\sigma_2$ can be proved by rather similar  arguments as in Proposition \ref{prop:1}, and the details are omitted here.

\noindent{\bf Step 2.} We show that  for each small  $\eta>0$,
 \begin{equation}\label{assump1}
   \liminf_{t\to \infty} u(t,(\sigma'_2-\eta)t)>0\,\,\text{ and }\,\, \liminf_{t\to \infty} v(t,(\sigma_1-\eta)t)>0.
 \end{equation}


Here, we just show the first one since the proof of the second one is analogous. For the case of $\lambda_u\geq \sqrt{1-a}$, by \eqref{eq:hk} and  $c_0<\sigma_2'$, the system \eqref{eq:extend} is approximately equal to \eqref{eq:1-1} in $\{(t,x):x\geq \frac{c_0+\sigma_2'}{2}t, t\geq T\}$ for sufficient large $T$, so that we can deduce \eqref{assump1} by applying the similar arguments in Steps 4 of the proof of  \cite[Proposition 2.1]{LLL2019}. It remains to consider the case of $\lambda_u<\sqrt{1-a}$.

Fix any $c'\in (\max\{\frac{c_0+\sigma_2'}{2},2\sqrt{1-a}\},\sigma_2')$. It is enough to show that there exist positive constants $\delta, \tilde\lambda_1, \tilde\lambda_2, T$ such that  $\tilde\lambda_1< \tilde\lambda_2$ and
\begin{equation}\label{eq:compareB}
u(t, x + c't) \geq  \frac{\delta}{4}\max\left\{ \left[e^{-\tilde{\lambda}_1x}-e^{-\tilde{\lambda}_2x}\right],0\right\} \quad \text{ for }t \geq T, x \geq 0.
\end{equation}
This implies $\underline{c}_2 \geq c'$ for each $c'\in (\max\{\frac{c_0+\sigma_2'}{2},2\sqrt{1-a}\},\sigma_2')$, i.e., $\underline{c}_2 \geq \sigma_2'$.

To this end,
choose $\delta_1>0$ small enough so that
\begin{equation}\label{eq:delta_small}
\tilde\lambda_1:= \frac{1}{2} \left[ c' - \sqrt{(c')^2 - 4(1-a - 2\delta_1)}\right] > \lambda_u. 
\end{equation}
This is possible since $c' < \sigma_2'$ and that $s \mapsto \frac{s - \sqrt{s^2 - 4(1-a)}}{2}$ is monotone, so that
$$
 \frac{1}{2} \left[ c' - \sqrt{(c')^2 - 4(1-a)}\right] >  \frac{1}{2} \left[\sigma_2' - \sqrt{(\sigma_2')^2 - 4(1-a)}\right] =\lambda_u.
$$
Next, choose $T>0$ large so that 
\begin{equation}\label{eq:chooseT}
|h(t,x)|\leq \delta\, \quad \text{ for } \,t\geq T,\,\, x\geq {c}'t,
\end{equation}
and then choose $\delta \in (0,\delta_1]$ so that
\begin{equation}\label{eq:chooseD}
 u(T,x) \geq \frac{\delta}{4} e^{-\tilde{\lambda}_1(x- c' T)}  \quad \text{ for } \, x\geq {c}'T,
\end{equation}
where \eqref{eq:chooseT} follows from \eqref{eq:hk} by noting  ${c}'>c_0$; and that \eqref{eq:chooseD} holds due to $u(T,x)\sim e^{-\lambda_u x} $ at $\infty$ and $\tilde \lambda_1>\lambda_u$ (see, e.g. \cite[Corollary 1 of Ch. 1]{Volpert_1994}). By the choice of $\tilde\lambda_1 < \tilde\lambda_2$, $\delta, \delta_1, T$, it follows that 
\begin{equation}\label{eq:underlineu}
 \underline{u}(t,x):=\max\left\{\frac{\delta}{4} \left[e^{-\tilde{\lambda}_1(x-c't)}-e^{-\tilde{\lambda}_2(x-c't)}\right],0\right\},
\end{equation}
is a sub-solution of the KPP-type equation
\begin{equation}\label{eq:vlkpp}
\partial_t u = \partial_{xx} u+ r u(1-a -h(x,t)- u)\,\quad \text{ in }\Omega,
\end{equation}
where $\Omega:=\{ (t,x): t \geq T,\,\,x \geq c't\}$.

For $\delta \in (0,\delta_1]$ to be specified later, define
\begin{equation}\label{eq:underlineu}
 \underline{u}(t,x):=\max\left\{\frac{\delta}{4} \left[e^{-\tilde{\lambda}_1(x-c't)}-e^{-\tilde{\lambda}_2(x-c't)}\right],0\right\},
\end{equation}
where $\tilde\lambda_1$ is given in \eqref{eq:delta_small} and $\tilde\lambda_2 = \frac{1}{2} \left[ c' + \sqrt{(c')^2 - 4(1-a - 2\delta)}\right]$. We will choose $T>0$ and $\delta \in (0,\delta_1]$ so that
\begin{equation*}
\left\{
\begin{array}{ll}
\partial_t \underline {u} - \partial_{xx} \underline {u}-\underline {u}(1-a-h(x,t)-\underline {u}) \leq -\underline{u} \left( 2\delta - h(x,t) - \underline{u}\right) \leq 0   & \hspace{-.2cm}\text{ in }\Omega,\\
u(t, c' t) \geq  0 = \underline{u}(t, c't) &\hspace{-.2cm}\text{ for }t \geq T,\\
u(T, x) \geq  \frac{\delta}{4} e^{-\tilde{\lambda}_1(x- c' T)} \geq \underline{u}(T, x) &\hspace{-.2cm}\text{ for }x \geq c'T,\\
 \end{array}\right.
\end{equation*}
i.e., $u$ and $\underline{u}$ is a pair of super- and sub-solutions of the KPP-type equation
\begin{equation}\label{eq:vlkpp}
\partial_t u = \partial_{xx} u+ r u(1-a -h(x,t)- u)\,\quad \text{ in }\Omega,
\end{equation}
where $\Omega:=\{ (t,x): t \geq T,\,\,x \geq c't\}$. Hence, by comparison, \eqref{eq:compareB} holds.

To proceed further, as in Section \ref{S3}, based on the scaling \eqref{scaling}, 
we introduce the WKB ansatz $w_2^\epsilon$, which is given by
\begin{equation*}
 w_2^\epsilon(t,x)=-\epsilon\log{u^\epsilon(t,x)},
\end{equation*}
 satisfying the equation:
 \begin{equation*}
   \begin{split}
\begin{cases}
\partial_tw_2^\epsilon-\epsilon\partial_{xx} w_2^\epsilon+| \partial_xw_2^\epsilon|^2+1-u^\epsilon-av^\epsilon-h^\epsilon=0, & \text{in } (0,\infty)\times\mathbb{R},\\
w_2^\epsilon(0,x)=-\epsilon\log{u^\epsilon(0,x)},  & \mathrm{on} ~\mathbb{R}.
\end{cases}
   \end{split}
 \end{equation*}
Here $h^\epsilon(t,x)=h(\frac{t}{\epsilon},\frac{x}{\epsilon})$. 
By Remark \ref{rmk:w1w20}, we also use the half-relaxed limit method and introduce $w_2^*$  and $w_{2,*}$. By \eqref{assump1},
$$\liminf_{\epsilon\to 0} u^\epsilon (t,(\sigma'_2-\eta)t)>0,$$
and $u^\epsilon$ is moreover bounded by $1$.
We have then, by definitions, that
\begin{equation}\label{wstar12}
w_2^*(t,(\sigma'_2-\eta)t)=w_{2,*}(t,(\sigma'_2-\eta)t)=0.
\end{equation}

\noindent{\bf Step 3.} We prove $\underline{c}_1\geq \sigma_1$. 

This follows from  \eqref{assump1} and definition of $\underline{c}_1$.

 \noindent{\bf Step 4.}  We prove $\underline{c}_2\geq \hat{c}_{{\rm{nlp}}}$.

By Step 1 and  $h\geq 0$, we have 
\begin{equation*}
 0\leq \limsup_{\substack{(t',x')\to (t,x)\\ \epsilon\to 0}} v^\epsilon(t',x')\leq \chi_{\{x \leq \sigma_1 t\}}.
\end{equation*}
In view of $\sigma_2'>c_0$, we choose $0<\eta\ll1$ such that $\sigma'_2-\eta>c_0$. We then use \eqref{eq:hk} to  derive that
\begin{equation*}
  \begin{array}{l}
     \lim\limits_{\epsilon \to 0}\sup\limits_{x\geq (\sigma'_2-\eta)t} h^{\epsilon}(t,x)= 0.
   \end{array}
\end{equation*}
Based on \eqref{wstar12}, similar to Lemma \ref{lem:sub-solutionw1}, we can deduce that 
 $w_2^*$  is a viscosity sub-solution of
\begin{align*}
\begin{cases}
\min\{\partial_t w+|\partial_xw|^2+1-a\chi_{\{x\leq \sigma_1 t\}} ,w\}=0,& \text{for}\,\, x> (\sigma_2' -\eta)t,\\
w(0,x)=\lambda_ux, &\text{for}\,\,x\geq 0,\\
w(t,(\sigma'_2-\eta)t)=0, &\text{for}\,\,t\geq 0.
\end{cases}
\end{align*}

We then apply the same arguments developed in Lemmas \ref{lem:underlinec2} by constructing the same super-solutions, to deduce that $\underline{c}_2\geq \hat{c}_{{\rm{nlp}}}$. 

 \noindent{\bf Step 5.} We show $\overline{c}_2\leq\max\{c_{\rm LLW}, \hat{c}_{{\rm{nlp}}}\}$ and \eqref{eq:spreadingly'}.

 By \eqref{eq:hk} again,  similar to Corollary \ref{cor:underlinec1}, we can get
$$ \liminf_{\substack{(t',x')\to (t,x)\\ \epsilon\to 0}} v^\epsilon(t',x') \geq \chi_{\{\sigma_2 t<x<\sigma_1t\}},$$
 so that we may use \eqref{wstar12} to deduce that $w_{2,*}$ is  a viscosity  super-solution of
 \begin{align*}
\begin{cases}
\min\{\partial_tw+|\partial_xw|^2+1-a\chi_{\{\sigma_2 t<x<\sigma_1t\}},w\}=0&\text{for}\,\, x> (\sigma_2'-\eta)t,\\
w(0,x)=\lambda_ux, & \text{for}\,\,x\geq 0,\\
w(t,(\sigma'_2-\eta)t)=0, &\text{for}\,\,t\geq 0,
\end{cases}
\end{align*}
as in Lemma \ref{lem:supsolutionw2}.
 Then we can get $\overline{c}_2\leq\max\{c_{\rm LLW}, \hat{c}_{{\rm{nlp}}}\}$ by the same arguments  developed in  Proposition \ref{prop:overlinec2}. We finally  deduce \eqref{eq:spreadingly'} by  similar arguments as in the proof of Theorem \ref{thm:1-2}, which completes the proof.
 \end{proof}

\section*{Acknowledgement}
QL (201706360310) and SL (201806360223) would like to thank the China Scholarship Council for financial support during the period of their overseas study and  express their gratitude to the Department of Mathematics, The Ohio State University for the warm hospitality. SL is partially supported by the Outstanding Innovative Talents Cultivation Funded Programs 2018 of Renmin University of China. KYL
is partially supported by NSF grant DMS-1853561.

\appendixtitleon
\appendixtitletocon
\begin{appendices}

\section{Comparison principle}\label{SD}

This section is devoted to the proof of a comparison lemma for Hamilton-Jacobi equation for discontinuous super and sub-solutions and for piecewise Lipschitz continuous Hamiltonian. Our proof is inspired by the arguments developed by Ishii \cite{Ishii_1997} and Tourin \cite{Tourin_1992} (see also \cite{Alvarez_1997, Chen_2008,Giga_2010}). Ishii used a crucial observation of \cite{Alvarez_1997} to prove the comparison principle for discontinuous super- and sub-soultion  of Hamilton-Jacobi equations with nonconvex but continuous Hamiltonian, whereas Tourin gave the uniqueness of continuous solution of Hamilton-Jacobi equations with piecewise Lipschitz continuous Hamiltonian.  The uniqueness of viscosity solution for nonlinear first-order partial differential equations was first introduced by Crandall and Lions in \cite{Crandall_1983}, then Crandall, Ishii and Lions \cite{Crandall_1987} gave a simpler proof. Ishii \cite{Ishii_1985} study the discontinuous Hamiltonian with time measure and Tourin and Ostrov\cite{Ostrov_2002} studied the piecewise Lipschitz continuous, convex Hamiltonian, based on the dynamic programming principle.

Let $\Omega$ be a smooth domain in $(0,T] \times \mathbb{R}^N$, which is allowed to be unbounded or even equal to $(0,T]\times\mathbb{R}^N$.  We assume without loss that $T = \sup\{t>0:~(t,x) \in \Omega\}$, and define the parabolic boundary of $\Omega$ as
$$
\partial_p \Omega = \{(t,x) \in \partial\Omega:~ t \in [0,T)\}.
$$
Consider the following  Hamilton-Jacobi equation:
\begin{equation}\label{eq:D1}
\min\{\partial_t w + H(t,x, \partial_x w),w - Lt \} = 0 \quad \text{ in }\Omega.
\end{equation}
Let $H^*$ and $H_*$ be, respectively, the upper semicontinuous (usc) and lower semicontinuous (lsc) envelope of $H$ with respect to its first two variables. Precisely,
$$H^*(t,x,p)=\limsup_{(t',x')\to (t,x)}H(t',x',p)\quad \text{ and } \quad H_*(t,x,p)=\liminf_{(t',x')\to (t,x)}H(t',x',p).$$

We say that a  lower semicontinuous (lsc) function $w$ is  a  viscosity super-solution of \eqref{eq:D1} if $w-Lt\geq 0$ in $ \Omega$, and for all test functions $\varphi\in C^\infty(\Omega)$, if $(t_0,x_0)\in \Omega$ is a strict local minimum point of $w-\varphi$, then
 $$\partial_t\varphi(t_0,x_0)+H^*(t_0,x_0,\partial_x\varphi(t_0,x_0))\geq 0$$
  holds; A upper semicontinuous (usc) function $w$ is a  viscosity sub-solution of \eqref{eq:D1} if for all test functions $\varphi\in C^\infty(\Omega)$,  if $(t_0,x_0)\in \Omega$  is a strict local maximum point of $w-\varphi$ such that $w(t_0,x_0) - Lt_0>0$, then
  $$\partial_t\varphi(t_0,x_0)+H_*(t_0,x_0,\partial_x\varphi(t_0,x_0))\leq 0$$
holds. Finally, $w$ is a  viscosity solution  of \eqref{eq:D1} if and only if $w$ is simultaneously a viscosity super-solution and a viscosity sub-solution of \eqref{eq:D1}.


We impose additional assumptions on the domain $\Omega$ and the Hamiltonian $H:\Omega\times\mathbb{R}^N\rightarrow \mathbb{R}$. Namely,
there exists a closed set $\Gamma \subset [0,T]\times \mathbb{R}^N$ and, for each $R>0$, a continuous function $\omega_R: [0,\infty) \to [0,\infty)$ such that $\omega_R(0) =0$ and $\omega_R(r)>0$ for $r>0$, such that the following holds:
\begin{description}
\item[\rm{(A1)}] $H\in C((\Omega \setminus \Gamma)\times \mathbb{R}^N)$;

\item[\rm{(A2)}] For each $(t_0,x_0) \in (\Omega \setminus \Gamma) \cap ((0,T) \times B_R(0))$, there exist a constant $\delta_0=\delta_0(R)>0$ such that
$$
H(t,y,p) - H(t,x,p) \leq \omega_R\left( |x-y|( 1 + |p|)\right)
$$
for $t,x,y,p$ such that $\|(t,x) - (t_0,x_0)\| +\|(t,y) - (t_0,x_0)\| < \delta_0$ and $p \in \mathbb{R}^N$; 

\item[\rm{(A3)}] For each $(t_0,x_0) \in \Omega \cap \Gamma \cap ((0,T) \times B_R(0))$,  there exist a constant $\delta_0=\delta_0(R)>0$ and  a unit vector $(h_0, k_0) \in \mathbb{R} \times \mathbb{R}^N$ such that
$$
H^*(s,y,p) - H_*(t,x,p)\leq \omega_R\left( (|t-s| + |x-y|)(1 + |p|) \right)
$$
for all $p \in \mathbb{R}^N$ and $s, t, y, x$ satisfying
$$
\left\{
\begin{array}{ll}
\|(t,x) - (t_0,x_0)\|+ \|(s,y) - (t_0,x_0)\| < \delta_0,  \\
{\small \left\| \frac{(t-s,x-y)}{\|(t-s, x-y)\|} - (h_0,k_0)\right\| < \delta_0};
\end{array}\right.
$$

\item[\rm{(A4)}] There exists some $M \geq 0$ such that for each $\lambda \in [0,1)$ and $x_0 \in \mathbb{R}^N$, there exists constants $\bar\epsilon(\lambda,x_0)>0$ and $\bar C(\lambda,x_0) >0$ such that for all $p \in \mathbb{R}^N$, $(t,x) \in \Omega$, if $\epsilon \in [0, \bar\epsilon(\lambda,x_0)]$, then
$$
H\left(t,x,\lambda p - \frac{\epsilon (x-x_0)}{|x-x_0|^2 + 1}\right) - M \leq \lambda\left( H(t,x,p) - M\right) + \epsilon \bar C(\lambda,x_0).
$$

\end{description}

\begin{theorem}\label{thm:D}
Suppose that $H$ satisfies the hypotheses $\rm{(A1)}$-$\rm{(A4)}$. Let $\overline{w}$ and $\underline{w}$ be a pair of super- and sub-solutions of \eqref{eq:D1}
such that $\overline{w} \geq \underline{w}$ on $\partial_p \Omega$, then
$$
\overline{w} \geq \underline{w} \quad \text{ in } \Omega.
$$
\end{theorem}

\begin{remark}
Let
$$
H(t,x,p) = \textbf{H}(p) + R(x/t),
$$
where $\textbf{H}$ is convex and coercive in $p$, and $s\mapsto R(s)$ has bounded variation and satisfies $|R(s)| \leq M$ for some $M\geq 0$. Then it is easy to verify that the hypotheses {\rm(A1)-(A4)} hold. In particular, it applies for all our purposes in this paper.

Our condition {\rm(A3)} is a quantitative version of the ``local monotonicity condition'' that was introduced in \cite{Chen_2008}. See \cite{Ishii_1997,Chen_2008} for more examples of Hamiltonians verifying the hypotheses {\rm(A1)}-{\rm(A4)}.
\end{remark}

\begin{proof}
Assume to the contrary that
\begin{align}\label{eq:sigma'}
\sup\limits_\Omega(\underline w-\overline w)>0.
\end{align}

\noindent{\bf Step 1.}  We may assume, without loss of generality, that $M=0$ in the hypothesis {\rm(A4)}. Indeed, if we make the change of variables $\underline{w}'(t,x) = \underline{w}(t,x) + Mt$ and $\overline{w}'(t,x) = \overline{w}(t,x)+Mt$, then $\underline{w}', \overline{w}'$ are, respectively, a sub-solution and a super-solution of \eqref{eq:D1} with $L$ replaced by ${L}'=L+M$, and $H(t,x,p)$ replaced by $H'(t,x,p)= H(t,x,p) - M$. This function $H'$ satisfies the hypotheses {\rm(A1)}-{\rm(A4)} with $M=0$. Henceforth in the proof we assume that the hypothesis {\rm(A4)} holds with $M=0$.

\noindent{\bf Step 2.}  It suffices to show that $\underline{w} \leq \overline{w}$ under the additional assumption that $\underline{w} \leq K$ for some $K>0$.

Indeed, if $\underline{w}$ is unbounded in $\Omega$, then fix a constant $K>0$ and take a sequence $\{g_j\}$ of smooth functions satisfying $ g_j(r) \nearrow \min\{r,K\}$ and
$$
0 \leq g_j'(r) \leq 1, \quad g_j'(r)r \leq r, \quad g_j(r) \leq \min\{r,K\}\,\quad \text{ for all }r \in \mathbb{R}.
$$
Then $\hat{w}:= g_j(\underline w)$ is a viscosity sub-solution of \eqref{eq:D1}, since in the region $\{(t,x):~\hat{w} - Lt >0\} \subset\{(t,x):~\underline{w} - Lt >0\}$, we may use the hypothesis {\rm(A4)} to yield
\begin{equation*}
\begin{split}
\quad\partial_t \hat w+H^*(t,x, D\hat w)&=g_k'(\underline w)\partial_t \underline w+H^*(t,x, g_k'(\underline w)D\underline w)\\
&\leq g_k'(\underline w) \left[\partial_t \underline w+H^*(t,x,D\underline w)\right]\leq 0. 
\end{split}
\end{equation*}
By the stability of property of viscosity super and sub-solutions \cite[Theorem 6.2]{BarlesLect}, we may let $j \to \infty$ to conclude that $\min\{\underline w,K\}$ is a viscosity sub-solution of \eqref{eq:D1} for each $K>0$.
It now remains to prove Theorem \ref{thm:D} for all bounded above viscosity sub-solutions,  since then
$$\min\{\underline w, K\} \leq \overline w \quad \text{ for all } K>0\quad \Rightarrow\quad \underline{w} \leq \overline{w}.$$

For $\lambda,\delta\in(0,1)$, denote
$$W(t,x)=\lambda^2\underline w(t,x)-\overline w(t,x)-\delta(\psi(x)+Ct+\frac{1}{T-t})-\lambda\delta Ct,$$
where $\psi(x) = \frac{1}{2} \log (|x|^2 + 1)$ and $C = \bar C(\lambda,0)$ as in the hypothesis {\rm (A4)}.

 \noindent{\bf Step 3.}  We choose $\lambda \nearrow 1$, $\delta \in (0, \bar \epsilon(\lambda,0)]$, $R>0$ and $(t_0,x_0) \in \Omega_R:= \Omega \cap [(0,T) \times B_R(0)]$
   such that
   \begin{equation}\label{eq:sigma11}
W(t_0,x_0)=\max\limits_{\Omega_R}{W(t,x)}=\max\limits_{\Omega}{W(t,x)}>0.
   \end{equation}

From \eqref{eq:sigma'} and Step 3, we may fix $\lambda\nearrow 1$ and $\delta\searrow0 $ such that
\begin{equation*}
\sup_{\Omega}W(t,x)>0,\quad\text{ and }\quad
W(t,x)\leq-\frac{\delta}{T-t} \quad \text{ on }{\partial_p\Omega}.
\end{equation*}
Since  $\psi(R)\to \infty$ as $R\to\infty$ and $\underline w - \overline{w} \leq K$, we arrive at
$$\sup_{(t,x) \in \Omega:~|x|=R}W(t,x)\to -\infty \quad \text{ as } R\to \infty,$$
whence we may fix $R\gg1$ so that $\max\limits_{\Omega_R}{W(t,x)}=\max\limits_{\Omega}{W(t,x)}>0$ holds. It remains to observe that the maximum $(t_0,x_0)$ in $\overline{\Omega_R}$ is attained in the interior, since $W(t,x) <0$ when $t = T$ or when $(t,x) \in \partial_p \Omega$.

\noindent {\bf Step 4.} With $x_0$ as being given in Step 3, fix $\epsilon>0$ small enough so that
\begin{equation}\label{eq:ssss}
\epsilon  \bar C(\lambda, x_0) \leq \bar C(\lambda,0) \quad \text{ and }\quad \delta\epsilon \leq \bar \epsilon(\lambda,x_0),
\end{equation}
and define
\begin{equation}\label{eq:Wprime}
\tilde{W}(t,x) := W(t,x) - \delta\lambda\epsilon \psi(x-x_0) - \frac{1}{2}|t-t_0|^2,
\end{equation}
where $\psi(x)=\frac{1}{2} \log{(|x|^2+1)}$ and $C = \bar C(\lambda,0)$ is as before. Then, $(t_0,x_0)$ is a strict global maximum of $\tilde{W}(t,x)$. Define also
\begin{align*}
\begin{split}
\Psi_{\alpha,\beta}(t,x,s,y)
=& \lambda^2 \underline w(t,x) -\overline w(s,y)-\delta(\psi (x)+Ct+\frac{1}{T-t})-\lambda\delta(\epsilon \psi (x-x_0)+Ct)\\
&-\frac{\alpha}{2}|x-y|^2- \frac{\beta}{2}|t-s|^2-\frac{1}{2}|t-t_0|^2.
\end{split}
\end{align*}

\noindent {\bf Step 5.} We claim that there exists $\underline\alpha>0$ such that if  $\min\{\alpha,\beta\} \geq \underline\alpha$, then
\begin{itemize}
\item[{\rm (i)}] $\Psi_{\alpha,\beta}$ has a local maximum point $(t_1,x_1,s_1,y_1)$ in $\Omega_R\times \Omega_R$;
\item[{\rm (ii)}] $\Psi_{\alpha,\beta}(t_1,x_1,s_1,y_1) \geq  \tilde W(t_0,x_0) = W(t_0,x_0) >0$;
\item[{\rm (iii)}] $\beta|t_1-s_1|^2 + \alpha |x_1-y_1|^2\to 0,  \text{ as } \min\{\alpha,\beta\} \to \infty$;
\item[{\rm (iv)}] $(t_1,x_1) \to (t_0,x_0)$ and $(s_1,y_1) \to (t_0,x_0)  \text{ as } \min\{\alpha,\beta\} \to \infty$,
\end{itemize}
where $\Omega_R=\Omega\cap[(0,T)\times B_R(0)]$. Since $\overline{w} \geq 0$ and $\underline{w} \leq K$ by Step 2, we see that $\sup_{\Omega_R \times \Omega_R} \Psi_{\alpha,\beta} \leq K$ independently of $\alpha$ and $\beta$, and has a maximum point $(t_1,x_1,s_1,y_1) \in \overline{\Omega}_R \times \overline{\Omega}_R$. Now, by  \eqref{eq:sigma11},
$$
\Psi_{\alpha,\beta}(t_1,x_1,s_1,y_1) \geq \max_{\Omega_R}\Psi_{\alpha,\beta}(t,x,t,x) = \tilde{W}(t_0,x_0) = W(t_0,x_0).
$$
This proves assertion (ii).

Furthermore, the boundedness also yields $\beta|t_1-s_1|^2 + \alpha|x_1 - y_1|^2 = O(1)$. We claim that $(t_1,x_1) \to (t_0,x_0)$ and $(s_1,y_1) \to (t_0,x_0)$. Indeed,  we may pass to a subsequence to get $(\hat t, \hat x)$ such that
$(t_1,x_1) \to (\hat t, \hat x)$ and $(s_1,y_1) \to (\hat t, \hat x)$ as $\min\{\alpha,\beta\} \to \infty$. Now, by (ii) we can write
\begin{align*}
 \frac{\alpha}{2}|x_1-y_1|^2 + \frac{\beta}{2}|t_1-s_1|^2 \leq - \tilde{W}(t_0,x_0) +(\tilde{W}(t_1,x_1)  + \overline{w}(t_1,x_1)) - \overline{w}(s_1,y_1) .
\end{align*}
Letting $\min\{\alpha,\beta\} \to \infty$, then $(t_1,x_1,s_1,y_1) \to (\hat t, \hat x, \hat t, \hat x)$. Using the fact that $\tilde{W}(t,x) + \overline{w}(t,x)$ (which is essentially $\lambda^2 \underline{w}(t,x)$ up to addition of continuous functions) and $-\overline{w}(s,y)$ are both upper semi-continuous in $\Omega$, we may take limsup as $\min\{\alpha,\beta\} \to \infty$ and deduce that
$$
0 \leq \limsup \left[\frac{\alpha}{2}|x_1-y_1|^2 + \frac{\beta}{2}|t_1-s_1|^2\right] \leq - \tilde{W}(t_0,x_0)  + \tilde{W}(\hat t, \hat x) \leq 0.
$$
Since $(t_0,x_0)$ is a strict maximum point of $\tilde{W}$, we must have $(\hat t, \hat x) = (t_0,x_0)$.  This proves assertions (iii) and (iv).

Finally, $(t_1,x_1, s_1,y_1) \to (t_0,x_0,t_0,x_0)$ and hence must be an interior point of $\Omega_R \times \Omega_R$ when $\min\{\alpha,\beta\}$ is sufficiently large. This proves (i).

\noindent {\bf Step 6. } We show the following inequality:
\begin{equation}\label{eq:maini}
\frac{\delta}{T^2} \leq H^*(s_1,y_1,\alpha (x_1-y_1))-H_*(t_1,x_1, \alpha (x_1-y_1))+|t_1-t_0|.
\end{equation}

Observe that $(t_1,x_1)$ is an interior maximum point of the function $\underline w(t,x)-\varphi(t,x)$, where
\begin{equation*}
\begin{split}
\varphi(t,x)=& \frac{1}{\lambda^2}[\overline w(s_1,y_1)+\delta(\psi(x)+Ct+\frac{1}{T-t})+\lambda\delta(\epsilon\psi(x-x_0)+Ct)\\
&+\frac{\alpha}{2}|x-y_1|^2+ \frac{\beta}{2}|t-s_1|^2 +\frac{1}{2}|t-t_0|^2  ].
\end{split}
\end{equation*}
Also $\underline w(t_1,x_1)>0$, which is a consequence of $\overline w(s_1,y_1) \geq 0$ and $\Psi_\alpha(t_1,x_1,s_1,y_1) >0$.
By definition of $\underline w$ being a viscosity sub-solution of \eqref{eq:D1}, we have
\begin{equation*}
\begin{split}
&\quad \frac{1}{\lambda^2}\left[\delta (C+ \frac{1}{(T-t_1)^2}+\lambda C)+\beta (t_1-s_1)+(t_1-t_0) \right]\\
&+H_*\left(t_1,x_1, \frac{1}{\lambda^2}\left(\delta D_x \psi(x_1)+\lambda\delta \epsilon D_x \psi(x_1-x_0)+\alpha(x_1-y_1)\right)\right)\leq 0,
\end{split}
\end{equation*}
which can be rewritten  as
\begin{equation}\label{eq:Dsub-solutionw'}
\begin{split}
&\quad
\delta (C+\frac{1}{T^2}+\lambda C)+\beta (t_1-s_1)+(t_1-t_0) \\
&\qquad \qquad +\lambda^2H_*\left(t_1,x_1, \frac{1}{\lambda}\left(\delta\epsilon D_x \psi(x_1-x_0)+ \hat{q}\right)\right)\leq 0,
\end{split}
\end{equation}
where $\hat q = \frac{1}{\lambda}\left(\delta D_x \psi(x_1)+\alpha(x_1-y_1)\right)$.
In the view of $D\psi(x_1-x_0)=\frac{x_1-x_0}{|x_1-x_0|^2+1}$, we may apply the hypothesis \rm{(A4)} to get
\begin{align*}
&\quad - \delta (C+\frac{1}{T^2}+\lambda C)-\beta (t_1-s_1)-(t_1-t_0) \\
& \geq \lambda \left[ H_*\left(t_1,x_1, \hat q \right)   - \delta\epsilon \bar C(\lambda,x_0) \right] \\
& \geq  \lambda H_*\left(t_1,x_1, \frac{1}{\lambda}\left(\delta D_x \psi(x_1)+\alpha(x_1-y_1)\right) \right)   - \lambda\delta  C,
\end{align*}
where we used $\epsilon \bar C(\lambda,x_0) \leq \bar C(\lambda,0) = C$ (due to \eqref{eq:ssss}) in the last inequality.
Applying the hypothesis {\rm(A4)} once more, we have
\begin{align*}
&\quad - \delta (C+\frac{1}{T^2}+\lambda C)-\beta (t_1-s_1)-(t_1-t_0) \\
&\geq \left[H_*\left(t_1,x_1, \alpha(x_1-y_1) \right)  - \delta C\right]   - \lambda\delta  C\\
&\geq H_*\left(t_1,x_1, \alpha(x_1-y_1) \right)    - \delta C - \lambda\delta  C,
\end{align*}
and hence
\begin{equation}\label{eq:Dsub-solutionw''}
\frac{\delta}{T^2} +\beta (t_1-s_1)+(t_1-t_0)+H_*(t_1,x_1, \alpha(x_1-y_1))\leq 0.
\end{equation}

In the same way, $(s_1,y_1)$ is a interior minimum point of the function  $\overline w(s,y)-\psi(s,y)$ with
\begin{equation*}
\begin{split}
\psi(s,y)=&\lambda^2 \underline w(t_1,x_1)-
\delta(\psi(x_1)+Ct_1+\frac{1}{T-t})-\lambda\delta(\psi(x_1-x_0)+Ct_1)\\
&-\frac{\alpha}{2}|x_1-y|^2-\frac{\beta}{2}|t_1-s|^2)-\frac{1}{2}|t_1-t_0|^2,
\end{split}
\end{equation*}
 whence
\begin{align}\label{eq:Dsuper-solutionw}
 \beta(t_1-s_1)+H^*(s_1,y_1,\alpha (x_1-y_1))\geq 0.
\end{align}
Subtracting \eqref{eq:Dsub-solutionw''} from \eqref{eq:Dsuper-solutionw},  we obtain \eqref {eq:maini} as claimed.

By Step 5 \rm{(iv)}, we have $(t_1,x_1)\to (t_0,x_0)$ and $(s_1,y_1)\to (t_0,x_0)$ as $\min\{\alpha,\beta\}\to \infty$.
On the one hand, if $(t_0,x_0)\notin \Gamma$, then there exists $\alpha_1>0$ such that  $(t_1,x_1)$ and $(s_1,y_1)$ enter the $(\delta_0/2)$-neighborhood of $(t_0,x_0)$ whenever $\min\{\alpha,\beta\} \geq \alpha_1$. Now, fix $\alpha$ and let $\beta \to \infty$, then after passing to a sequence, we have
$$
t_1, s_1 \to \bar t, \quad x_1 \to \bar x, \quad y_1 \to \bar y.
$$
Furthermore, by Step 5, we have
\begin{equation}\label{eq:txbar}
\bar t \to t_0,\quad \bar x \to x_0, \quad \bar y \to x_0,\quad \text{ and }\quad  {\alpha}|\bar x - \bar y|^2  \to 0 \quad \text{ as }\quad \alpha \to \infty.
\end{equation}
Hence, we deduce from \eqref{eq:maini} and the hypothesis \rm{(A2)}  that
\begin{align*}
\frac{\delta}{T^2} &\leq H^*(\bar t,\bar y,\alpha (\bar x - \bar y))-H_*(\bar t,\bar x, \alpha (\bar x - \bar y))+|\bar t-t_0|\\
&\leq \omega_R\left( \alpha |\bar x - \bar y|^2 + \frac{1}{\alpha}\right) + o(1),
\end{align*}
from which we derive a contradiction for large enough $\alpha$.
This proves Theorem \ref{thm:D} in case $(t_0,x_0) \in \Omega \setminus \Gamma$.

On the other hand, $(t_0,x_0)\in \Gamma$. Let $\delta_0$ and the unit vector $(h_0,k_0) \in \mathbb{R} \times \mathbb{R}^N$ be given by the hypothesis {\rm(A3)}. Define
\begin{equation}\label{eq:Psi'}
\begin{split}
\tilde{\Psi}_{\alpha,\beta}(t,x,s,y) = &\lambda^2 \underline w(t,x) -\overline w(s-\alpha^{-1/2}h_0,y-\alpha^{-1/2}k_0) \\
&-\delta(\psi (x)+Ct+\frac{1}{T-t})-\lambda\delta(\epsilon \psi (x-x_0)+Ct)\\
&-\frac{\alpha}{2}|x-y|^2- \frac{\beta}{2}|t-s|^2-\frac{1}{2}|t-t_0|^2.\end{split}
\end{equation}
By repeating Steps 5 and 6, we can again obtain
\begin{equation}\label{eq:sss}
\frac{\delta}{T^2} \leq H^*(\bar t -\alpha^{-1/2}h_0,\bar y -\alpha^{-1/2}k_0,\alpha (\bar x - \bar y))-H_*(\bar t, \bar x, \alpha (\bar x - \bar y))+|\bar t-t_0|
\end{equation}
for some  $\bar t, \bar x, \bar y$ such that for all $\alpha$ large, $(\bar t, \bar x), (\bar t, \bar y)$ enter the $(\delta_0/2)$-neighborhood of $(t_0,x_0)$ and \eqref{eq:txbar} holds. Furthermore, by verifying that
\begin{align*}
\frac{((\bar t - \alpha^{-1/2}h_0) - \bar t, \bar y - \alpha^{-1/2}k_0 - \bar x)}{\|((\bar t - \alpha^{-1/2}h_0) - \bar t, \bar y - \alpha^{-1/2}k_0 - \bar x)\|}&= -\frac{(\alpha^{-1/2}h_0, \bar x - \bar y + \alpha^{-1/2}k_0 )}{\|(\alpha^{-1/2}h_0, \bar x - \bar y + \alpha^{-1/2}k_0 )\|}\\
&=  -\frac{(h_0, \sqrt{\alpha}(\bar x - \bar y) + k_0 )}{\|(h_0,\sqrt{\alpha}(\bar x - \bar y) + k_0 )\|}\\
& \to (h_0,k_0) \quad \text{ as }\alpha \to \infty,
\end{align*}
we may apply hypothesis {\rm(A3)} to inequality \eqref{eq:sss} to get
\begin{align*} 
\frac{\delta}{T^2} &\leq \omega_R\left(\Big[\alpha^{-1/2}(h_0 + |k_0|) + |\bar x - \bar y|\Big](1+\alpha |\bar x - \bar y|)\right) + |\bar t- t_0| \\
&\leq \omega_R\left(|\bar x - \bar y| + \alpha |\bar x - \bar y|^2 + C_0 ( \sqrt\alpha|\bar x - \bar y| +\frac{1}{\sqrt{\alpha}})\right) + |\bar t - t_0|. 
\end{align*}
Letting $\alpha \to \infty$,  we can similarly obtain a contradiction.
\end{proof}

A direct consequence  of Theorem \ref{thm:D} is the following uniqueness result.
\begin{corollary}\label{cor:uniquD}
Assume that the Hamiltonian $H$ satisfies the hypotheses {\rm{(A1)}}-{\rm{(A4)}}. Suppose that \eqref{eq:D1} has a continuous viscosity solution $w$. Then, $w$ is the unique viscosity solution of \eqref{eq:D1} in the class of all (continuous) viscosity solution.
\end{corollary}

\section{Two useful lemmas from \cite{LLL2019}}\label{SL}
In this section, we present two lemmas from  \cite{LLL2019}, which are used in this paper.
The first result is used to prove Propositions \ref{prop:1} and Corollary \ref{cor:underlinec1}, Lemma \ref{lem:tildeuv} and Proof of Theorem \ref{thm:1-2}.

\begin{lemma}[{\cite[Lemma 2.3]{LLL2019}}]\label{lem:entire1}
Let $-\infty\leq\underline c<\overline c\leq\infty$, and
let $(u,v)$ be a solution of \eqref{eq:1-1} in $\{(t,x): \underline{c} t \leq x \leq \overline{c}t\}$.
\begin{itemize}
\item[{\rm(a)}]If
$\displaystyle \liminf_{t \to \infty} \inf_{(\underline c + \eta) t < x < (\overline c - \eta)t} v(t,x) >0$ for each $0<\eta < (\overline{c}-\underline{c})/2$, then
 $$\displaystyle \limsup_{t \to \infty} \sup_{(\underline c +\eta) t < x < (\overline c - \eta)t}u(t,x) \leq k_1, \quad \liminf_{t \to \infty} \inf_{(\underline c + \eta) t < x < (\overline c - \eta)t}v(t,x) \geq k_2,$$
 for each $0<\eta< (\overline{c}-\underline{c})/2$;

\item[{\rm(b)}] If $\displaystyle \lim_{t \to \infty} \sup_{(\underline c + \eta) t < x < (\overline c - \eta) t} v(t,x)=0$ and $\displaystyle \liminf_{t \to \infty} \inf_{(\underline c +\eta) t < x < (\overline c - \eta)t} u(t,x) >0$ for each $0<\eta < (\overline{c}-\underline{c})/2$, then
    $$\displaystyle \lim_{t \to \infty} \sup_{(\underline c + \eta) t < x < (\overline c -\eta)t}|u(t,x) -1|=0, \quad \text{ for each }0 < \eta < (\overline{c}-\underline{c})/2;$$

\item[{\rm(c)}]
If
 $\displaystyle \liminf_{t \to \infty} \inf_{(\underline c + \eta) t < x < (\overline c - \eta)t} u(t,x) >0$ for each $0 < \eta< (\overline{c}-\underline{c})/2$, then
 $$\displaystyle \liminf_{t \to \infty} \inf_{(\underline c + \eta) t < x < (\overline c - \eta)t}u(t,x) \geq k_1, \quad \limsup_{t \to \infty} \sup_{(\underline c + \eta) t < x < (\overline c - \eta)t}v(t,x) \leq k_2
 $$
 for each $0<\eta < (\overline{c}-\underline{c})/2$;
 \item[{\rm(d)}] If $\displaystyle \lim_{t \to \infty} \sup_{(\underline c + \eta) t < x < (\overline c - \eta)t} u(t,x)=0$ and $\displaystyle \liminf_{t \to \infty} \inf_{(\underline c +\eta) t < x < (\overline c - \eta)t} v(t,x) >0$  for each $0 <\eta < (\overline{c}-\underline{c})/2$, then
     $$\displaystyle \lim_{t \to \infty} \sup_{(\underline c + \eta) t < x < (\overline c - \eta)t }|v(t,x) -1|=0, \quad \text{ for each }0 < \eta < (\overline{c}-\underline{c})/2.$$
\end{itemize}
\end{lemma}
\begin{proof}
We only prove  {\rm(c)} and the other assertions  follow from the similar arguments.
 Suppose {\rm(c)} is false, then there exists $(t_n,x_n)$ such that
\begin{equation*}
    \begin{array}{c}
        c_n:= \frac{x_n}{t_n} \to c \in (\underline c, \overline c) \,\text{ and }\, \lim\limits_{n\to\infty}u(t_n,x_n)<k_1 \text{ or }  \lim\limits_{n\to\infty}v(t_n,x_n)>k_2.
     \end{array}
\end{equation*}
Define $(u_n,v_n)(t,x):= (u,v)(t_n + t, x_n + x)$. We  pass to the limit so that $(u_n,v_n)$ converges in $C_{loc}(\mathbb{R} \times \mathbb{R})$ to an entire solution
$(\hat u, \hat v)$ of \eqref{eq:1-1}. And there exists $\delta >0$ such that $ (\hat{u}, \hat{v})(t,x) \succeq (\delta, 1)$ for $(t,x) \in \mathbb{R}^2$. Let $(\underline{U}, \overline{V})$ be the solutions of the Lotka-Volterra system of ODEs
$$
U_t = U(1-U - aV), \quad V_t = rV(1-bU-V),
$$
with initial data $(U(0), V(0)) = (\delta,1)$, so that $(\underline{U}, \overline{V})(\infty) = (k_1,k_2)$.
By comparison in the time interval $[-T,0]$, we reveal that for each $T>0$,
$$
 (\hat u,\hat v)(t,x)\succeq (\underline{U}, \overline{V})(t+T) \, \text{ for }(t,x) \in [-T,0] \times \mathbb{R},
$$
so that we in particular have, for every $T>0$,
$$
 (\hat u,\hat v)(0,0)\succeq (\underline{U}, \overline{V})(T).
$$
Letting $T \to \infty$, we obtain $(\hat u,\hat v)(0,0) \succeq (k_1,k_2)$.
In particular, we deduce that
$$
\lim_{n \to \infty} (u,v)(t_n,x_n) = \lim_{n \to \infty} (u_n,v_n)(0,0) = (\hat u, \hat v)(0,0)\succeq (k_1,k_2).
$$
This is a contradiction and  proves {\rm(c)}.
\end{proof}

The following result is applied to prove Proposition \ref{prop:overlinec2} and Proposition \ref{prop:underlinec3}. 
\begin{lemma}[{\cite[Lemma 2.4]{LLL2019}}]\label{lem:appen1} 
Let $\hat{c}>0$, $t_0>0$, and  $(\tilde u,\tilde v)$ be a solution of
\begin{equation}\label{eq:A2}
\left\{
\begin{array}{ll}
\medskip
\partial_t \tilde u-\partial_{xx}\tilde u=\tilde u(1-\tilde u-a\tilde v),& 0\leq x\leq\hat ct, t>t_0,\\
\medskip
\partial _t \tilde v-d\partial_{xx}\tilde v=r \tilde v(1-b\tilde u-\tilde v),&  0\leq x\leq\hat ct, t>t_0,\\
\tilde u(t_0,x)=\tilde u_0(x), \tilde v(t_0,x)=\tilde v_0(x), &  0\leq x\leq\hat ct_0.
\end{array}
\right .
\end{equation}
\begin{itemize}
\item[{\rm (a)}] If $\hat{c}>2$ and there exists $\hat \mu>0$ such that 
\begin{itemize}
\item[{\rm(i)}] $\lim\limits_{t\to\infty}(\tilde u,\tilde v)(t,0)=(k_1,k_2)$ and $\lim\limits_{t\to\infty}(\tilde u,\tilde v)(t,\hat ct)=(0,1)$,
\item[{\rm(ii)}]  $\lim_{t\to\infty} e^{\mu t}\tilde u(t,\hat ct)=0~$ for each $\mu \in [0,\hat \mu),$

\end{itemize}
then 
$$
\lim_{t\to\infty} \sup_{ct<x\leq \hat{c}t} \tilde u(t,x) = 0   \quad \text{ for each }c >  c_{\hat{c},\hat\mu},
$$
where
\begin{equation*}
 c_{\hat{c},\hat\mu}=\left\{
\begin{array}{ll}
\medskip
c_{\textup{LLW}},& \text{if } \hat\mu\geq   \lambda_{\textup{LLW}} (\hat{c} - c_{\textup{LLW}}),\\
\hat c-\frac{2\hat\mu}{\hat c-\sqrt{\hat c^2-4(\hat\mu+1-a)}},& \text{if }0< \hat\mu< \lambda_{\textup{LLW}}(\hat{c} - c_{\textup{LLW}}); \\
\end{array}
\right .
\end{equation*}
\item[{\rm (b)}] If $\hat{c} > 2\sqrt{dr}$ and there exists $\hat\mu >0$ such that

\begin{itemize}

\item[{\rm(i)}] $\lim\limits_{t\to\infty}(\tilde u,\tilde v)(t,0)=(k_1,k_2)$, and $\lim\limits_{t\to\infty}(\tilde u,\tilde v)(t,\hat ct)=(1,0)$,

\item[{\rm(ii)}]  $\lim_{t\to\infty} e^{\mu t}\tilde v(t,\hat ct)=0~$ for each $\mu \in [0, \hat \mu),$

\end{itemize}
then 
$$
\lim_{t\to\infty} \sup_{ct<x\leq \hat{c}t} \tilde v(t,x) = 0   \quad \text{ for each }c >\tilde{c}_{\hat{c},\hat\mu},
$$
where
\begin{equation*}
\tilde{c}_{\hat{c},\hat\mu}= \left\{
\begin{array}{ll}
\medskip
\tilde{c}_{\textup{LLW}},& \text{ if } \hat\mu\geq \tilde{\lambda}_{\rm LLW} (\hat{c} - \tilde{c}_{\textup{LLW}}),\\
\hat c-\frac{2d\hat\mu}{\hat c-\sqrt{\hat c^2-4d[\hat\mu+r(1-b)]}},&  \text{ if } 0<\hat\mu<\tilde{\lambda}_{\rm LLW} (\hat{c} - \tilde{c}_{\textup{LLW}}).
%
\end{array}
\right.
\end{equation*}
\end{itemize}
Here  ${c}_{\textup{LLW}}, \tilde{c}_{\textup{LLW}}$ are given in Theorem \ref{thm:LLW} and Remark \ref{rmk:LLW}, and
\begin{equation}\label{eq:LLL}
\lambda_{\rm LLW}=\frac{{c}_{\textup{LLW}}-\sqrt{{c}_{\textup{LLW}}^2-4(1-a)}}{2}, \quad  \tilde{\lambda}_{\rm LLW}=\frac{\tilde{c}_{\textup{LLW}}-\sqrt{\tilde{c}_{\textup{LLW}}^2-4dr(1-b)}}{2d}.\end{equation}
\end{lemma}

\end{appendices}

\begin{center}

\end{center}


\begin{thebibliography}{20}
\small{
\bibitem{BarlesLect} Y. Achdou, G. Barles, H. Ishii, G.L. Litvinov, Hamilton-Jacobi equations: approximations, numerical analysis and applications, Lecture Notes in Math., Vol. 2074, Springer, Heidelberg, 2013.
\bibitem{Alvarez_1997}
O. Alvarez, Bounded-from-below viscosity solutions of Hamilton-Jacobi equations. Differential Integral Equations 10 (1997) 419-436.
\bibitem{Aronson_1975} D.G. Aronson, H.F. Weinberger, Nonlinear diffusion in population genetics, combustion, and nerve pulse propagation, Partial differential equations and related topics (ed. J.A. Goldstein; Springer, Berlin, 1975) 5-49.  \url{https://doi.org/10.1007/BFb0070595}
\bibitem{Berestycki_2008} H. Berestycki, F. Hamel, G. Nadin, Asymptotic spreading in heterogeneous diffusive excitable media, J. Func. Anal. 255 (2008) 2146-2189. \url{https://doi.org/10.1016/j.jfa.2008.06.030}
\bibitem{Berestycki_2012} H. Berestycki,  G. Nadin, Spreading speeds for one-dimensional monostable reaction- diffusion equations, J. Math. Phys. 53 (2012) 115619. \url{https://doi.org/115619. 10.1063/1.4764932}
\bibitem{BNpreprint} H. Berestycki, G. Nadin, Asymptotic spreading for general heterogeneous equations, submitted.
 \bibitem{Booty_1993} M.R. Booty, R. Haberman, A.A. Minzoni, The accommodation of traveling waves of Fisher’s to the dynamics of the leading tail, SIAM J. Appl. Math. 53 (1993) 1009-1025. \url{https://doi.org/10.1137/0153050}
  \bibitem{Carrere_2018} C. Carr\`{e}re, Spreading speeds for a two-species competition-diffusion system, J. Differential  Equations 264 (2018) 2133-2156. \url{https://doi.org/10.1016/j.jde.2017.10.017}
 \bibitem{Chen_2008} X. Chen, B. Hu, Viscosity solutions of discontinuous Hamilton-Jacobi equations, Interfaces Free Bound. 10 (2008) 339-359. \url{https://doi.org/10.471/IFB/192}

     \bibitem{Crandall_1983} M.G. Crandall, P.L. Lions, Viscosity solution of Hamilton-Jacobi equations, Trans. Amer. Math. Soc. 277 (1983) 1-42. \url{https://doi.org/10.2307/1999343}
 \bibitem{Crandall_1987}  M.G.  Crandall, P.L. Lions, H. Ishii, Uniqueness of viscosity solutions of Hamilton-Jacobi equations revisited, J. Math. Soc. Japan 39 (1987) 581-596. \url{https://doi.org/10.2969/jmsj/03940581}
\bibitem{Du_2018} {Y.H. Du, C.H. Wu},  Spreading with two speeds and mass segregation in a diffusive competition system with free boundaries, Calc. Var. Partial Differential Equations 57 (2018), Art. 52.  \url{https://doi.org/10.1007/s00526-018-1339-5}
\bibitem{Ducrot_preprint} A. Ducrot, T. Giletti, H. Matano, Spreading speeds for multidimensional reaction-diffusion systems of the prey-predator type, Calc. Var. Partial Differential Equations 58 (2019) Art.  137. \url{https://doi.org/10.1007/s00526-019-1576-2}
 \bibitem{Ebert_2000} U. Ebert, W.V. Saarloos, Front propagation into unstable states: Universal algebraic convergence towards uniformly translating pulled fronts, Physica D 146 (2000) 1-99. \url{https://doi.org/10.1016/S0167-2789(00)00068-3}
\bibitem{Evans_1989} L.C. Evans, P.E. Souganidis, A PDE approach to geometric optics for certain semilinear parabolic equations, Indiana Univ. Math. J. 38 (1989) 141-172. \url{https://www.jstor.org/stable/24895345}
\bibitem{Evans_1989b} L.C. Evans, P.E. Souganidis, G. Fournier, M. Willem, A PDE approach to certain large deviation problems for systems of parabolic equations, Annales de l?I. H. P., section C, tome S6 (1989),  229-258 \url{http://www.numdam.org/item?id=AIHPC_1989__S6__229_0}
\bibitem{Fang_2016} J. Fang, Y.J. Lou, J.H. Wu, Can pathogen spread keep pace with its host invasion? SIAM J. Appl. Math. 74 (2016) 1633-1657.  \url{https://doi.org/10.1137/15M1029564}
\bibitem{Fang_2014} J. Fang, X.Q. Zhao, Traveling waves for monotone semiflows with weak compactness, SIAM J. Math. Anal. 46 (2014) 3678-3704. \url{https://doi.org/10.1137/140953939}
\bibitem{Fisher_1937} R.A. Fisher, The wave of advance of advantageous genes, Ann. Hum. Genet 7 (1937) 355-369. \url{https://doi.org/10.1111/j.1469-1809.1937.tb02153.x}
\bibitem{Freidlin_1985} M.I. Freidlin, Limit theorems for large deviations and reaction-diffusion equation, Ann. Probab. 13 (1985) 639-675. \url{https://doi.org/10.1214/aop/1176992901}
\bibitem{Freidlin_1991}  M.I. Freidlin, Coupled reaction-diffusion equations. Ann. Probab. 19 (1991)  29-57 \url{https://www.jstor.org/stable/2244251} 
\bibitem{Garnier_2012} J. Garnier, T. Giletti, G. Nadin, Maximal and minimal spreading speeds for re- action diffusion equations in nonperiodic slowly varying media, J. Dynam. Differential Equations 24 (2012) 521-538. \url{https://doi.org/10.1007/s10884-012-9254-5}
\bibitem{Giga_2010} Y. Giga, P. G\'{o}rka, P. Rybka, A comparison principle for Hamilton-Jacobi equations with discontinuous Hamiltonians, Proc. Amer. Math. Soc. 139 (2011) 1777-1785. \url{https://doi.org/10.1090/S0002-9939-2010-10630-5}
\bibitem{Girardin_2018}  L. Girardin, K.Y. Lam, Invasion of an empty habitat by two competitors: spreading properties of monostable two-species competition-diffusion systems, P. Lond. Math. Soc. 119 (2019) 1279-1335. \url{https://doi.org/10.1112/plms.12270}
\bibitem{Guo_2015}  J.S. Guo, C.H. Wu, Dynamics for a two-species competition-diffusion model with two free boundaries, Nonlinearity 28 (2015) 1-27.  \url{https://doi.org/10.1088/0951-7715/28/1/1}
\bibitem{Hamel2012} F. Hamel, G. Nadin, Spreading properties and complex dynamics for monostable reaction-diffusion equations, Comm. Partial Differential Equations 37 (2012) 511-537. \url{https://doi.org/10.1080/03605302.2011.647198}
\bibitem{Holzer_2014} M. Holzer, A. Scheel,  Accelerated fronts in a two stage invasion process, SIAM J. Math. Anal. 46 (2014) 397-427. \url{https://doi.org/10.1137/120887746}
\bibitem{Frank_1975} F.C. Hoppensteadt, Mathematical theories of populations: demographics, genetics, and epidemics, CBMS-NSF Regional Conf. Ser. in Appl. Math. 15, SIAM, Philadelphia, 1975. \url{https://doi.org/10.1137/1.9781611970487}
\bibitem{Iida_2011} M. Iida, R. Lui, H. Ninomiya, Stacked fronts for cooperative systems with equal diffusion coefficients, SIAM J. Math. Anal. 43 (2011) 1369-1389. \url{https://doi.org/10.1137/100792846}
 \bibitem{Ishii_1985} H. Ishii, Hamilton-Jacobi equations with discontinuous Hamiltonians on arbitrary open Sets, Bull. Facul. Sci Eng. Chuo Univ. 28 (1985) 33-77.
\bibitem{Ishii_1997} H. Ishii, Comparison results for hamilton-jacobi equations without growth condition on solutions from above, Appl. Anal. 67 (1997) 357-372. \url{https://doi.org/10.1080/00036819708840617}
   \bibitem{Kametaka_1976} Y. Kametaka, On the nonlinear diffusion equation of Kolmogorov-Petrovski-Piskunov type, Osaka J. Math. 13 (1976) 11-66. \url{https://doi.org/10.18910/9093}
  \bibitem{Kolmogorov_1937} A.N. Kolmogorov, I.G. Petrovsky, N.S. Piskunov,  \'{E}tude de l\'{e}quation de la diffusion avec croissance de la quantit\'{e} de mati\'{e}re et son application \`{a} un probl\'{e}me biologique, Bulletin Universit\'{e} d'\'{E}tat \`{a} Moscou 1 (1937) 1-25.
\bibitem{Lewis_2002} M.A. Lewis, B.T. Li, H.F. Weinberger, Spreading speed and linear determinacy for two-species competition models, J. Math. Biol. 45 (2002) 219-233. \url{https://doi.org/10.1007/s002850200144}
\bibitem{Li_2018} B.T. Li, Multiple invasion speeds in a two-species integro-difference competition model, J. Math. Biol. 76 (2018) 1975-2009. \url{https://doi.org/10.1007/s00285-017- 1200-z}
  \bibitem{Li_2005} B.T. Li, H.F. Weinberger, M.A. Lewis, Spreading speeds as slowest wave speeds for cooperative systems. Math. Biosci. 196 (2005) 82-98. \url{https://doi.org/10.1016/j.mbs.2005.03.008}
\bibitem{Liang_2007} X. Liang, X.Q. Zhao, Asymptotic speeds of spread and traveling waves for monotone semiflows with applications, Comm. Pure Appl. Math. 60 (2007) 1-40. \url{https://doi.org/10.1002/cpa.20154}
\bibitem{Lin_2012} G. Lin, W.T. Li,  Asymptotic spreading of competition diffusion systems: the role of interspecific competitions, European J. Appl. Math. 23 (2012) 669-689.  \url{https://doi.org/10.1017/S0956792512000198}
 \bibitem{LLL2019} Q. Liu, S. Liu, K.Y. Lam, Asymptotic spreading of interacting species with multiple fronts I: A geometric optics approach, \href{https://arxiv.org/pdf/1908.05025}{arXiv:1908.05025 [math.AP]}.
 \bibitem{Liu_2019} S.Y. Liu, H.M. Huang, M.X. Wang, Asymptotic spreading of a diffusive competition model with different free boundaries, J. Differential Equations 266 (2019) 4769-4799. \url{https://doi.org/10.1016/j.jde.2018.10.009}
\bibitem{Lui_1989}  R. Lui, Biological growth and spread modeled by systems of recursions, Math. Biosciences 93 (1989) 269-295 \url{https://doi.org/10.1016/0025-5564(89)90026-6}



 \bibitem{Mckean_1975} H.P. Mckean, Application of Brownian motion to the equation of Kolmogorov-Petrovskii-Piskunov, Commun. Pure Appl. Math. 28 (1975) 323-331.  \url{https://doi.org/10.1002/cpa.3160280302}
 \bibitem{Nolen_2008} J. Nolen, J. Xin,  Asymptotic spreading of KPP reactive fronts in incompressible space-time random flows, Ann. de l’Inst. Henri Poincare Analyse Non Lineaire 26 (2008) 815-839. \url{https://doi.org/10.1016/j.anihpc.2008.02.005}
\bibitem{Ostrov_2002} D.N. Ostrov, Solutions of Hamilton-Jacobi equations and scalar conservation laws with discontinuous space-time dependence, J. Differential Equations 182 (2002) 51-77. \url{https://doi.org/10.1006/jdeq.2001.4088}
\bibitem{Saarloos_2003} W.V Saarloos, Front propagation into unstable states, Phys. Rep. 386 (2003) 29-222. \url{https://doi.org/10.1016/j.physrep.2003.08.001}
 \bibitem{Shigesada_1997} N. Shigesada, K. Kawasaki, Biological invasions: Theory and Practice, Oxford University Press, Oxford, 1997.
\bibitem{Smith_1995} H.L. Smith, Monotone Dynamical Systems: An Introduction to the Theory of Competitive and Cooperative Systems, AMS, Providence, RI 1995.
\bibitem{Tang_1980} M.M. Tang, P.C. Fife, Propagating fronts for competing species equations with diffusion, Arch. Ration. Mech. Anal. 73 (1980) 69-77. \url{https://doi.org/10.1007/BF00283257}
\bibitem{Tourin_1992} A. Tourin, A comparison theorem for a piecewise
Lipschitz continuous Hamiltonian and application to Shape-from-Shading problems, Numer. Math. 62 (1992) 75-85. \url{https://doi.org/10.1007/BF01396221}
 \bibitem{Volpert_1994} A.I. Volpert, V.A. Volpert, V.A. Volpert, Traveling Wave Solutions of Parabolic Systems, Transl. Math. Monogr., Amer. Math. Soc. 1994.
 \bibitem{Wang_2017} M.X. Wang Y. Zhang, Note on a two-species competition-diffusion model with two free boundaries, Nonlinear Anal. 159 (2017) 458-467. \url{https://doi.org/10.1016/j.na.2017.01.005}
 \bibitem{Wang_2018} M.X. Wang, Y. Zhang, Dynamics for a diffusive prey-predator model with different free boundaries, J. Differential Equations 264 (2018) 3527-3558.  \url{https://doi.org/10.1016/j.jde.2017.11.027}
 \bibitem{Weinberger_1982} H.F. Weinberger, Long-time behavior of a class of biological models, SIAM J. Math. Anal. 13 (1982) 353-396 \url{https://doi.org/10.1137/0513028}

 
 
 
\bibitem{Weinberger_2002} H.F. Weinberger, On spreading speed and traveling waves for growth and migration models in a periodic habitat, J. Math. Biol. 45 (2002) 511-548. \url{https://doi.org/10.1007/s00285-002-0169-3}
\bibitem{Weinberger_2002b} H.F. Weinberger, M.A. Lewis and B.T. Li, Analysis of linear determinacy for spread in cooperative models, J. Math. Biol. 45 (2002) 183-218. \url{https://doi.org/10.1007/s002850200145}
\bibitem{Wu_2015} C.H. Wu, The minimal habitat size for spreading in a weak competition system with two free boundaries, J. Differential Equations 259 (2015) 873-897. \url{https://doi.org/10.1016/j.jde.2015.02.021}









\end{thebibliography}
\end{document}